\documentclass[11pt,a4paper]{amsart}

\usepackage{}
\usepackage[mathscr]{eucal}
\usepackage{amssymb}
\usepackage{amsbsy}

\usepackage{CJK,CJKnumb}
\usepackage{color,soul}              
\usepackage{indentfirst}        
\usepackage{latexsym,bm}        
\usepackage{graphics,graphicx}
\usepackage{cases}
\usepackage{pifont}
\usepackage{txfonts}
\usepackage{xcolor}
\usepackage[all,knot,poly]{xy}
\usepackage[usenames,dvipsnames]{pstricks}
\usepackage{epsfig}
\usepackage{pst-grad} 
\usepackage{pst-plot} 

\allowdisplaybreaks[4]  

\numberwithin{equation}{section}

\newcommand{\al}{\alpha}
\newcommand{\be}{\beta}
\newcommand{\de}{\delta}
\newcommand{\De}{\Delta}
\newcommand{\ep}{\epsilon}
\newcommand{\ga}{\gamma}
\newcommand{\Ga}{\Gamma}
\newcommand{\la}{\lambda}

\newcommand{\La}{\Lambda}
\newcommand{\pa}{/\hspace{-.2em}/}
\newcommand{\ot}{\otimes}

\newcommand{\si}{\sigmaup}
\newcommand{\te}{\theta}

\newcommand{\sh}{~\raisebox{-0.08em}{\rotatebox{90}{$\in$}}~}
\newcommand{\shq}{~\raisebox{-0.08em}{\rotatebox{90}{$\in$}$_q$}~}
\newcommand{\osh}{~\raisebox{-0.08em}{\rotatebox{90}{$\in$}$_{-1}$}~}
\newcommand{\mb}[1]{\mbox{\bfseries #1}}

\newcommand{\ms}[1]{\mbox{\sffamily #1}}

\newcommand{\vs}{\varsigmaup}
\newcommand{\s}[1]{{\scriptsize #1}}
\newcommand{\ti}[1]{{\tiny #1}}
\newcommand{\stt}[1]{{\scriptstyle #1}}
\newcommand{\sst}[1]{{\scriptscriptstyle #1}}
\newcommand{\ol}[1]{\overline{#1}}

\newcommand{\us}[1]{\mbox{\upshape #1}}
\newcommand{\lan}{\langle}
\newcommand{\ran}{\rangle}
\newcommand{\lb}{\left(}
\newcommand{\rb}{\right)}

\begin{document}

\newtheorem{theorem}{Theorem}[section]

\newtheorem{lemma}[theorem]{Lemma}

\newtheorem{corollary}[theorem]{Corollary}
\newtheorem{proposition}[theorem]{Proposition}

\theoremstyle{remark}
\newtheorem{remark}[theorem]{Remark}

\newtheorem{definition}[theorem]{Definition}

\newtheorem{example}[theorem]{Example}


\title[$q$-symmetric functions and $q$-quasisymmetric functions]{$q$-symmetric functions and $q$-quasisymmetric functions}
\author[Li]{Yunnan Li}
\address{Department of Mathematics, East China Normal University,
Minhang Campus, Dong Chuan Road 500, Shanghai 200241, PR China}
\email{yunnan814@163.com}

\begin{abstract}
In this paper, we construct the $q$-analogue of Poirier-Reutenauer algebras, related deeply with other $q$-combinatorial Hopf algebras. As an application, we use them to realize the odd Schur functions defined in \cite{EK}, and naturally obtain the odd Littlewood-Richardson rule concerned in \cite{Ell}. Moreover, we construct the refinement of the odd Schur functions, called odd quasisymmetric Schur functions, parallel to the consideration in \cite{HLMW1}. All the $q$-Hopf algebras we discuss here provide the corresponding $q$-dual graded graphs in the sense of \cite{BLL}.

\end{abstract}
\date{February 21th, 2012}
\maketitle

\section{Introduction}
The graded Hopf algebra of symmetric functions, $\ms{Sym}$, has many fascinating properties related with several aspects of mathematics, involving the representation theory of symmetric groups, the theory of combinatorial Hopf algebras, the geometry of the Grassmannian, etc. Meanwhile, $\ms{Sym}$ has two of the most important generalizations: the algebra of quasisymmetric functions, $\ms{QSym}$, as the nonsymmetric analogue and the algebra of noncommutative symmetric functions, $\ms{NSym}$ as the noncommutative analogue. $\ms{QSym}$ was first introduced by I. Gessel \cite{Ge} as a source of generating functions for P-partitions, while $\ms{NSym}$ was first studied extensively in \cite{GKL} and proved to be isomorphic to Solomon's descent algebra of symmetric groups. The duality between $\ms{QSym}$ and $\ms{NSym}$ has been concerned in \cite{MR}, where the authors also constructed a remarkable self-dual Hopf algebra structure on the free abelian group generated by permutations $\mathfrak{S}$, now named after them as the Malvenuto-Reutenauer algebra. The MR-algebra deeply relates to many algebras central to algebraic combinatorics. For instance, $\ms{NSym}$ serves as a subalgebra of the MR-algebra and thus $\ms{QSym}$ as a Hopf quotient.

On the other hand, the MR-algebra nicely possesses a Hopf quotient or dually a Hopf subalgebra, called the Poirier-Reutenauer algebra, with a canonical basis indexed by standard Young tableaux \cite{PR}. It's significant to see that when restricting the projection from the MR-algebra to $\ms{QSym}$ on the PR-algebra, it exactly has the image $\ms{Sym}$, with the canonical basis mapped to the Schur functions.

Bearing the interplay above in mind, we want to generalize the whole picture to the $q$-Hopf algebra version. Note that the $q$-Hopf algebras are just a special kind of braided Hopf algebras endowed with the twisted braiding. In \cite{TU}, Thibon and Ung first introduced the algebra of quantum quasisymmetric functions $\ms{QSym}_q$, which also realized as the graded characteristics of finitely generated 0-Hecke algebra $H_N(0)$-modules \cite{KT}. Recently, a new basis for $\ms{QSym}$, called quasisymmetric Schur functions, has been introduced by Haglund, Luoto, Mason van Willigenburg in \cite{HLMW1}. It serves as a refinement of the Schur functions and behaves particularly well when multiplying with the Schur functions \cite{HLMW}.
 As very little papers investigate $\ms{QSym}_q$, we do one step further here by mainly focusing on the odd case, where $\ms{QSym}_{-1}$ is a Hopf super algebra. Specializing the construction of the $q$-PR-algebra to $q=-1$, we realize the odd Schur functions defined in \cite{EK,EKL}.  Meanwhile, we find the odd version of quasisymmetric Schur functions as the refinement of the odd Schur functions, parallel to the original construction in \cite{BLW,HLMW1}.

 On the other hand, Bergeron, Lam and Li considered the following composition of two constructions in \cite{BLL}: filtered tower of algebras $\longrightarrow$ $q$-combinatorial Hopf algebras $\longrightarrow$ $q$-dual graded graphs.
As a result, we can adopt their methods to realize plenty of $q$-analogues of well-known dual graded graphs, while the representation aspect of them still need to be investigated.

The paper is organized as follows.
In Section 2 we first introduce $q$-symmetric functions, $\ms{QSym}_q$ and the dual $\ms{NSym}$, then relate them from a canonical bilinear form explicitly. In Section 3 we define the $q$-analogues of the MR-algebras. In Section 4 we construct the $q$-analogues of the PR-algebras, thus bring in the combinatorics of Young tableaux. Consequently, we can build the relation diagram of all the $q$-Hopf algebras above generalizing the classical case. In Section 5 we realize the odd Schur functions from the $q$-PR algebras with the odd Littlewood-Richardson rule derived naturally. In Section 6 we construct the odd quasisymmetric Schur functions and thus the odd version of the Pieri rules, the Littlewood-Richardson rules, etc.
In Section 7 we use the construction mentioned in \cite{BLL} to obtain the $q$-dual graded graphs associated with all the $q$-Hopf algebras we concern.

\section{Introduction to $q$-(quasi)symmetric functions}
\subsection{Compositions and tableaux}
Denote $\mathbb{Z}^+$ the set of positive integers,  $[n]=\{1,\dots,n\},~n\in\mathbb{N}$ and $[0]=\emptyset$ for convenience. Let $C(n)$ be the \textit{composition} set of $n\in\mathbb{Z}^+$, consisting of those ordered tuples of positive integers summed to be $n$. Take $C=\bigcup_{n\in\mathbb{Z}^+}C(n)$. Write $\al\vDash n$ when $\al\in C(n)$. For $\al=(\al_1,\dots,\al_r)\in C$, all the entries in $\al$ are \textit{parts} of it. Take the \textit{length} $\ell(\al)$ of $\al$ to be $r$ and the \textit{weight} $|\al|$ of $\al$ to be $\sum\limits_{i=1}^r\al_i$. Meanwhile, the \textit{descent set} $\us{des}(\al)$ of $\al$ is defined by
\[\us{des}(\al)=\{\al_1,\al_1+\al_2,\dots,\al_1+\cdots\al_{r-1}\}.\]
Conversely, for any $S\subseteq [n-1],~n\in\mathbb{Z}^+$, we can associate a composition $c(S)\vDash n$ with $S$, defined by $c(S)=(s_1,s_2-s_1,\cdots,s_{r-1}-s_{r-2},n-s_{r-1})$,
where $S=\{s_1,\dots,s_{r-1}\},~0<s_1<\cdots<s_{r-1}<n$. We describe such subset $S$ by $\{s_1<\cdots<s_{r-1}\}$ for simplicity.

Besides, let $C_w(n)$ be the \textit{weak composition} set of $n\in\mathbb{Z}^+$, consisting of all the ordered tuples of nonnegative integers with the sum equal to $n$. Take $C_w=\bigcup_{n\in\mathbb{Z}^+}C_w(n)$. For any $\ga\in C_w(n)$, its \textit{collapse} $\al(\ga)\vDash n$ is obtained by removing all the zero parts of $\ga$.

Now we can endow the composition set $C$ with several partial orders.
For $\al,\be\vDash n$, if we can obtain $\al$ by adding together adjacent parts of $\be$, then we say that $\be$ is a \textit{refinement} of $\al$, denoted $\al\preceq\be$ (or $\be\succeq\al$).  Note that one can also endow $C_w$ with such refinement order. Note that for any $S,S'\subseteq [n-1]$, $c(S)\preceq c(S')$ if and only if $S\subseteq S'$.

Given $\al=(\al_1,\dots,\al_r),~\be=(\be_1,\dots,\be_s)\in C$, we define two operations on them: the \textit{concatenation} $\al\be=(\al_1,\dots,\al_r,\be_1,\dots,\be_s)$ and the \textit{near concatenation} $\al\vee\be=(\al_1,\dots,\al_r+\be_1,\dots,\be_s)$. Meanwhile, we define the \textit{complement composition} $\al^c$ of $\al$ as
\[\al^c=(1^{\al_1})\vee(1^{\al_2})\vee\cdots\vee(1^{\al_r}),\]
 the \textit{reverse composition} of $\al$ as $\ol{\al}=(\al_r,\dots,\al_1)$ and the associated partition $\tilde{\al}$ of $\al$ obtained by rearranging the parts of $\al$ in weakly decreasing order. Note that $\us{des}(\al^c)=[n-1]\backslash\us{des}(\al),
 ~\us{des}(\ol{\al})=\{i\in[n-1]:n-i\in\us{des}(\al)\}$
  for any $\al\vDash n$.

For $w\in\mathfrak{S}_n$, we can also define the \textit{descent set} $\us{des}(w)$ of $w$ by
\[\us{des}(w)=\{i\in[n-1]:w(i)>w(i+1)\},\]
and thus the associated composition
$c(w)=c(\us{des}(w))\vDash n$.

We assume that the readers are familiar with the terminology of Young tableaux, for which we follow the standard textbook \cite{Ful}. We denote $\mathcal {P}$ the set of partitions and write $\la\vdash n$ when $\la\in\mathcal {P}\cap C(n)$. Let $\us{SSYT}=\bigcup_{n\geq0}\us{SSYT}_n$ be the set of \textit{semistandard tableaux} (or simply \textit{tableaux}), where $\us{SSYT}_0=\emptyset$ and $\us{SSYT}_n=\bigcup\limits_{\la\vdash n}\us{SSYT}(\la),~n\in\mathbb{Z}^+$ with $\us{SSYT}(\la)$ consisting of all \textit{semistandard tableaux} of \textit{shape} $\la$. Similarly, let $\us{SYT}=\bigcup_{n\geq0}\us{SYT}_n$ be the set of \textit{standard tableaux}. We also need the notation $\us{SYT}(\la/\mu)$ (resp. $\us{SSYT}(\la/\mu)$) for the set of (semi)standard tableaux of skew shape $\la/\mu$ and $\us{sh}(T)$ for the shape of $T\in\us{SSYT}$. Sometimes we use $\us{SSYT}[m]$ to emphasize the tableaux with entries in $[m]$ as well.

\subsection{The definition of $q$-(quasi)symmetric functions}
First we introduce the main study object, the \textit{$q$-(quasi)symmetric functions}, defined by Thibon and Ung in \cite{TU}.  Though there are plenty of papers on noncommutative and quasisymmetric functions, very little treat the quantum case. Here are other papers involving some material for such direction, \cite{DHT,EK,KT}.

We work over a commutative ring $k$ with unity and fix $q\in k$ invertible if without further emphasis, thus $\mathbb{Z}[q,q^{-1}]$ is a nice candidate. Consider the quantum affine $n$-space \[A_q^{n|0}=k\lan x_1,\dots,x_n\ran/(x_ix_j-qx_jx_i,i>j).\]
When $n=\infty$, we take
\[A_q^{\infty|0}=k\lan\lan x_1,x_2,\dots\ran\ran/(x_ix_j-qx_jx_i,i>j),\]
abbreviated as $A_q$.
Define the \textit{$q$-monomial quasisymmetric function}
\[M_\al:=M_\al(x)=\sum\limits_{\ga\in C_w\atop\al(\ga)=\al}x^\ga=\sum\limits_{i_1<\cdots< i_r}x_{i_1}^{\al_1}\cdots x_{i_r}^{\al_r},\]
where $\al=(\al_1,\cdots,\al_r)$ varies over the composition set $C$ and $x^\ga:=\prod\limits_{i\geq1}^{\longrightarrow} x_i^{\ga_i}$. Note that $\{M_\al\}_{\al\in C}$ spans a subalgebra of $A_q$ as a $k$-basis, called the algebra of \textit{$q$-quasisymmetric functions} and denoted $\ms{QSym}_q$. Define the $q$-\textit{quasishuffle product} $\bowtie_q$ on the  $\mathbb{Z}[q]$-module freely generated by $C$ recursively as follows,
\[\begin{array}{l}
\al\bowtie_q\be=(\al_1)(\al'\bowtie_q\be)
+q^{\be_1|\al|}(\be_1)(\al\bowtie_q\be')+q^{\al_1\be_1}(\al_1+\be_1)(\al'\bowtie_q\be'),\\
\al\bowtie_q\emptyset=\emptyset\bowtie_q\al=\al,
\end{array}\]
where $\al=(\al_1)\al',~\be=(\be_1)\be'\in C$. Then the multiplication rule of $M_\al$ can be described by $\bowtie_q$ as follows,
\begin{equation}\label{mom}
M_\al M_\be=\sum_\ga f_\ga(q)M_\ga,
\end{equation}
where $\al\bowtie_q\be=\sum_\ga f_\ga(q)\ga$. For the further detail of this interesting product, one can refer to \cite{Hof,JRZ}. Meanwhile, we need another important basis of $\ms{QSym}_q$, the \textit{$q$-fundamental quasisymmetric functions}, as follows,
\[F_\al:=F_\al(x)=\sum\limits_{i_1\leq\cdots\leq i_n\atop i_k<i_{k+1}\us{ \ti{if} }k\in\us{\ti{des}}(\al)}x_{i_1}\cdots x_{i_n},\]
where $\al=(\al_1,\cdots,\al_r)$ varies over the composition set $C(n)$ for any $n\geq0$. Note that $F_\al=\sum\limits_{\be\succeq \al}M_\be$, and thus  $M_\al=\sum\limits_{\be\succeq \al}(-1)^{\ell(\be)-\ell(\al)}F_\be$ by the M$\ddot{\mbox{o}}$bius inversion.

In order to define all kinds of $q$-analogue of symmetric functions as the classical case, one can use the following generating function in $A_q[[t]]$,
\begin{equation}\label{com}
h(x,t)=\sum_{n\geq0}h_n(x)t^n=\prod\limits_{i\geq1}^{\longrightarrow}\dfrac{1}{1-x_it}.
\end{equation}
It gives
\[h_n:=h_n(x)=\sum\limits_{i_1\leq\cdots\leq i_n}x_{i_1}\cdots x_{i_n},~n\in\mathbb{Z}^+,\]
which are the \textit{$q$-analogues of complete symmetric functions}.

The generating function (\ref{com}) has the inverse
\begin{equation}\label{ele}
\begin{split}
\prod\limits_{i\geq1}^{\longleftarrow}(1-x_it)&=\sum\limits_{n\geq0}\lb\sum\limits_{i_1>\cdots> i_n}(-1)^nx_{i_1}\cdots x_{i_n}\rb t^n\\
&=\sum\limits_{n\geq0}(-1)^nq^{\tfrac{n(n-1)}{2}}\lb\sum\limits_{i_1<\cdots<i_n}x_{i_1}\cdots x_{i_n}\rb t^n.
\end{split}
\end{equation}
It gives
\[e_n:=e_n(x)=\sum\limits_{i_1<\cdots<i_n}x_{i_1}\cdots x_{i_n},~n\in\mathbb{Z}^+,\]
which are the \textit{$q$-analogues of elementary symmetric functions}, and the relation
\begin{equation}\label{ce}
h_0=e_0=1,~\sum\limits_{k=0}^n(-1)^kq^{\tfrac{k(k-1)}{2}}e_kh_{n-k}=0,n\in\mathbb{Z}^+.
\end{equation}
Now define the algebra of \textit{$q$-symmetric functions} as the subalgebra
of $\ms{QSym}_q$ generated by $\{h_n\}_{n\in\mathbb{Z}^+}$, and denote it as $\ms{Sym}_q$.

On the other hand, given the alphabet $A=\{a_1,a_2,\dots\}$, one has the free associative algebra $\mathcal {F}=k\lan\lan a_1,a_2,\dots\ran\ran$. Similarly define the following generating functions in $\mathcal {F}[[t]]$,
\begin{equation}\label{ncom}
H(A,t)=\sum_{n\geq0}H_n(A)t^n=\prod\limits_{i\geq1}^{\longrightarrow}\dfrac{1}{1-a_it}.
\end{equation}
and
\begin{equation}\label{nele}
E(A,t)=\sum_{n\geq0}(-1)^nE_n(A)t^n=\prod\limits_{i\geq1}^{\longleftarrow}(1-a_it).
\end{equation}
Then $\{H_n(A)\}_{n\in\mathbb{Z}^+}$ (or $\{E_n(A)\}_{n\in\mathbb{Z}^+}$) generates a subalgebra of $\mathcal {F}$, called the algebra of \textit{noncommutative symmetric functions} and denoted as $\ms{NSym}$. Similarly, we have
\[H_0=E_0=1,~\sum\limits_{k=0}^n(-1)^kE_kH_{n-k}=0,~n\in\mathbb{Z}^+.\]
 Let $H_\al=H_{\al_1}\dots H_{\al_r},E_\al=E_{\al_1}\dots E_{\al_r},~\al=(\al_1,\cdots,\al_r)\vDash n$.
 Note that $\{H_\al\}_{\al\in C}$ (or $\{E_\al\}_{\al\in C}$) forms a basis of $\ms{NSym}$, hence there exists a canonical bilinear form
  \begin{equation}\label{bi}
  \lan\cdot,\cdot\ran:\ms{QSym}_q\times\ms{NSym}\rightarrow k,~\lan M_\al,H_\be\ran=\de_{\al,\be},~\al,\be\in C.
  \end{equation}
It can be canonically extended as a bilinear form on $\ms{QSym}_q^{\ot n}\times\ms{NSym}^{\ot n}$ for all $n\in\mathbb{Z}^+$. Meanwhile, we have another basis of $\ms{NSym}$, the \textit{noncommutative ribbon Schur functions},
\[R_\al:=\sum_{\be\preceq \al}(-1)^{\ell(\al)-\ell(\be)}H_\be,\]
such that $\lan F_\al,R_\be\ran=\de_{\al,\be},~\al,\be\in C$.

\subsection{$q$-Hopf algebra duality between $\ms{QSym}_q$ and $\ms{NSym}$}
The \textit{graded $q$-Hopf algebras} are a kind of $\mathbb{Z}$-graded colored Hopf algebras defined in \cite[\S10.5]{Mon}. Now endow $\ms{QSym}_q$ with the graded $q$-Hopf algebra structure as follows,
\[\ms{QSym}_q=\bigoplus_{n\in\mathbb{N}}(\ms{QSym}_q)_n,\]
where $(\ms{QSym}_q)_n=\mbox{span}_k\{M_\al|\al\vDash n\}$ is the $n$th graded component, with
the coproduct $\De'$ defined by
\begin{equation}\label{mo}
\De'(M_\al)=\sum\limits_{\be\ga=\al}M_\be\ot M_\ga.
\end{equation}
Consequently, one can see that
\begin{equation}\label{fun}
\De'(F_\al)=\sum\limits_{\be\ga=\al}F_\be\ot F_\ga+\sum\limits_{\be\vee\ga=\al}F_\be\ot F_\ga.
\end{equation}
Note that for $F,G\in\ms{QSym}_q$, we have
\[\De'(FG)=\sum q^{|F_{(2)}||G_{(1)}|}F_{(1)}G_{(1)}\ot F_{(2)}G_{(2)},\]
where $\De'(F)=\sum F_{(1)}\ot F_{(2)},\De'(G)=\sum G_{(1)}\ot G_{(2)}$, and $|\cdot|$ is the evaluation of grading on homogeneous elements.

Using the canonical bilinear form (\ref{bi}), one can endow $\ms{NSym}$ with the graded $q$-Hopf algebra structure which is graded dual to $\ms{QSym}_q$. Namely, $\ms{NSym}=\bigoplus_{n\in\mathbb{N}}\ms{NSym}_n$, where $\ms{NSym}_n=\mbox{span}_k\{H_\al|\al\vDash n\}$ is the $n$th graded component. As $\ms{NSym}$ is freely generated by $\{H_n(A)\}_{n\in\mathbb{Z}^+}$, define the coproduct $\De_q$ on $\ms{NSym}$ by
\begin{equation}\label{coq}
\De_q(H_n)=\sum\limits_{k=0}^nH_k\ot H_{n-k},
\end{equation}
and the rule
\begin{equation}\label{rule}
\De_q(HH')=\sum q^{|H_{(2)}||H'_{(1)}|}H_{(1)}H'_{(1)}\ot H_{(2)}H'_{(2)},
\end{equation}
where $H,H'\in\ms{NSym},~\De_q(H)=\sum H_{(1)}\ot H_{(2)},\De_q(H')=\sum H'_{(1)}\ot H'_{(2)}$, and $|\cdot|$ is again the evaluation of grading on homogeneous elements. In order to simplify some coefficients, we introduce the following bicharacter $\te$ on $\mathbb{Z}^n$
\[\te:\mathbb{Z}^n\times\mathbb{Z}^n\rightarrow k^\times,\te(
\al,\be)=q^{\sum_{i>j}\al_i\be_j},~\al=(\al_1,\dots,\al_n),~\be=(\be_1,\dots,\be_n)\in\mathbb{Z}^n.\] In particular, we have
\[\De_q(H_\ga)=\sum\te(\ga',\ga'')H_{\ga'}\ot H_{\ga''},\]
where the sum is over all ordered pair $(\ga',\ga'')\in \mathbb{N}^{\ell(\ga)}\times\mathbb{N}^{\ell(\ga)}$, such that $\ga'+\ga''=\ga$, and $H_{\ga'},H_{\ga''}$ are interpreted naturally as before. Now one can see that
\begin{proposition}\label{qdual}
$(\ms{QSym}_q,\cdot,\De')$ and $(\ms{NSym},\cdot,\De_q)$ are graded dual to each other as $q$-Hopf algebras with respect to the canonical bilinear form (\ref{bi}).
\end{proposition}

From the formula (\ref{mo}) of $\De'$ on $M_\al$'s and the obvious identity $h_n=\sum\limits_{\al\vDash n}M_\al$, we have
\begin{equation}\label{co}
\De'(h_n)=\sum\limits_{k=0}^n h_k\ot h_{n-k}.
\end{equation}
It means the coproduct formulas (\ref{coq}) and (\ref{co}) are compatible. On the other hand, $\{H_n\}_{n\in\mathbb{Z}^+}$ freely generates $\ms{NSym}$, which admits the following homomorphism of graded $q$-Hopf algebras
\[\phi:\ms{NSym}\rightarrow \ms{QSym}_q,~\phi(H_n)=h_n,~n\in\mathbb{Z}^+,\]
referred as the \textit{forgetful map}. Note that $\phi(E_n)=q^{n(n-1)/2}e_n$. For generic $q$, Thibon and Ung in \cite{TU} pointed out that $\phi$ is an isomorphism. Indeed, if denote $e^{\al}=e_{\al_1}\cdots e_{\al_r}$ for $\al=(\al_1,\dots,\al_r)\in C$, then the map $F_\al\mapsto e^{\al^c}$ is invertible. We also denote $r_\al=\phi(R_\al)$ as the $q$-\textit{ribbon Schur function}.

Now define a bilinear form $(\cdot,\cdot)$ on $\ms{NSym}$ by
\begin{equation}\label{bif}
  (H_\al,H_\be)=\lan\phi(H_\al),H_\be\ran=\lan h_\al,H_\be\ran,~\al,\be\in C.
\end{equation}

Note that the bilinear form $(\cdot,\cdot)$ first appears in \cite[3]{TU}, and serves as the basic tool directly in \cite{EK}. The whole setting is quite similar to Lusztig's construction of the algebra $\mb{f}$ \cite[Chapter 1]{Lus}. Let's explain the nice properties of such bilinear form in details as follows.

Using the duality between $\ms{QSym}_q$ and $\ms{NSym}$, we can compute that
\begin{equation}\label{va}
\begin{split}
(H_\al,H_\be)&=\lan h_\al,H_\be\ran=\lan h_{\al_1}\ot\cdots\ot h_{\al_k},\De_q^{(k-1)}(H_\be)\ran\\
&=\sum\limits_{\be^{(1)},\dots,\be^{(k)}\in \mathbb{N}^l \atop \be^{(1)}+\cdots+\be^{(k)}=\be,~|\be^{(i)}|=\al_i}\prod\limits_{1\leq i<j\leq k}\te(\be^{(i)},\be^{(j)})
=\begin{cases}\sum\limits_{w\in\mathfrak{S}_{\al,\be}}q^{\ell(w)},\mbox{ if }|\al|=|\be|,\\0,\mbox{ otherwise}.\end{cases}
\end{split}
\end{equation}
for any $\al,\be\in C$ and $k=\ell(\al),~l=\ell(\be)$, where $\mathfrak{S}_{\al,\be}$ is the set of minimal representatives of the double coset of the Young subgroups $\mathfrak{S}_{\al},\mathfrak{S}_{\be}$ for $\mathfrak{S}_n$, $n=|\al|=|\be|$.

Note that $\mathfrak{S}_{\al,\be}=\{w\in\mathfrak{S}_n:c(w^{-1})\preceq\al,~c(w)\preceq\be\}$, thus by the formula (\ref{va}) one can easily deduce that
\begin{equation}\label{rib}
(R_\al,R_\be)=\sum\limits_{c(w^{-1})=\al\atop c(w)=\be}q^{\ell(w)}.
\end{equation}

For any $w\in\mathfrak{S}_{\al,\be}$, we can present it by the following diagram
\[\xy 0;/r.2pc/:
(0,0)*{\scalebox{0.4}{\includegraphics{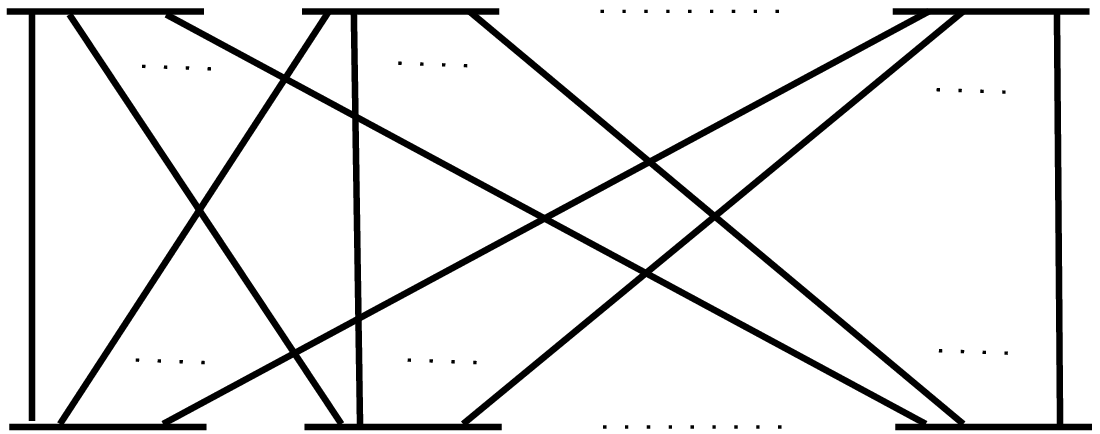}}};
(-21,12)*{\sst{\al_1}};(-7,12)*{\sst{\al_2}};(21,12)*{\sst{\al_k}};
(-21,-13)*{\sst{\be_1}};(-7,-13)*{\sst{\be_2}};(21,-13)*{\sst{\be_l}};
(-27.5,2)*{\sst{\be^{(1)}_1}};(-21.5,2)*{\sst{\be^{(1)}_2}};(6.5,-7)*{\sst{\be^{(1)}_l}};
\endxy\]
and the length $\ell(w)$ of $w$ equals the number of crossings in its corresponding diagram. Since we can put the diagram upside down to get $w^{-1}\in\mathfrak{S}_{\be,\al}$, whose length remains the same, we know that the form $(\cdot,\cdot)$ is symmetric by the formula (\ref{va}). Now we list some useful properties of the bilinear form $(\cdot,\cdot)$.
\begin{proposition}\label{bp}
(1) $(\cdot,\cdot)$ is symmetric and the radical is $\us{Ker}\phi=\bigoplus_{n\in\mathbb{Z}^+}\us{Ker}\phi_n$.

(2) The graded components of $\ms{NSym}$ are orthogonal with respect to $(\cdot,\cdot)$.

(3) $\us{Im}\phi=\ms{Sym}_q$ is the orthogonal complement of $\us{Ker}\phi$ with respect to the canonical pairing $\lan\cdot,\cdot\ran$, thus $(\cdot,\cdot)$ induces a nondegenerate bilinear form on $\ms{Sym}_q$.

(4) Since $\mathbb{Z}[q,q^{-1}]$ is a PID, $\ms{Sym}_q$ is still a free $\mathbb{Z}[q,q^{-1}]$-module, expressed as the quotient of the free $\mathbb{Z}[q,q^{-1}]$-module $\ms{NSym}$ by the radical of $(\cdot,\cdot)$.
\end{proposition}

From (3), we know that for any $F\in\ms{QSym}_q$, $F$ lies in $\ms{Sym}_q$ if and only if $F\in\us{Ker}\phi^\perp$. However, the radical of $(\cdot,\cdot)$ on $\ms{NSym}$ is quite ruleless when $q$ is an algebraic integer. The results known so far are only when $q=\pm1$. If $q=1$, then $\us{Ker}\phi$ is the ideal generated by $H_nH_m-H_mH_n,~n,m\in\mathbb{N}$, which back to the usual symmetric functions. While in the odd case, we know that
\begin{proposition}[{\cite[Prop. 2.9.]{EK}}]
$\us{Ker}\phi$ is the ideal generated by
\begin{equation}
\label{odd}\begin{array}{l}
 H_nH_m-H_mH_n,\mbox{ if }n+m\mbox{ even,}\\
 H_nH_m+(-1)^nH_mH_n-(-1)^nH_{n+1}H_{m-1}-H_{m-1}H_{n+1},\mbox{ if }n+m\mbox{ odd.}
\end{array}
\end{equation}
\end{proposition}

Now we introduce three commuting linear automorphisms of $\ms{NSym}$, the composition of which gives the antipode of $\ms{NSym}$.
\[\begin{array}{l}
\psi_1:\mbox{a }q\mbox{-Hopf algebra automorphism of }\ms{NSym}\mbox{ defined by }\psi_1(H_n)=q^{-n(n-1)/2}E_n,~n\in\mathbb{N},\\
\psi_2:\mbox{a linear automorphism of }\ms{NSym}\mbox{ defined by }\psi_2(H_\al)=(-1)^{|\al|}q^{|\al|\choose2}H_\al,~\al\in C,\\
\psi_3:\mbox{a bialgebra anti-automorphism of }\ms{NSym}\mbox{ defined by }\psi_3(H_n)=H_n,~n\in\mathbb{N}.
\end{array}
\]
\begin{proposition}\label{auto}
All $\psi_i,~i=1,2,3$, can induce their corresponding (anti-)automorphisms of $\ms{Sym}_q$. Namely, there exist three (anti-)automorphisms of $\ms{Sym}_q$, also denoted $\psi_1,\psi_2$ and $\psi_3$ respectively, such that $\psi_i\circ\phi=\phi\circ\psi_i,~i=1,2,3$.
\end{proposition}
\begin{proof}
In order to get such $\psi_i$'s on $\ms{Sym}_q$, we only need to check
\[(\psi_i(H),H')=(H,\psi_i(H')),~\forall H,H'\in\ms{NSym}.\]
Since then we have $\psi_i(\us{Ker}\phi)\subseteq\us{Ker}\phi$, which means that $\phi\circ\psi_i$ can be factored through by $\phi$.
Now we check them using the basis $\{H_\al\}_{\al\in C}$. For any $\al,\be\in C$, we have
\[\begin{array}{l}
(\psi_1(H_\al),H_\be)=(e_\al,h_\be)=(h_\al,e_\be)=(H_\al,\psi_1(H_\be)),\\
(\psi_2(H_\al),H_\be)=(-1)^{|\al|}q^{|\al| \choose 2}(H_\al,H_\be)=(-1)^{|\be|}q^{|\be| \choose 2}(H_\al,H_\be)=(H_\al,\psi_2(H_\be)),\\
(\psi_3(H_\al),H_\be)=(H_{\ol{\al}},H_\be)
=(H_\al,H_{\ol{\be}})=(H_\al,\psi_3(H_\be)),
\end{array}\]
where the equality for $\psi_3$ can be seen by rotating all diagrams via the vertical axis which keeps all the values unchanged.
\end{proof}

\section{Malvenuto-Reutenauer algebra and its $q$-analogue}
In the paper \cite{MR}, Malvenuto and Reutenauer defined a self-dual graded Hopf algebra structure on the free abelian group $\mathbb{Z}\mathfrak{S}=\bigoplus_{n\geq0}\mathbb{Z}\mathfrak{S}_n$ of permutations, denoted the \textit{MR-algebra}.  The relation between the MR-algebra and the algebra of (quasi)symmetric functions inspires us to construct the $q$-analogue of the MR-algebra for the investigation of the $q$-(quasi)symmetric functions.

We first introduce the MR-algebra, which is a dual pair of graded Hopf algebras $(\mathbb{Z}\mathfrak{S},*,\De)$ and $(\mathbb{Z}\mathfrak{S},*',\De')$, denoted $\ms{MR}$ and $\ms{MR}'$ respectively. One can check the details in \cite{AgS,MR}.
For any word $w$ of length $n$ on a totally ordered alphabet $A$, denote $\us{alph}(w)\subset A$ the set of letters in $w$ and $\us{st}(w)\in\mathfrak{S}_n$, called the \textit{standardization} of $w$, the unique permutation such that
\[\us{st}(w)(i)<\us{st}(w)(j)\mbox{ iff }w_i<w_j\mbox{ or }w_i=w_j,i<j,\]
where $w=w_1\cdots w_n$. Note that $\ell(\us{st}(w))$ is just the inversion number of $w$. For any $w\in\mathfrak{S}_n$, one can view it as a word on $[n]$. For $I\subseteq[n]$, let $w|_I$ denote the subword of $w$ keeping only the digits in $I$. We also need the shuffle product \sh and the concatenation coproduct $\de'$.
For any $w\in\mathfrak{S}_p,~w'\in\mathfrak{S}_q$,
\[
w*w'=\sum\limits_{\us{\ti{alph}}(u)\cup\us{\ti{alph}}(v)=[p+q]
\atop\us{\ti{st}}(u)=w,~\us{\ti{st}}(v)=w'}uv,~
\De(w)=\sum\limits_{i=0}^nw|_{\{1,\dots,i\}}\ot\us{st}(w|_{\{i+1,\dots,n\}}).
\]
On the other hand,
\[
w*'w'=w\sh\ol{w'},~
\De'(w)=(\us{st}\ot\us{st})\de(w),
\]
where $\ol{w'}$ means shifting the digits in $w'$ by $p$.

Under the canonical pairing on $\mathbb{Z}\mathfrak{S}$ defined by $\lan w,w'\ran=\de_{w,w'}$, $\ms{MR}$ and $\ms{MR}'$ are dual to each other as graded Hopf algebras.
Meanwhile, there exists an involution $\te$ of graded Hopf algebras between them defined by $\te(w)=w^{-1}$ and the embedding $\iota:\ms{NSym}\rightarrow \ms{MR}$ defined by $\iota(H_\al)=D_{\preceq\al}$ (resp. $\iota(R_\al)=D_\al$), where $D_{\preceq\al}=\sum_{c(w)\preceq\al}w$ (resp. $D_\al=\sum_{c(w)=\al}w$). Such embedding identifies $\ms{NSym}$ with the \textit{Solomon's descent subalgebra} of $\mathbb{Z}\mathfrak{S}$ with a basis $\{D_{\preceq\al}:\al\in C\}$.
Dually there exists a surjection $\pi':\ms{MR}'\rightarrow \ms{QSym}$ defined by $\pi'(w)=F_{c(w)}$.

Next we construct the $q$-analogues $\ms{MR}_q$ and $\ms{MR}'_q$ over $k$. For $\ms{MR}'_q$ we modify the product $*'$ of $\ms{MR}'$ to $*'_q$ defined by
\begin{equation}\label{ashq'}
w*'_qw'=w\shq\ol{w'},
\end{equation}
where \shq is the $q$-shuffle on words defined in \cite{Hof}. That is, $\ms{MR}'_q=(k\mathfrak{S},*'_q,\De')$.  Dually for $\ms{MR}_q$ we modify the coproduct $\De$ of $\ms{MR}$ to $\De_q$ defined by
\begin{equation}\label{deq}
\De_q(w)=\sum\limits_{i=0}^nq^{\us{\s{inv}}(i,w)}w|_{\{1,\dots,i\}}\ot\us{st}(w|_{\{i+1,\dots,n\}}),
\end{equation}
where $\us{inv}(i,w)$ is the number of pairs $(k,l),~k>l$ such that $w_k\in\{1,\dots,i\},w_l\in\{i+1,\dots,n\}$ as $w=w_1\cdots w_n$. That is, $\ms{MR}_q=(k\mathfrak{S},*,\De_q)$. We remind that the multiplication of tensor products now involves the $q$-flip $\si$ defined by $\si(w\ot w')=q^{nm}w'\ot w,~w\in\mathfrak{S}_n,w'\in\mathfrak{S}_m$.  Now we are in the position to check the following result
\begin{theorem}\label{mr}
(1) $\ms{MR}_q$ and $\ms{MR}'_q$ are graded dual to each other as $q$-Hopf algebras with respect to the canonical pairing $\lan,\ran:\ms{MR}'_q\times\ms{MR}_q\rightarrow k$.

(2) There exists an isomorphism $\te_q:\ms{MR}_q\rightarrow\ms{MR}'_q$, defined by $\te_q(w)=q^{\ell(w)}w^{-1}$, of graded $q$-Hopf algebras.

(3) There exists a surjection  $\pi'_q:\ms{MR}'_q\rightarrow\ms{QSym}_q,~\pi'_q(w)=F_{c(w)},~w\in\mathfrak{S}$ and an embedding $\iota_q:\ms{NSym}\rightarrow\ms{MR}_q,~\iota_q(H_\al)=D_{\preceq\al},~\al\in C$ of graded $q$-Hopf algebras such that $\lan \pi'_q(w),H\ran=\lan w,\iota_q(H)\ran,~w\in\mathfrak{S},H\in\ms{NSym}$.

(4) The composition $\pi'_q\circ\te_q\circ\iota_q:\ms{NSym}\rightarrow\ms{QSym}_q$ is just the forgetful map $\phi$.
\end{theorem}
\noindent
\begin{proof} (1) The duality are straightforward to check. We only need to show that $\ms{MR}'_q$ is a $q$-Hopf algebra. For $w\in\mathfrak{S}_n,w'\in\mathfrak{S}_m$, we have
\[\begin{split}
\De'(w*'_qw')&=(\us{st}\ot\us{st})\de'(w*'_qw')=(\us{st}\ot\us{st})\de'(w{\shq}\ol{w'})
=(\us{st}\ot\us{st})(\de'(w){\shq}\de'(\ol{w'}))\\&=(\us{st}\ot\us{st})\sum\limits_{uv=w\atop xy=\ol{w'}}(\shq\ot\shq)(\us{id}\ot\si\ot\us{id})(u\ot v\ot x\ot y)\\
&=(\us{st}\ot\us{st})\sum\limits_{uv=w\atop xy=\ol{w'}}q^{|v||x|}(u\shq x)\ot(v\shq y)\\&=\sum\limits_{uv=w\atop xy=w'}q^{|v||x|}(\us{st}(u)*'_q\us{st}(x))\ot(\us{st}(v)*'_q\us{st}(y))=\De'(w)*'_q\De'(w'),
\end{split}\]
where we use the graded $q$-Hopf algebra structure $(k\lan A\ran,\shq,\de')$ and the evaluation $|\cdot|$ of grading is just the length function of words.

(2) As we have known that $\te(w*w')=\te(w)*'\te(w')$, thus
\[\te_q(w*w')=\sum_{u\in w*w'}q^{\ell(u)}u^{-1}=\sum_{u\in w^{-1}*'w'^{-1}}q^{\ell(u)}u=\te_q(w)*'_q\te_q(w'),\]
where we use the identity $(q^{\ell(w)}w)*'_q(q^{\ell(w')}w')=\sum\limits_{u\in w*'w'}q^{\ell(u)}u$.
Meanwhile,
\[\begin{split}(\te_q\ot\te_q)\De_q(w)&=\sum\limits_{i=0}^nq^{\us{\s{inv}}(i,w)+\ell(w|_{\{1,\dots,i\}})
+\ell(\us{\s{st}}(w|_{\{i+1,\dots,n\}}))}
(w|_{\{1,\dots,i\}})^{-1}\ot(\us{st}(w|_{\{i+1,\dots,n\}}))^{-1}\\
&=q^{\ell(w)}\sum\limits_{i=0}^n\us{st}(w^{-1}(1)\cdots w^{-1}(i))\ot\us{st}(w^{-1}(i+1)\cdots w^{-1}(n))=\De'\circ\te_q(w),
\end{split}\]
where we use the identity $\us{inv}(i,w)+\ell(w|_{\{1,\dots,i\}})
+\ell(\us{st}(w|_{\{i+1,\dots,n\}}))=\ell(w)$. Hence, $\te_q$ is a homomorphism of graded $q$-Hopf algebras.

(3) Note that $\iota$ and $\iota_q$ (resp. $\pi'$ and $\pi'_q$) are exactly the same as maps of $k$-vector spaces. What we need to show is that they are homomorphisms of  graded $q$-Hopf algebras. Let's check $\pi'_q$ as follows,
\[\pi'_q(w*'_qw')=\sum\limits_{u\in \mathfrak{S}}\lan w*'_qw',u\ran \pi'_q(u)=\sum\limits_{u\in \mathfrak{S}}\lan w*'_qw',u\ran F_{c(u)}=\pi'_q(w)\pi'_q(w'),\]
where we use the notation $w*'_qw'=\sum\limits_{u\in\mathfrak{S}}\lan w*'_qw',u\ran u$ and the multiplication formula
\begin{equation}\label{funm}
F_{c(w)}F_{c(w')}=\sum\limits_{u\in\mathfrak{S}}\lan w*'_qw',u\ran F_{c(u)}
\end{equation}
in $\ms{QSym}_q$ proved in \cite[4]{TU} using the language of $P$-partitions.
On the other hand, using the formula (\ref{fun}) we have
\[\begin{split}
(\pi'_q\ot\pi'_q)\De'(w)&=\sum\limits_{i=0}^n\pi'_q(\us{st}(w_1\cdots w_i))\ot\pi'_q(\us{st}(w_{i+1}\cdots w_n))\\
&=\sum\limits_{i=0}^nF_{c(\us{\s{st}}(w_1\cdots w_i))}\ot F_{c(\us{\s{st}}(w_{i+1}\cdots w_n))}=\De'(F_{c(w)})=\De'\pi'_q(w),
\end{split}\]
where $w=w_1\cdots w_n\in \mathfrak{S}_n$ and we use the fact that $c(\us{st}(w_1\cdots w_i)),c(\us{st}(w_{i+1}\cdots w_n)),~i=0,1,\dots,n$ are those $\al,\be \in C$ such that $\al\be=c(w)$ or $\al\vee\be=c(w)$.
By the adjunction between $\pi'_q$ and $\iota_q$, it implies $\iota_q$ is also a homomorphism of graded $q$-Hopf algebras.

(4) Using the basis $\{R_\al\}_{\al\in C}$ of $\ms{NSym}$ , we obtain
\[\begin{split}\pi'_q&\circ\te_q\circ\iota_q(R_\al)=\pi'_q\circ\te_q(D_\al)=\pi'_q(\sum\limits_{c(w)
=\al}q^{\ell(w)}w^{-1})=\sum\limits_{c(w)=\al}q^{\ell(w)}F_{c(w^{-1})}\\&=\sum\limits_{\be\in C}\lb\sum\limits_{c(w)=\al \atop c(w^{-1})=\be}q^{\ell(w)}\rb F_\be=
\sum\limits_{\be\in C}(R_\al,R_\be ) F_\be=\sum\limits_{\be\in C}\lan\phi(R_\al),R_\be\ran F_\be=\phi(R_\al),
\end{split}\]
where we apply the formula (\ref{rib}) and the duality between $R_\al$ and $F_\be$.
\end{proof}

\section{Poirier-Reutenauer algebra and its $q$-analogue}
In the paper \cite{PR}, Poirier and Reutenauer defined a graded Hopf quotient $\ms{PR}$ of $\ms{MR}$ and a graded Hopf subalgebra $\ms{PR}'$ of $\ms{MR}'$ dual to each other with respect to the canonical pairing, denoted the \textit{PR-algebra}. Later in the paper \cite{DHT}, Duchamp, Hivert and Thibon realize the MR-algebra as the algebra of \textit{free quasisymmetric functions}, denoted $\ms{FQSym}$, and identify $\ms{PR}'$ as the algebra $\ms{FSym}$ of \textit{free symmetric functions}, restricting the isomorphism between $\ms{MR}'$ and $\ms{FQSym}$. In order to introduce the PR-algebra, we first recall the Knuth equivalence as follows.

There exists a remarkable correspondence between two-rowed arrays in lexicographic order and pairs of semistandard tableaux with the same shape, called the \textit{Robinson-Schensted-Knuth correspondence}. For any such two-rowed array $w$, we denote $(P(w),Q(w))$ the image under the RSK-correspondence. $P(w)$ comes from the Schensted \textit{row-insertion} of $w$, and $Q(w)$ is the recording tableau. For further details, one can refer to \cite[Chapter 4]{Ful}.

For words on $\mathbb{Z}^+$, Knuth defined an equivalence relation on them, referred to as the \textit{Knuth equivalence} $\sim_K$. It's generated by the following elementary ones,
\[\begin{array}{lll}
  (K')\quad&yxz\sim yzx,& \mbox{ if } x<y\leq z,\\
  (K'')\quad&xzy\sim zxy,& \mbox{ if } x\leq y<z.
\end{array}\]
For any two words $w,w'$, note that $P(w)=P(w')$ if and only if they are Knuth equivalent. Hence, one can identify the Knuth equivalence classes with semistandard tableaux. When $w$ is a word on $\mathbb{Z}^+$, we write $w=w_1\cdots w_n$ instead of $\begin{pmatrix}
1&\cdots&n\\w_1&\cdots&w_n
\end{pmatrix}$ throughout the paper. For any $T\in\us{SYT}_n$, one can define the \textit{descent} set of $T$ as \[\us{des}(T)=\{k\in[n-1]:
k+1\mbox{ appears weakly left of }k\mbox{ in }T\},\]
together with the composition $c(T)$ associated with $\us{des}(T)$. From the RSK-algorithm, it's easy to see that $\us{des}(w)=\us{des}(Q(w))$ for any $w\in\mathfrak{S}$ \cite[\S4, Exercise 17]{Ful}.

\subsection{$q$-analogue construction of the PR-algebras}
First we recall the definition of the PR-algebras. Define $cl(T):=\sum\limits_{P(w)=T}w\in\mathbb{Z}\mathfrak{S}_n$ for any $T\in\us{SYT}_n$. Poirier and Reutenauer in \cite{PR} proved that $\ms{PR}'=\bigoplus\limits_{T\in\us{\s{SYT}}}k\cdot cl(T)$ is a Hopf subalgebra of $\ms{MR}'$. Meanwhile, the ideal $J$ of $\ms{MR}$ generated by the Knuth equivalence relations forms a Hopf ideal of $\ms{MR}$ and thus gives the Hopf quotient $\ms{PR}=\ms{MR}/J$ such that $\ms{PR}$ and $\ms{PR}'$ are dual to each other with respect to the canonical pairing.

Now let's define the $q$-analogue of the PR-algebra. The original definition inspires us to naturally construct a $q$-Hopf subalgbra $\ms{PR}'_q$ of $\ms{MR}'_q$ and a $q$-Hopf quotient $\ms{PR}_q$ of $\ms{MR}_q$ again dual to each other with respect to the canonical pairing. Define the following $q$-analogue of elementary Knuth relations,
\[\begin{array}{lll}
  (K'_q)\quad&q yxz\sim yzx,& \mbox{ if } x<y\leq z,\\
  (K''_q)\quad&q xzy\sim zxy,& \mbox{ if } x\leq y<z.
\end{array}\]
Let $J_q$ be the $k$-submodule of $k\mathfrak{S}$ spanned by those $w'-q^{\ell(w')-\ell(w)}w$ such that $w\sim_K w'$. For example, $3124-q1324,~3124-1342\in J_q$.
\begin{definition}
  Define $\ms{PR}_q:=\ms{MR}_q/J_q$ and $\ms{PR}'_q$ to be the orthogonal complement of $J_q$ in $\ms{MR}'_q$. That is,
 \[\ms{PR}'_q=\bigoplus\limits_{T\in\us{\s{SYT}}}k\cdot c_q(T), ~c_q(T)=\sum\limits_{P(w)=T}q^{\ell(w)}w.\]
\end{definition}

\begin{theorem}
$\ms{PR}_q$ and $\ms{PR}'_q$ are graded dual to each other as $q$-Hopf algebras with respect to the canonical pairing $\lan,\ran:\ms{PR}'_q\times\ms{PR}_q\rightarrow k$.
\end{theorem}
\begin{proof}
By proposition \ref{mr}, once we prove that $J_q$ is a Hopf ideal of $\ms{MR}_q$, then $\ms{PR}_q$ is a $q$-Hopf quotient with $\ms{PR}'_q$ dual to it. On one hand, for any word $w$ without repeating letters, the inversion number of $w$ and $\us{st}(w)$ are the same. Hence by the definition of $*$, if
$w'-q^{\ell(w')-\ell(w)}w\in J_q$, then
\[(w'-q^{\ell(w')-\ell(w)}w)*w''=\sum\limits_{\us{\ti{st}}(u)=w,~\us{\ti{st}}(u')=w'
\atop\us{\ti{st}}(v)=w''}(u'-q^{\ell(u')-\ell(u)}u)v\in J_q,\]
where the sum is only over those $u,u',v$ such that $\us{alph}(u)\cup\us{alph}(v)=\us{alph}(u')\cup\us{alph}(v)=[n]$ for some $n$.
Similarly, $w''*(w'-q^{\ell(w')-\ell(w)}w)$ lies in $J_q$. It means $J_q$ is an ideal.

On the other hand, suppose $w,w'$ are up to an equivalence $(K'_q)$ and fix the index $i$ from the formula (\ref{deq}) of $\De_q$. If $x,y,z$ all lie in $\{1,\dots,i\}$ (or $\{i+1,\dots,n\}$), then the corresponding terms in (\ref{deq}) are still up to an equivalence $(K'_q)$. Otherwise, there are two other possibilities: (1) $x\in\{1,\dots,i\}$ and $y,z\in\{i+1,\dots,n\}$; (2)  $x,y\in\{1,\dots,i\}$ and $z\in\{i+1,\dots,n\}$. For both cases we have $\us{inv}(w,i)=\us{inv}(w',i)-1$, thus the corresponding terms of $\De_q(w'-q^{\ell(w')-\ell(w)}w)$ vanish and finally we have $\De_q(w'-q^{\ell(w')-\ell(w)}w)\in J_q\ot k\mathfrak{S}+k\mathfrak{S}\ot J_q$. It's similar for the case up to an equivalence $(K''_q)$. Hence, $J_q$ is also a coideal.
\end{proof}

Next we recall some notion concerning about tableaux. Given any $T\in\us{SSYT}(\la/\mu)$, we use the coordinate $(i,j)\in\la/\mu$ to locate the box lying in the $i$th row and the $j$th column of $T$ counting from left to right and top to bottom, also denote $T(i,j)$ the number occupying that box. For $i\in\mathbb{Z^+}$, we denote $i\in T$ if $i$ is an entry of $T$.
Define the \textit{row word} $w(T)$ of $T$ as the word obtained by reading the entries of $T$ ``from left to right and bottom to top'', which gives $P(w(T))=T$ by the row-insertion. Meanwhile, there's a dot product $\cdot$ on tableaux such that $w(T\cdot T')\sim_K w(T)w(T')$.

 Define the \textit{inversion number} of $T$ as
\[\us{inv}(T)=|\{((i,j),(i',j'))\in\la\times\la:T(i,j)\geq T(i',j'),~i<i'\}|,\]
and the \textit{standardization} $\us{st}(T)$ of $T$ as the unique standard tableau such that $\us{inv}(\us{st}(T))=\us{inv}(T)$. For instance,
\[T=\raisebox{1.8em}{\xy 0;/r.2pc/:
(0,0)*{};(15,0)*{}**\dir{-};
(0,-5)*{};(15,-5)*{}**\dir{-};
(0,-10)*{};(10,-10)*{}**\dir{-};
(0,-15)*{};(5,-15)*{}**\dir{-};
(0,0)*{};(0,-15)*{}**\dir{-};
(5,0)*{};(5,-15)*{}**\dir{-};
(10,0)*{};(10,-10)*{}**\dir{-};
(15,0)*{};(15,-5)*{}**\dir{-};
(2.5,-2.5)*{1};(7.5,-2.5)*{2};(12.5,-2.5)*{3};
(2.5,-7.5)*{2};(7.5,-7.5)*{4};(2.5,-12.5)*{3};
\endxy}~,~
\us{st}(T)=\raisebox{1.8em}{\xy 0;/r.2pc/:
(0,0)*{};(15,0)*{}**\dir{-};
(0,-5)*{};(15,-5)*{}**\dir{-};
(0,-10)*{};(10,-10)*{}**\dir{-};
(0,-15)*{};(5,-15)*{}**\dir{-};
(0,0)*{};(0,-15)*{}**\dir{-};
(5,0)*{};(5,-15)*{}**\dir{-};
(10,0)*{};(10,-10)*{}**\dir{-};
(15,0)*{};(15,-5)*{}**\dir{-};
(2.5,-2.5)*{1};(7.5,-2.5)*{3};(12.5,-2.5)*{5};
(2.5,-7.5)*{2};(7.5,-7.5)*{6};(2.5,-12.5)*{4};
\endxy}~,
\]
with $w(T)=324123,~w(\us{st}(T))=426135$.
Note that for any $T\in\us{SSYT}$, one can obtain $\us{st}(T)$  by defining the standard order $<_T$ of boxes determined by $T$ as follows,
\begin{equation}\label{st}
(i,j)<_T(i',j')\mbox{ if }T(i,j)<T(i',j')\mbox{ or }T(i,j)=T(i',j')\mbox{ with }j<j'.
\end{equation}
Now renumbering the boxes of $T$ in their standard order gives $\us{st}(T)$. Meanwhile, one also have $w(\us{st}(T))=\us{st}(w(T))$.

Given $\la\in\mathcal {P}$, let $T_\la$ to be the unique standard tableau of shape $\la$ such that $\us{inv}(T_\la)=0$. For any $T\in\us{SSYT}(\la)$, we define the \textit{sign} of $T$ as
\begin{equation}\label{sign}
\us{sign}(T)=(-1)^{\ell(w(\us{\s{st}}(T)))}.
\end{equation}
One can easily see that $\ell(w(\us{st}(T)))=\sum\limits_{i<j}\la_i\la_j-\us{inv}(T)$, thus $\us{sign}(T)=\us{sign}(T_\la)(-1)^{\us{\s{inv}}(T)}$.

We also need the notion of \textit{rectification} of skew tableaux. Given $S\in\us{SSYT}(\la/\mu)$, the rectification $\us{rect}(S)$ is the tableau obtained from $S$ by carrying out the Sch$\ddot{\mbox{u}}$tzenberger's \textit{slide} repeatedly. Note that $w(\us{rect}(S))\sim_Kw(S)$ for any skew tableau $S$. On the other hand, given $T\in\us{SYT}(\mu),~S\in\us{SYT}(\la/\mu)$, one shifts all entries of $S$ by $n$ ,where $n=|\mu|$, and concatenates the resulting skew tableau with $T$ naturally. It gives a standard tableau of shape $\la$, denoted $(T)_S$. For instance,
\[T=\raisebox{1.8em}{\xy 0;/r.2pc/:
(0,0)*{};(10,0)*{}**\dir{-};
(0,-5)*{};(10,-5)*{}**\dir{-};
(0,-10)*{};(5,-10)*{}**\dir{-};
(0,-15)*{};(5,-15)*{}**\dir{-};
(0,0)*{};(0,-15)*{}**\dir{-};
(5,0)*{};(5,-15)*{}**\dir{-};
(10,0)*{};(10,-5)*{}**\dir{-};
(2.5,-2.5)*{1};(7.5,-2.5)*{3};
(2.5,-7.5)*{2};(2.5,-12.5)*{4};
\endxy}~,~S=\raisebox{1.8em}{\xy 0;/r.2pc/:
(5,0)*{};(10,0)*{}**\dir{-};
(0,-5)*{};(10,-5)*{}**\dir{-};
(0,-10)*{};(10,-10)*{}**\dir{-};
(0,-15)*{};(5,-15)*{}**\dir{-};
(0,-5)*{};(0,-15)*{}**\dir{-};
(5,0)*{};(5,-15)*{}**\dir{-};
(10,0)*{};(10,-10)*{}**\dir{-};
(7.5,-2.5)*{2};(2.5,-7.5)*{1};
(7.5,-7.5)*{4};(2.5,-12.5)*{3};
\endxy}~,~
(T)_S=\raisebox{1.8em}{\xy 0;/r.2pc/:
(0,0)*{};(15,0)*{}**\dir{-};
(0,-5)*{};(15,-5)*{}**\dir{-};
(0,-10)*{};(15,-10)*{}**\dir{-};
(0,-15)*{};(10,-15)*{}**\dir{-};
(0,0)*{};(0,-15)*{}**\dir{-};
(5,0)*{};(5,-15)*{}**\dir{-};
(10,0)*{};(10,-15)*{}**\dir{-};
(15,0)*{};(15,-10)*{}**\dir{-};
(2.5,-2.5)*{1};(7.5,-2.5)*{3};(12.5,-2.5)*{6};
(2.5,-7.5)*{2};(7.5,-7.5)*{5};(12.5,-7.5)*{8};
(2.5,-12.5)*{4};(7.5,-12.5)*{7};
\endxy}~.\]

Now we are in the position to describe the $q$-Hopf algebra structure of $\ms{PR}_q$ and $\ms{PR}'_q$ by modifying the classical case. First note that the duality between $\ms{PR}_q$ and $\ms{PR}'_q$ means
\[\lan c_q(T),q^{-\ell(\al)}(\al+J_q)\ran=\de_{P(\al),T},~\al\in\mathfrak{S},T\in\us{SYT}.\]
Hence, we can denote $q^{-\ell(\al)}(\al+J_q)$ by $c^*_q(T)$ dual to $c_q(T)$.

\begin{proposition}
For $(\ms{PR}_q,*,\De_q)$, we have
\[c^*_q(T_1)*c^*_q(T_2)=\sum\limits_{T'_1\cdot T'_2\in\us{\ti{SYT}}\atop \us{\ti{st}}(T_1')=T_1,~\us{\ti{st}}(T'_2)=T_2}q^{\us{\s{inv}}(T'_1,T'_2)}c^*_q(T'_1\cdot T'_2),\]
\[\De_q(c^*_q(T))=\sum\limits_{T=(T')_S}c^*_q(T')\ot c^*_q(\us{rect}(S)),\]
where for any $T,T'\in\us{SSYT}$ without repeated entries, $\us{inv}(T,T')$ is the number of pairs $(i,j)\in T\times T'$ such that $i>j$.

Dually for $(\ms{PR}'_q,*'_q,\De')$, we have
\begin{equation}\label{pr'm}
c_q(T_1)*'_q c_q(T_2)=\sum\limits_{T=(T_1)_S \atop \us{\ti{rect}}(S)=T_2}c_q(T),
\end{equation}
\[\De'(c_q(T))=\sum\limits_{T_1,T_2\in\us{\ti{SSYT}}\atop T=T_1\cdot T_2}q^{\us{\s{inv}}(T_1,T_2)}c_q(\us{st}(T_1))\ot c_q(\us{st}(T_2)).\]
\end{proposition}
\begin{proof}
One can deduce the formula for $*$ from $*$ of $\ms{MR}$ when paying attention to the inversion numbers of words, as
$\us{st}(u)\sim_Kw(T_1),\us{st}(v)\sim_Kw(T_2)
\Rightarrow \us{st}(P(u))=P(\us{st}(u))=T_1,\us{st}(P(v))=P(\us{st}(v))=T_2$ and $P(uv)=P(u)\cdot P(v)$. The formula for $\De_q$ comes from (\ref{deq}) together with the identity $\us{inv}(i,w)+\ell(w|_{\{1,\dots,i\}})
+\ell(\us{st}(w|_{\{i+1,\dots,n\}}))=\ell(w)$.
\end{proof}

\subsection{$q$-analogue of the commuting diagram}
The most remarkable property of the PR-algebra $\ms{PR}'$ is that it can factor through the embedding $\te\circ\iota:\ms{NSym}\rightarrow\ms{MR}'$ and has the image lying in $\ms{Sym}$ when projecting to $\ms{QSym}$. That's the following commuting  diagram of Hopf algebras (See also \cite[Section 1]{DHT}),
\begin{equation}\label{dia}
\xymatrix@H=1.5em{
\ms{NSym }\ar@(ur,ul)[rr]^-{\phi} \ar@{>->}[d]^-\iota \ar@{>->}[r]& \ms{PR' } \ar@{>>}[r] \ar@{>->}[d]&\ms{Sym }\ar@{>->}[d]\\
\ms{MR }\ar[r]^-{\te\atop\sim} & \ms{MR' } \ar@{>>}[r]^-{\pi'}&\ms{QSym}
}
\end{equation}

On the other hand, it's nice to see that the image of $cl(T)$ in $\ms{Sym}$ is the \textit{Schur function} $s_\la$ with the following combinatorial definition,
\[s_\la=\sum\limits_{T\in\us{\s{SSYT}}(\la)}x^{\us{\s{cont}}(T)},~\la\in\mathcal {P},\]
where the \textit{content} of $T$, denoted $\us{cont}(T)$, is the weak composition $\ga$ where $\ga_i$ denotes the number of entries of $T$ with value $i$.
 Indeed, there exists a natural bijection as follows \cite[Prop. 2.15]{BLW}, \cite[Prop. 5.3.6]{Sag},
\[f:\us{SSYT}(\la)\rightarrow\{(T,\ga)\in\us{SYT}(\la)\times C_w:\ga\succeq c(T)\},~T\mapsto (\us{st}(T),\us{cont}(T)).\]
Now given $T\in\us{SYT}(\la)$,
\[\begin{split}
 \pi'|_{\ms{\s{PR'}}}(cl(T))&=\sum\limits_{P(w)=T}F_{c(w)}
=\sum\limits_{P(w)=T}F_{ c(Q(w))}=\sum\limits_{Q\in\us{\s{SYT}}(\la)}F_{c(Q)}
=\sum\limits_{Q\in\us{\s{SYT}}(\la)}\sum\limits_{\al\succeq c(Q)}M_\al\\
&=\sum\limits_{Q\in\us{\s{SYT}}(\la)}\sum\limits_{\ga\in C_w\atop\ga\succeq c(Q)}x^\ga=\sum\limits_{U\in\us{\s{SSYT}}(\la)}x^{\us{\s{cont}}(U)}=s_\la,
\end{split}\]
where we use the RSK-correspondence and the bijection $f$. The Schur functions form a orthonormal basis of $\ms{Sym}$ with respect to the bilinear form $(\cdot,\cdot)$.
Therefore, $\ms{PR}'$ is also referred as the algebra of \textit{free symmetric functions} $\ms{FSym}$ and $cl(T)$ is called the \textit{free Schur function}.

Now from proposition \ref{mr}, we have
\begin{proposition}
The following commuting diagram of $q$-Hopf algebras generalizes (\ref{dia}).
\begin{equation}\label{diaq}
\xymatrix@H=1.5em{
\ms{NSym }\ar@(ur,ul)[rr]^-{\phi} \ar@{>->}[d]^-{\iota_q} \ar@{>->}[r]& {\ms{PR}'_q} \ar@{.>>}[r]|-? \ar@{>->}[d]&{\ms{Sym}_q}\ar@{>->}[d]\\
{\ms{MR}_q} \ar[r]^-{\te_q\atop\sim} & \ms{MR}'_q \ar@{>>}[r]^-{\pi'_q}&\ms{QSym}_q
}
\end{equation}
where the question mark means we can't judge whether the image of $\pi'_q$ restricting on $\ms{PR}'$ lies in $\ms{Sym}_q$, for the difficulty to figure out the kernel of $\phi$ in general.
\end{proposition}
\begin{proof}
In fact,
\[
\begin{split}
\te_q\circ\iota_q(R_\al)&=\te_q(D_\al)=\sum\limits_{c(w)=\al}q^{\ell(w)}w^{-1}
=\sum\limits_{c(Q(w))=\al}q^{\ell(w)}w^{-1}\\
&=\sum\limits_{c(P(w))=\al}q^{\ell(w)}w
=\sum\limits_{T\in\us{\ti{SYT}}\atop c(T)=\al}\sum\limits_{P(w)=T}q^{\ell(w)}w=\sum\limits_{T\in\us{\ti{SYT}}\atop c(T)=\al}c_q(T)\in\ms{PR}'_q,
\end{split}
\]
where we use the symmetric theorem \cite[\S4.1]{Ful}, $(P(w^{-1}),Q(w^{-1}))=(Q(w),P(w))$ and the identity $c(w)=c(Q(w))$ for any $w\in\mathfrak{S}$.
\end{proof}
Though we can't eliminate the question mark for $q$ in general, the answer is affirmative for the odd case, due to the relation (\ref{odd}). In the forthcoming section, we will use the $q$-Hopf algebra $\ms{PR}'_q$ to realize the odd Schur functions, as a basis of the odd symmetric functions.

\section{Application to the odd symmetric functions}
 Motivated by the ``oddification'' of Khovanov homology, Ellis, Khovanov and Lauda introduced the ring of odd symmetric functions and defined three versions of odd Schur functions in a sequence of papers \cite{EK}, \cite{EKL}, \cite{Ell}. In the last one, Ellis proved that all three definitions coincide to obtain the odd Littlewood-Richardson rule.

\subsection{Odd Schur functions}
 Our approach here bases on the $q$-Hopf algebra structure $\ms{PR}'_q$, and derives the odd Schur functions naturally. In this section, we write $\pi'_q(c_q(T))$, $\ms{Sym}_q$ as $\vs_T$ and $\ms{OSym}$ respectively for $q=-1$. First we need the following ``shuffle'' lemma concerning about two tableaux of row shape.
 \begin{lemma}\label{shu}
  For fixed $n\in\mathbb{Z}^+$, one can divide $[n]$ into two subsets $\{i_1<\dots<i_p\},~\{j_1<\dots<j_q\}$ with $p+q=n$. Now denote $S_{p,q}^n$ the set of all splits of $[n]$ into two parts of size $p$ and $q$ respectively, then

(1) for any fixed pair $(p,q),~p+q=n$, the map
\[S_{p,q}^n\rightarrow \us{SYT}_n,~(\{i_1<\dots<i_p\},~\{j_1<\dots<j_q\})\mapsto
~\xy 0;/r.15pc/:
(0,2.5)*{};(20,2.5)*{}**\dir{-};
(0,-2.5)*{};(20,-2.5)*{}**\dir{-};
(0,2.5)*{};(0,-2.5)*{}**\dir{-};
(5,2.5)*{};(5,-2.5)*{}**\dir{-};
(15,2.5)*{};(15,-2.5)*{}**\dir{-};
(20,2.5)*{};(20,-2.5)*{}**\dir{-};
(2.5,0)*{\stt{i_1}};(10,0)*{\dots};
(17.5,0)*{\stt{i_p}};
\endxy\cdot\xy 0;/r.15pc/:
(0,2.5)*{};(20,2.5)*{}**\dir{-};
(0,-2.5)*{};(20,-2.5)*{}**\dir{-};
(0,2.5)*{};(0,-2.5)*{}**\dir{-};
(5,2.5)*{};(5,-2.5)*{}**\dir{-};
(15,2.5)*{};(15,-2.5)*{}**\dir{-};
(20,2.5)*{};(20,-2.5)*{}**\dir{-};
(2.5,0)*{\stt{j_1}};(10,0)*{\dots};
(17.5,0)*{\stt{j_q}};
\endxy~\]
is injective.

(2) Given $T\in\us{SYT}(m,n-m),~n\in[m,2m]$, for any $p\in[n-m,m]$ and $q=n-p$, there exists a unique split $(\{i_1<\dots<i_p\},~\{j_1<\dots<j_q\})\in S_{p,q}^n$ such that $T=~\xy 0;/r.15pc/:
(0,2.5)*{};(20,2.5)*{}**\dir{-};
(0,-2.5)*{};(20,-2.5)*{}**\dir{-};
(0,2.5)*{};(0,-2.5)*{}**\dir{-};
(5,2.5)*{};(5,-2.5)*{}**\dir{-};
(15,2.5)*{};(15,-2.5)*{}**\dir{-};
(20,2.5)*{};(20,-2.5)*{}**\dir{-};
(2.5,0)*{\stt{i_1}};(10,0)*{\dots};
(17.5,0)*{\stt{i_p}};
\endxy\cdot\xy 0;/r.15pc/:
(0,2.5)*{};(20,2.5)*{}**\dir{-};
(0,-2.5)*{};(20,-2.5)*{}**\dir{-};
(0,2.5)*{};(0,-2.5)*{}**\dir{-};
(5,2.5)*{};(5,-2.5)*{}**\dir{-};
(15,2.5)*{};(15,-2.5)*{}**\dir{-};
(20,2.5)*{};(20,-2.5)*{}**\dir{-};
(2.5,0)*{\stt{j_1}};(10,0)*{\dots};
(17.5,0)*{\stt{j_q}};
\endxy~$.

(3) If $T=~\xy 0;/r.15pc/:
(0,2.5)*{};(20,2.5)*{}**\dir{-};
(0,-2.5)*{};(20,-2.5)*{}**\dir{-};
(0,2.5)*{};(0,-2.5)*{}**\dir{-};
(5,2.5)*{};(5,-2.5)*{}**\dir{-};
(15,2.5)*{};(15,-2.5)*{}**\dir{-};
(20,2.5)*{};(20,-2.5)*{}**\dir{-};
(2.5,0)*{\stt{i_1}};(10,0)*{\dots};
(17.5,0)*{\stt{i_p}};
\endxy\cdot\xy 0;/r.15pc/:
(0,2.5)*{};(20,2.5)*{}**\dir{-};
(0,-2.5)*{};(20,-2.5)*{}**\dir{-};
(0,2.5)*{};(0,-2.5)*{}**\dir{-};
(5,2.5)*{};(5,-2.5)*{}**\dir{-};
(15,2.5)*{};(15,-2.5)*{}**\dir{-};
(20,2.5)*{};(20,-2.5)*{}**\dir{-};
(2.5,0)*{\stt{j_1}};(10,0)*{\dots};
(17.5,0)*{\stt{j_q}};
\endxy~\in\us{SYT}(m,n-m)$, where $n\in[m,2m]$ and $p+q=n$, then
\begin{equation}\label{sig}
(-1)^{\us{\s{inv}}\lb~\xy 0;/r.15pc/:
(0,2.5)*{};(20,2.5)*{}**\dir{-};
(0,-2.5)*{};(20,-2.5)*{}**\dir{-};
(0,2.5)*{};(0,-2.5)*{}**\dir{-};
(5,2.5)*{};(5,-2.5)*{}**\dir{-};
(15,2.5)*{};(15,-2.5)*{}**\dir{-};
(20,2.5)*{};(20,-2.5)*{}**\dir{-};
(2.5,0)*{\stt{i_1}};(10,0)*{\dots};
(17.5,0)*{\stt{i_p}};
\endxy~,~\xy 0;/r.15pc/:
(0,2.5)*{};(20,2.5)*{}**\dir{-};
(0,-2.5)*{};(20,-2.5)*{}**\dir{-};
(0,2.5)*{};(0,-2.5)*{}**\dir{-};
(5,2.5)*{};(5,-2.5)*{}**\dir{-};
(15,2.5)*{};(15,-2.5)*{}**\dir{-};
(20,2.5)*{};(20,-2.5)*{}**\dir{-};
(2.5,0)*{\stt{j_1}};(10,0)*{\dots};
(17.5,0)*{\stt{j_q}};
\endxy~\rb}=(-1)^{(p-n+m)(n-m)}\us{sign}(T).
\end{equation}
 \end{lemma}
\begin{proof}
What we need is an important proposition in \cite[\S1.1]{Ful}. Roughly speaking, when row inserting $x_1,\dots,x_k$ in a tableau $T$ successively, if $x_1\leq\dots\leq x_k$ (resp. $x_1>\dots>x_k$), then
no two of the new boxes are in the same column (resp. row). Conversely, given a tableau $U$ of shape $\la$ and a skew shape $\la/\mu,~k=|\la/\mu|$, if no two of boxes in $\la/\mu$ are in the same column (resp. row), then there exists a
unique tableau $T$ of shape $\mu$, and unique $x_1\leq\dots\leq x_k$ (resp. $x_1>\dots>x_k$) such that $U=(T\leftarrow x_1)\leftarrow\cdots\leftarrow x_k$.

Now return to our situation. First note that $T=~\xy 0;/r.15pc/:
(0,2.5)*{};(20,2.5)*{}**\dir{-};
(0,-2.5)*{};(20,-2.5)*{}**\dir{-};
(0,2.5)*{};(0,-2.5)*{}**\dir{-};
(5,2.5)*{};(5,-2.5)*{}**\dir{-};
(15,2.5)*{};(15,-2.5)*{}**\dir{-};
(20,2.5)*{};(20,-2.5)*{}**\dir{-};
(2.5,0)*{\stt{i_1}};(10,0)*{\dots};
(17.5,0)*{\stt{i_p}};
\endxy~\cdot~\xy 0;/r.15pc/:
(0,2.5)*{};(20,2.5)*{}**\dir{-};
(0,-2.5)*{};(20,-2.5)*{}**\dir{-};
(0,2.5)*{};(0,-2.5)*{}**\dir{-};
(5,2.5)*{};(5,-2.5)*{}**\dir{-};
(15,2.5)*{};(15,-2.5)*{}**\dir{-};
(20,2.5)*{};(20,-2.5)*{}**\dir{-};
(2.5,0)*{\stt{j_1}};(10,0)*{\dots};
(17.5,0)*{\stt{j_q}};
\endxy~=\lb~\xy 0;/r.15pc/:
(0,2.5)*{};(20,2.5)*{}**\dir{-};
(0,-2.5)*{};(20,-2.5)*{}**\dir{-};
(0,2.5)*{};(0,-2.5)*{}**\dir{-};
(5,2.5)*{};(5,-2.5)*{}**\dir{-};
(15,2.5)*{};(15,-2.5)*{}**\dir{-};
(20,2.5)*{};(20,-2.5)*{}**\dir{-};
(2.5,0)*{\stt{i_1}};(10,0)*{\dots};
(17.5,0)*{\stt{i_p}};
\endxy~\leftarrow j_1\rb\leftarrow\cdots\leftarrow j_q$. Since
  $j_1<\cdots<j_q$, the new boxes are in distinct columns, and $T$ has shape $\la$ with at most two parts. Now the shape $\mu$ is fixed to be $(p)$, the uniqueness in the proposition above proves (1). For (2), fix $T\in\us{SYT}(m,n-m),~m\leq n\leq2m$ and $p\in[n-m,m]$. Take the skew subtableau of $T$ in $(m,n-m)/(p)$, then no two boxes of it are in the same column, illustrated by the figure below.
  \vspace{-1em}
  \[\xy 0;/r.2pc/:
(0,0)*{};(60,0)*{}**\dir{-};
(0,-5)*{};(60,-5)*{}**\dir{-};
(0,-10)*{};(20,-10)*{}**\dir{-};
(0,0)*{};(0,-10)*{}**\dir{-};
(5,0)*{};(5,-10)*{}**\dir{-};
(15,0)*{};(15,-10)*{}**\dir{-};
(20,0)*{};(20,-10)*{}**\dir{-};
(25,0)*{};(25,-5)*{}**\dir{-};
(35,0)*{};(35,-5)*{}**\dir{-};
(40,0)*{};(40,-5)*{}**\dir{-};
(45,0)*{};(45,-5)*{}**\dir{-};
(55,0)*{};(55,-5)*{}**\dir{-};
(60,0)*{};(60,-5)*{}**\dir{-};
(2.5,-2.5)*{*};(10,-2.5)*{\dots};(17.5,-2.5)*{*};
(22.5,-2.5)*{*};(30,-2.5)*{\dots};(37.5,-2.5)*{*};
(50,-2.5)*{\dots};(10,-7.5)*{\dots};
(20,3.5)*{\overbrace{\hspace{8.6em}}^{p}};
(50,3.5)*{\overbrace{\hspace{4.2em}}^{m-p}};
(10,-13.5)*{\underbrace{\hspace{4.2em}}_{n-m}};
\endxy\]
Hence the previous proposition gives us the unique product decomposition $T=~\xy 0;/r.15pc/:
(0,2.5)*{};(20,2.5)*{}**\dir{-};
(0,-2.5)*{};(20,-2.5)*{}**\dir{-};
(0,2.5)*{};(0,-2.5)*{}**\dir{-};
(5,2.5)*{};(5,-2.5)*{}**\dir{-};
(15,2.5)*{};(15,-2.5)*{}**\dir{-};
(20,2.5)*{};(20,-2.5)*{}**\dir{-};
(2.5,0)*{\stt{i_1}};(10,0)*{\dots};
(17.5,0)*{\stt{i_p}};
\endxy\cdot\xy 0;/r.15pc/:
(0,2.5)*{};(20,2.5)*{}**\dir{-};
(0,-2.5)*{};(20,-2.5)*{}**\dir{-};
(0,2.5)*{};(0,-2.5)*{}**\dir{-};
(5,2.5)*{};(5,-2.5)*{}**\dir{-};
(15,2.5)*{};(15,-2.5)*{}**\dir{-};
(20,2.5)*{};(20,-2.5)*{}**\dir{-};
(2.5,0)*{\stt{j_1}};(10,0)*{\dots};
(17.5,0)*{\stt{j_q}};
\endxy~$.

For (3) we apply $2m-n$ times the reverse slides to $T$ while always choosing the outside corner in the second row of the intermediate tableau, then the resulting tableau is a skew one of shape $(n,n)/(2m-n)$. One can check \cite[\S1.2]{Ful} for details of reverse slides. Let's illustrate it by the following example,
\[T=~\raisebox{1.3em}{\xy 0;/r.2pc/:
(0,0)*{};(25,0)*{}**\dir{-};
(0,-5)*{};(25,-5)*{}**\dir{-};
(0,-10)*{};(10,-10)*{}**\dir{-};
(0,0)*{};(0,-10)*{}**\dir{-};
(5,0)*{};(5,-10)*{}**\dir{-};
(10,0)*{};(10,-10)*{}**\dir{-};
(15,0)*{};(15,-5)*{}**\dir{-};
(20,0)*{};(20,-5)*{}**\dir{-};
(25,0)*{};(25,-5)*{}**\dir{-};
(2.5,-2.5)*{\mbox{\textbf{1}}};(7.5,-2.5)*{2};(12.5,-2.5)*{4};(17.5,-2.5)*{5};(22.5,-2.5)*{7};
(2.5,-7.5)*{\mbox{\textbf{3}}};(7.5,-7.5)*{\mbox{\textbf{6}}};(12.5,-7.5)*{\bullet};
\endxy}~\rightsquigarrow~\raisebox{1.3em}{\xy 0;/r.2pc/:
(0,0)*{};(25,0)*{}**\dir{-};
(0,-5)*{};(25,-5)*{}**\dir{-};
(0,-10)*{};(15,-10)*{}**\dir{-};
(0,0)*{};(0,-10)*{}**\dir{-};
(5,0)*{};(5,-10)*{}**\dir{-};
(10,0)*{};(10,-10)*{}**\dir{-};
(15,0)*{};(15,-10)*{}**\dir{-};
(20,0)*{};(20,-5)*{}**\dir{-};
(25,0)*{};(25,-5)*{}**\dir{-};
(2.5,-2.5)*{*};(7.5,-2.5)*{\mbox{\textbf{2}}};(12.5,-2.5)*{\mbox{\textbf{4}}};(17.5,-2.5)*{5};(22.5,-2.5)*{7};
(2.5,-7.5)*{1};(7.5,-7.5)*{3};(12.5,-7.5)*{\mbox{\textbf{6}}};(17.5,-7.5)*{\bullet};
\endxy}~\rightsquigarrow~\raisebox{1.3em}{\xy 0;/r.2pc/:
(0,0)*{};(25,0)*{}**\dir{-};
(0,-5)*{};(25,-5)*{}**\dir{-};
(0,-10)*{};(20,-10)*{}**\dir{-};
(0,0)*{};(0,-10)*{}**\dir{-};
(5,0)*{};(5,-10)*{}**\dir{-};
(10,0)*{};(10,-10)*{}**\dir{-};
(15,0)*{};(15,-10)*{}**\dir{-};
(20,0)*{};(20,-10)*{}**\dir{-};
(25,0)*{};(25,-5)*{}**\dir{-};
(2.5,-2.5)*{*};(7.5,-2.5)*{*};(12.5,-2.5)*{\mbox{\textbf{2}}};(17.5,-2.5)*{\mbox{\textbf{5}}};(22.5,-2.5)*{\mbox{\textbf{7}}};
(2.5,-7.5)*{1};(7.5,-7.5)*{3};(12.5,-7.5)*{4};(17.5,-7.5)*{6};(22.5,-7.5)*{\bullet};
\endxy}~\rightsquigarrow~\raisebox{1.3em}{\xy 0;/r.2pc/:
(0,0)*{};(25,0)*{}**\dir{-};
(0,-5)*{};(25,-5)*{}**\dir{-};
(0,-10)*{};(25,-10)*{}**\dir{-};
(0,0)*{};(0,-10)*{}**\dir{-};
(5,0)*{};(5,-10)*{}**\dir{-};
(10,0)*{};(10,-10)*{}**\dir{-};
(15,0)*{};(15,-10)*{}**\dir{-};
(20,0)*{};(20,-10)*{}**\dir{-};
(25,0)*{};(25,-10)*{}**\dir{-};
(2.5,-2.5)*{*};(7.5,-2.5)*{*};(12.5,-2.5)*{*};(17.5,-2.5)*{2};(22.5,-2.5)*{5};
(2.5,-7.5)*{1};(7.5,-7.5)*{3};(12.5,-7.5)*{4};(17.5,-7.5)*{6};(22.5,-7.5)*{7};
\endxy}~,\]
where $n=7,~m=5$. The bullets mark the outside corners and the numbers in bold show the routes of reverse slides. Note that in the whole process, given any $p\in[n-m,m]$, there exists exactly one step at which the intermediate tableau $U_p$ has $p$ boxes in the second row. Suppose $w(U_p)=i'_1\cdots i'_pj'_1\cdots j'_q$, then $\ell(w(U_p))=\us{inv}\lb~\xy 0;/r.15pc/:
(0,2.5)*{};(20,2.5)*{}**\dir{-};
(0,-2.5)*{};(20,-2.5)*{}**\dir{-};
(0,2.5)*{};(0,-2.5)*{}**\dir{-};
(5,2.5)*{};(5,-2.5)*{}**\dir{-};
(15,2.5)*{};(15,-2.5)*{}**\dir{-};
(20,2.5)*{};(20,-2.5)*{}**\dir{-};
(2.5,0)*{\stt{i'_1}};(10,0)*{\dots};
(17.5,0)*{\stt{i'_p}};
\endxy~,~\xy 0;/r.15pc/:
(0,2.5)*{};(20,2.5)*{}**\dir{-};
(0,-2.5)*{};(20,-2.5)*{}**\dir{-};
(0,2.5)*{};(0,-2.5)*{}**\dir{-};
(5,2.5)*{};(5,-2.5)*{}**\dir{-};
(15,2.5)*{};(15,-2.5)*{}**\dir{-};
(20,2.5)*{};(20,-2.5)*{}**\dir{-};
(2.5,0)*{\stt{j'_1}};(10,0)*{\dots};
(17.5,0)*{\stt{j'_q}};
\endxy~\rb$. Now as slides preserve the Knuth equivalence classes of words, we have
$T=P(w(U_p))=~\xy 0;/r.15pc/:
(0,2.5)*{};(20,2.5)*{}**\dir{-};
(0,-2.5)*{};(20,-2.5)*{}**\dir{-};
(0,2.5)*{};(0,-2.5)*{}**\dir{-};
(5,2.5)*{};(5,-2.5)*{}**\dir{-};
(15,2.5)*{};(15,-2.5)*{}**\dir{-};
(20,2.5)*{};(20,-2.5)*{}**\dir{-};
(2.5,0)*{\stt{i'_1}};(10,0)*{\dots};
(17.5,0)*{\stt{i'_p}};
\endxy\cdot\xy 0;/r.15pc/:
(0,2.5)*{};(20,2.5)*{}**\dir{-};
(0,-2.5)*{};(20,-2.5)*{}**\dir{-};
(0,2.5)*{};(0,-2.5)*{}**\dir{-};
(5,2.5)*{};(5,-2.5)*{}**\dir{-};
(15,2.5)*{};(15,-2.5)*{}**\dir{-};
(20,2.5)*{};(20,-2.5)*{}**\dir{-};
(2.5,0)*{\stt{j'_1}};(10,0)*{\dots};
(17.5,0)*{\stt{j'_q}};
\endxy~$. By the uniqueness in (1), we have
$\{i'_1<\cdots<i'_p\}=\{i_1<\cdots<i_p\},~\{j'_1<\cdots<j'_q\}=\{j_1<\cdots<j_q\}$.
Meanwhile, the key observation is the following identity  \[(-1)^{\ell(w(U_{p+1}))}=(-1)^{\ell(w(U_p))+n-m}\]
for any $p\in[n-m,m-1]$, due to chasing the number whose position changes from the first row to the second one when boxes slide.
Finally by the definition (\ref{sign}) and the previous identity, we get (\ref{sig}) as desired.
\end{proof}

Now we are in the position to prove that
 \begin{theorem}\label{osy}
 $\vs_T$ lies in $\ms{OSym}$ for any $T\in\us{SYT}$, thus there exists a surjection from $\ms{PR}'_q$ to $\ms{Sym}_q$ in (\ref{diaq}) for the odd case.
 \end{theorem}
\begin{proof}
In order to prove that $\vs_T\in\ms{OSym}$, we only need to check that $\vs_T\in(\us{Ker}\phi)^\bot$ with respect to the canonical pairing $\lan\cdot,\cdot\ran$ (Prop. \ref{bp} (3)). Due to the relation
(\ref{odd}), it is equivalent to
\begin{equation}\label{eo}
\begin{array}{l}
\lan\vs_T,H_pH_q-H_pH_q\ran=0,\mbox{ if }p+q\mbox{ even},\\
\lan\vs_T,H_pH_q+(-1)^pH_qH_p-(-1)^pH_{p+1}H_{q-1}-H_{q-1}H_{p+1}\ran=0,\mbox{ if }p+q\mbox{ odd}.
\end{array}
\end{equation}
Now $\vs_T=\sum\limits_{P(w)=T}(-1)^{\ell(w)}F_{c(w)}$, and
\[\begin{split}
\De'\circ\pi'_q(c_q(T))&=(\pi'_q\ot\pi'_q)\circ\De'(c_q(T))=\sum\limits_{T_1,T_2\in\us{\ti{SSYT}}\atop T=T_1\cdot T_2}q^{\us{\s{inv}}(T_1,T_2)}\pi'_q(c_q(\us{st}(T_1)))\ot\pi'_q(c_q(\us{st}(T_2)))\\
&=\sum\limits_{T_1,T_2\in\us{\ti{SSYT}}\atop T=T_1\cdot T_2}q^{\us{\s{inv}}(T_1,T_2)}\lb\sum\limits_{P(w_1)=\us{\s{st}}(T_1)}q^{\ell(w_1)}F_{c(w_1)}\rb
\ot\lb\sum\limits_{P(w_2)=\us{\s{st}}(T_2)}q^{\ell(w_2)}F_{c(w_2)}\rb
\end{split}\]
Since $F_\al=\sum\limits_{\be\succeq\al}M_\be$, we have $\lan F_\al,H_n\ran=\de_{\al,(n)}$. But $c(\al)=(n)\Leftrightarrow\al=1_n\in\mathfrak{S}_n\Leftrightarrow P(\al)=\xy 0;/r.15pc/:
(0,2.5)*{};(20,2.5)*{}**\dir{-};
(0,-2.5)*{};(20,-2.5)*{}**\dir{-};
(0,2.5)*{};(0,-2.5)*{}**\dir{-};
(5,2.5)*{};(5,-2.5)*{}**\dir{-};
(15,2.5)*{};(15,-2.5)*{}**\dir{-};
(20,2.5)*{};(20,-2.5)*{}**\dir{-};
(2.5,0)*{\stt{1}};(10,0)*{\dots};
(17.5,0)*{n};
\endxy~$.
Hence,
\[\begin{split}
\lan\pi'_q(c_q(T))&,H_pH_q\ran=\lan\De'\circ\pi'_q(c_q(T)),H_p\ot H_q\ran\\
&=\sum\limits_{T_1,T_2\in\us{\ti{SSYT}}\atop T=T_1\cdot T_2}q^{\us{\s{inv}}(T_1,T_2)}\lan\sum\limits_{P(w_1)=\us{\s{st}}(T_1)}q^{\ell(w_1)}F_{c(w_1)},H_p
\ran\lan\sum\limits_{P(w_2)=\us{\s{st}}(T_2)}q^{\ell(w_2)}F_{c(w_2)},H_q\ran\\
&=\sum\limits_{\{i_1<\cdots<i_p\}\cup\{j_1<\cdots<j_q\}=[p+q]}
q^{\us{\s{inv}}\lb~\xy 0;/r.15pc/:
(0,2.5)*{};(20,2.5)*{}**\dir{-};
(0,-2.5)*{};(20,-2.5)*{}**\dir{-};
(0,2.5)*{};(0,-2.5)*{}**\dir{-};
(5,2.5)*{};(5,-2.5)*{}**\dir{-};
(15,2.5)*{};(15,-2.5)*{}**\dir{-};
(20,2.5)*{};(20,-2.5)*{}**\dir{-};
(2.5,0)*{\stt{i_1}};(10,0)*{\dots};
(17.5,0)*{\stt{i_p}};
\endxy~,~\xy 0;/r.15pc/:
(0,2.5)*{};(20,2.5)*{}**\dir{-};
(0,-2.5)*{};(20,-2.5)*{}**\dir{-};
(0,2.5)*{};(0,-2.5)*{}**\dir{-};
(5,2.5)*{};(5,-2.5)*{}**\dir{-};
(15,2.5)*{};(15,-2.5)*{}**\dir{-};
(20,2.5)*{};(20,-2.5)*{}**\dir{-};
(2.5,0)*{\stt{j_1}};(10,0)*{\dots};
(17.5,0)*{\stt{j_q}};
\endxy~\rb}\de_{T,~\xy 0;/r.15pc/:
(0,2.5)*{};(20,2.5)*{}**\dir{-};
(0,-2.5)*{};(20,-2.5)*{}**\dir{-};
(0,2.5)*{};(0,-2.5)*{}**\dir{-};
(5,2.5)*{};(5,-2.5)*{}**\dir{-};
(15,2.5)*{};(15,-2.5)*{}**\dir{-};
(20,2.5)*{};(20,-2.5)*{}**\dir{-};
(2.5,0)*{\stt{i_1}};(10,0)*{\dots};
(17.5,0)*{\stt{i_p}};
\endxy~\cdot~\xy 0;/r.15pc/:
(0,2.5)*{};(20,2.5)*{}**\dir{-};
(0,-2.5)*{};(20,-2.5)*{}**\dir{-};
(0,2.5)*{};(0,-2.5)*{}**\dir{-};
(5,2.5)*{};(5,-2.5)*{}**\dir{-};
(15,2.5)*{};(15,-2.5)*{}**\dir{-};
(20,2.5)*{};(20,-2.5)*{}**\dir{-};
(2.5,0)*{\stt{j_1}};(10,0)*{\dots};
(17.5,0)*{\stt{j_q}};
\endxy}~.\\
\end{split}\]
Now when $p+q$ is even and $T=~\xy 0;/r.15pc/:
(0,2.5)*{};(20,2.5)*{}**\dir{-};
(0,-2.5)*{};(20,-2.5)*{}**\dir{-};
(0,2.5)*{};(0,-2.5)*{}**\dir{-};
(5,2.5)*{};(5,-2.5)*{}**\dir{-};
(15,2.5)*{};(15,-2.5)*{}**\dir{-};
(20,2.5)*{};(20,-2.5)*{}**\dir{-};
(2.5,0)*{\stt{i_1}};(10,0)*{\dots};
(17.5,0)*{\stt{i_p}};
\endxy\cdot\xy 0;/r.15pc/:
(0,2.5)*{};(20,2.5)*{}**\dir{-};
(0,-2.5)*{};(20,-2.5)*{}**\dir{-};
(0,2.5)*{};(0,-2.5)*{}**\dir{-};
(5,2.5)*{};(5,-2.5)*{}**\dir{-};
(15,2.5)*{};(15,-2.5)*{}**\dir{-};
(20,2.5)*{};(20,-2.5)*{}**\dir{-};
(2.5,0)*{\stt{j_1}};(10,0)*{\dots};
(17.5,0)*{\stt{j_q}};
\endxy~$, there also exists a unique split $(\{i'_1<\dots<i'_q\},\{j'_1<\dots<j'_p\})\in S^{p+q}_{q,p}$ such that
$T=~\xy 0;/r.15pc/:
(0,2.5)*{};(20,2.5)*{}**\dir{-};
(0,-2.5)*{};(20,-2.5)*{}**\dir{-};
(0,2.5)*{};(0,-2.5)*{}**\dir{-};
(5,2.5)*{};(5,-2.5)*{}**\dir{-};
(15,2.5)*{};(15,-2.5)*{}**\dir{-};
(20,2.5)*{};(20,-2.5)*{}**\dir{-};
(2.5,0)*{\stt{i'_1}};(10,0)*{\dots};
(17.5,0)*{\stt{i'_q}};
\endxy\cdot\xy 0;/r.15pc/:
(0,2.5)*{};(20,2.5)*{}**\dir{-};
(0,-2.5)*{};(20,-2.5)*{}**\dir{-};
(0,2.5)*{};(0,-2.5)*{}**\dir{-};
(5,2.5)*{};(5,-2.5)*{}**\dir{-};
(15,2.5)*{};(15,-2.5)*{}**\dir{-};
(20,2.5)*{};(20,-2.5)*{}**\dir{-};
(2.5,0)*{\stt{j'_1}};(10,0)*{\dots};
(17.5,0)*{\stt{j'_p}};
\endxy~$ and
\[(-1)^{\us{\s{inv}}\lb~\xy 0;/r.15pc/:
(0,2.5)*{};(20,2.5)*{}**\dir{-};
(0,-2.5)*{};(20,-2.5)*{}**\dir{-};
(0,2.5)*{};(0,-2.5)*{}**\dir{-};
(5,2.5)*{};(5,-2.5)*{}**\dir{-};
(15,2.5)*{};(15,-2.5)*{}**\dir{-};
(20,2.5)*{};(20,-2.5)*{}**\dir{-};
(2.5,0)*{\stt{i_1}};(10,0)*{\dots};
(17.5,0)*{\stt{i_p}};
\endxy~,~\xy 0;/r.15pc/:
(0,2.5)*{};(20,2.5)*{}**\dir{-};
(0,-2.5)*{};(20,-2.5)*{}**\dir{-};
(0,2.5)*{};(0,-2.5)*{}**\dir{-};
(5,2.5)*{};(5,-2.5)*{}**\dir{-};
(15,2.5)*{};(15,-2.5)*{}**\dir{-};
(20,2.5)*{};(20,-2.5)*{}**\dir{-};
(2.5,0)*{\stt{j_1}};(10,0)*{\dots};
(17.5,0)*{\stt{j_q}};
\endxy~\rb}=(-1)^{\us{\s{inv}}\lb~\xy 0;/r.15pc/:
(0,2.5)*{};(20,2.5)*{}**\dir{-};
(0,-2.5)*{};(20,-2.5)*{}**\dir{-};
(0,2.5)*{};(0,-2.5)*{}**\dir{-};
(5,2.5)*{};(5,-2.5)*{}**\dir{-};
(15,2.5)*{};(15,-2.5)*{}**\dir{-};
(20,2.5)*{};(20,-2.5)*{}**\dir{-};
(2.5,0)*{\stt{i'_1}};(10,0)*{\dots};
(17.5,0)*{\stt{i'_q}};
\endxy~,~\xy 0;/r.15pc/:
(0,2.5)*{};(20,2.5)*{}**\dir{-};
(0,-2.5)*{};(20,-2.5)*{}**\dir{-};
(0,2.5)*{};(0,-2.5)*{}**\dir{-};
(5,2.5)*{};(5,-2.5)*{}**\dir{-};
(15,2.5)*{};(15,-2.5)*{}**\dir{-};
(20,2.5)*{};(20,-2.5)*{}**\dir{-};
(2.5,0)*{\stt{j'_1}};(10,0)*{\dots};
(17.5,0)*{\stt{j'_p}};
\endxy~\rb}=\us{sign}(T)\]
by lemma \ref{shu} (2), (3). Hence, we prove the even case as the LHS of (\ref{eo}) vanishes in pairs for fixed $T$.

When $p+q$ is odd and  $T=~\xy 0;/r.15pc/:
(0,2.5)*{};(20,2.5)*{}**\dir{-};
(0,-2.5)*{};(20,-2.5)*{}**\dir{-};
(0,2.5)*{};(0,-2.5)*{}**\dir{-};
(5,2.5)*{};(5,-2.5)*{}**\dir{-};
(15,2.5)*{};(15,-2.5)*{}**\dir{-};
(20,2.5)*{};(20,-2.5)*{}**\dir{-};
(2.5,0)*{\stt{i_1}};(10,0)*{\dots};
(17.5,0)*{\stt{i_p}};
\endxy\cdot\xy 0;/r.15pc/:
(0,2.5)*{};(20,2.5)*{}**\dir{-};
(0,-2.5)*{};(20,-2.5)*{}**\dir{-};
(0,2.5)*{};(0,-2.5)*{}**\dir{-};
(5,2.5)*{};(5,-2.5)*{}**\dir{-};
(15,2.5)*{};(15,-2.5)*{}**\dir{-};
(20,2.5)*{};(20,-2.5)*{}**\dir{-};
(2.5,0)*{\stt{j_1}};(10,0)*{\dots};
(17.5,0)*{\stt{j_q}};
\endxy~\in\us{SYT}(m,n-m)$, by lemma \ref{shu} (2), (3) again, the LHS of (\ref{eo}) is equal to \[\begin{split}
&\lb(-1)^{(p-n+m)(n-m)}+(-1)^p(-1)^{(q-n+m)(n-m)}-(-1)^p(-1)^{(p+1-n+m)(n-m)}-(-1)^{(q-1-n+m)(n-m)}\rb \us{sign}(T)\\
&=(-1)^{(p-n+m)(n-m)}\lb1+(-1)^{p+n(n-m)}-(-1)^{p+n-m}-(-1)^{(n-1)(n-m)}\rb\us{sign}(T),
\end{split}
\]
which also vanishes as $n=p+q$ odd.
\end{proof}

Next we give the explicit formula for $\vs_T,~T\in\us{SYT}$.
In general $\ell(w)$ can't be read directly from the pair $(P(w),Q(w))$, but the sign of $w$ can. According to \cite[Theorem 4.3]{Rei}, we have the formula
\begin{equation}\label{sil}
(-1)^{\ell(w)}=(-1)^{{\la^T \choose 2}+\us{\s{inv}}(P(w))+\us{\s{inv}}(Q(w))}=(-1)^{\la^T \choose 2}\us{sign}(P(w))\us{sign}(Q(w)),
\end{equation}
where $\la=\us{sh}(P(w))$ and $\la^T$ is its transpose.
That gives us the chance to expand $\vs_T$ like the classical case via the bijection $f$ as follows. \[\begin{split}
\vs_T&=\sum\limits_{P(w)=T}(-1)^{\ell(w)}F_{c(w)}=
\sum\limits_{P(w)=T}(-1)^{{\la^T \choose 2}+\us{\s{inv}}(P(w))+\us{\s{inv}}(Q(w))}F_{c(w)}\\
&=(-1)^{{\la^T \choose 2}+\us{\s{inv}}(T)}
\sum\limits_{Q\in\us{\s{SYT}}(\la)}(-1)^{\us{\s{inv}}(Q)}F_{c(Q)}
=(-1)^{{\la^T \choose 2}+\us{\s{inv}}(T)}\sum\limits_{U\in\us{\s{SSYT}}(\la)}(-1)^{\us{\s{inv}}(U)}
x^{\us{\s{cont}}(U)},
\end{split}\]
where $T\in\us{SYT}(\la)$.

For any $U_1,U_2\in\us{SSYT}(\la)$, if $\us{st}(U_1)=\us{st}(U_2)$, then $\us{inv}(U_1)=\us{inv}(U_2)$ by definition. Hence, the above formula can be rewritten as an expansion of monomial quasisymmetric functions.
\begin{lemma}
 \begin{equation}\label{mon}
\vs_T=(-1)^{{\la^T \choose 2}+\us{\s{inv}}(T)}\sum\limits_{U\in\us{\ti{SSYT}}(\la)\atop \us{\ti{cont}}(U)\vDash|\la|}(-1)^{\us{\s{inv}}(U)}M_{\us{\s{cont}}(U)},~T\in\us{SYT}(\la).
 \end{equation}
\end{lemma}

Now we point out that when taking $T=T_\la$, $s_\la:=\vs_{T_\la}$ is just the odd Schur function defined in \cite[\S3.3]{EK}. In fact, one can define the \textit{odd monomial symmetric function} $m_\la$ as the dual of $h_\la$ with respect to the bilinear form $(\cdot,\cdot)$. By Theorem \ref{osy}, $\vs_T$ lies in $\ms{OSym}$, we can expand it as a linear combination of $m_\la$'s. Using formula (\ref{mon}), we have
\[\begin{split}
\vs_T&=\sum\limits_{\mu\in\mathcal{P}}(\vs_T,h_\mu)m_\mu
=\sum\limits_{\mu\in\mathcal{P}}\lan(-1)^{{\la^T \choose 2}+\us{\s{inv}}(T)}\sum\limits_{U\in\us{\ti{SSYT}}(\la)\atop \us{\ti{cont}}(U)\vDash|\la|}(-1)^{\us{\s{inv}}(U)}M_{\us{\s{cont}}(U)},H_\mu\ran m_\mu\\
&=(-1)^{{\la^T \choose 2}+\us{\s{inv}}(T)}\sum\limits_{\mu\in\mathcal{P}}\lb\sum\limits_{U\in\us{\ti{SSYT}}(\la)\atop \us{\ti{cont}}(U)=\mu}(-1)^{\us{\s{inv}}(U)}\rb m_\mu,
\end{split}
\]
where $T\in\us{SYT}(\la)$. In particular, as $\us{inv}(T_\la)=0$, we regain the formula in \cite[Cor. 3.8]{EK}
\begin{equation}\label{schur}
s_\la=(-1)^{\la^T \choose 2}\sum\limits_{\mu\in\mathcal{P}}OK_{\la\mu} m_\mu,
\end{equation}
where $OK_{\la\mu}:=\sum\limits_{U\in\us{\ti{SSYT}}(\la)\atop \us{\ti{cont}}(U)=\mu}(-1)^{\us{\s{inv}}(U)}=\us{sign}(T_\la)\sum\limits_{U\in\us{\ti{SSYT}}(\la)\atop \us{\ti{cont}}(U)=\mu}\us{sign}(U)$, called the \textit{odd Kostka number}. In comparison, the expansion formula (\ref{mon}) is more available than (\ref{schur}) when computing the odd Schur functions. For instance, let $\la=211$, then ${\la^T \choose 2}=3$ and we have the following tableaux in $\us{SSYT}(\la)$ whose contents are compositions,
\[\begin{array}{ccccccc}
~U~&~\raisebox{2em}{\xy 0;/r.2pc/:
(0,0)*{};(10,0)*{}**\dir{-};
(0,-5)*{};(10,-5)*{}**\dir{-};
(0,-10)*{};(5,-10)*{}**\dir{-};
(0,-15)*{};(5,-15)*{}**\dir{-};
(0,0)*{};(0,-15)*{}**\dir{-};
(5,0)*{};(5,-15)*{}**\dir{-};
(10,0)*{};(10,-5)*{}**\dir{-};
(2.5,-2.5)*{1};(7.5,-2.5)*{1};
(2.5,-7.5)*{2};(2.5,-12.5)*{3};
\endxy}~&
~\raisebox{2em}{\xy 0;/r.2pc/:
(0,0)*{};(10,0)*{}**\dir{-};
(0,-5)*{};(10,-5)*{}**\dir{-};
(0,-10)*{};(5,-10)*{}**\dir{-};
(0,-15)*{};(5,-15)*{}**\dir{-};
(0,0)*{};(0,-15)*{}**\dir{-};
(5,0)*{};(5,-15)*{}**\dir{-};
(10,0)*{};(10,-5)*{}**\dir{-};
(2.5,-2.5)*{1};(7.5,-2.5)*{2};
(2.5,-7.5)*{2};(2.5,-12.5)*{3};
\endxy}~&
~\raisebox{2em}{\xy 0;/r.2pc/:
(0,0)*{};(10,0)*{}**\dir{-};
(0,-5)*{};(10,-5)*{}**\dir{-};
(0,-10)*{};(5,-10)*{}**\dir{-};
(0,-15)*{};(5,-15)*{}**\dir{-};
(0,0)*{};(0,-15)*{}**\dir{-};
(5,0)*{};(5,-15)*{}**\dir{-};
(10,0)*{};(10,-5)*{}**\dir{-};
(2.5,-2.5)*{1};(7.5,-2.5)*{3};
(2.5,-7.5)*{2};(2.5,-12.5)*{3};
\endxy}~&
~\raisebox{2em}{\xy 0;/r.2pc/:
(0,0)*{};(10,0)*{}**\dir{-};
(0,-5)*{};(10,-5)*{}**\dir{-};
(0,-10)*{};(5,-10)*{}**\dir{-};
(0,-15)*{};(5,-15)*{}**\dir{-};
(0,0)*{};(0,-15)*{}**\dir{-};
(5,0)*{};(5,-15)*{}**\dir{-};
(10,0)*{};(10,-5)*{}**\dir{-};
(2.5,-2.5)*{1};(7.5,-2.5)*{2};
(2.5,-7.5)*{3};(2.5,-12.5)*{4};
\endxy}~&
~\raisebox{2em}{\xy 0;/r.2pc/:
(0,0)*{};(10,0)*{}**\dir{-};
(0,-5)*{};(10,-5)*{}**\dir{-};
(0,-10)*{};(5,-10)*{}**\dir{-};
(0,-15)*{};(5,-15)*{}**\dir{-};
(0,0)*{};(0,-15)*{}**\dir{-};
(5,0)*{};(5,-15)*{}**\dir{-};
(10,0)*{};(10,-5)*{}**\dir{-};
(2.5,-2.5)*{1};(7.5,-2.5)*{3};
(2.5,-7.5)*{2};(2.5,-12.5)*{4};
\endxy}~&
~\raisebox{2em}{\xy 0;/r.2pc/:
(0,0)*{};(10,0)*{}**\dir{-};
(0,-5)*{};(10,-5)*{}**\dir{-};
(0,-10)*{};(5,-10)*{}**\dir{-};
(0,-15)*{};(5,-15)*{}**\dir{-};
(0,0)*{};(0,-15)*{}**\dir{-};
(5,0)*{};(5,-15)*{}**\dir{-};
(10,0)*{};(10,-5)*{}**\dir{-};
(2.5,-2.5)*{1};(7.5,-2.5)*{4};
(2.5,-7.5)*{2};(2.5,-12.5)*{3};
\endxy}~\\
~\us{inv}(U)~&~0~&~1~&~2~&~0~&~1~&~2~\\
~\us{cont}(U)~&~211~&~121~&~112~&~1111~&~1111~&~1111~
\end{array}
\]
Hence, $s_{211}=-(M_{211}-M_{121}+M_{112}+M_{1111})$.

\subsection{Odd Littlewood-Richardson rule}
It's well known that the algebra structure of $\ms{PR}'$ implies the Littlewood-Richardson rule for Schur functions. Now we only need to apply $\pi'_q$ to the multiplication rule (\ref{pr'm}) of $\ms{PR}'_q$ for $q=-1$, which gives the \textit{odd Littlewood-Richardson rule} as follows,
\[s_\la s_\mu=\sum\limits_{\nu}oc_{\la\mu}^\nu s_\nu,\]
where $oc_{\la\mu}^\nu:=\sum\limits_{S\in\us{\ti{SYT}}(\nu/\la)\atop \us{\ti{rect}}(S)=T_\mu}(-1)^{\us{\s{inv}}((T_\la)_S)}$ if $\la\subset
\nu$ and 0 otherwise, called the \textit{odd Littlewood-Richardson number}. Note that we haven't used the Littlewood-Richardson skew tableaux to characterize $oc_{\la\mu}^\nu$ as in \cite{Ell}.

In particular, the formula (\ref{mon}) implies that $s_{(n)}=\sum\limits_{\al\vDash n}M_\al=h_n,~s_{(1^n)}=(-1)^{n\choose2}M_{(1^n)}=(-1)^{n\choose2}e_n$, as $(e_n,h_\la)=\de_{\la,(1^n)}$. Now we can write down the \textit{odd Pieri rule} as follows.
\begin{theorem}\label{pieri}
For any $\la\in\mathcal {P}$ and $n\in\mathbb{Z}^+$, we have

(1)\[s_\la s_{(n)}=\sum\limits_{\mu}oc_{\la(n)}^\mu s_\mu,\]
where the sum is over those $\mu$ such that $\mu/\la$ is a horizontal strip of size $n$ (no two boxes in the same column). Besides, $oc_{\la(n)}^\mu=(-1)^{\us{\s{inv}}((T_\la)_S)}$, where $S$ is a horizontal strip of shape $\mu/\la$ with boxes labeled $1,\dots,n$ successively from left to right.

(2)\[s_\la s_{(1^n)}=\sum\limits_{\mu}oc_{\la(1^n)}^\mu s_\mu,\]
where the sum is over those $\mu$ such that $\mu/\la$ is a vertical strip of size $n$ (no two boxes in the same row). Besides, $oc_{\la(1^n)}^\mu=(-1)^{\us{\s{inv}}((T_\la)_S)}$, where $S$ is a vertical strip of shape $\mu/\la$ with boxes labeled $1,\dots,n$ successively from top to bottom.
\end{theorem}

\begin{remark}\label{rem}
We translate our formulas into those in \cite[\S2.2.3]{EKL}. For a partition $\la$, let $\tfrac{i}{\la}$ be the Young diagram obtained by removing rows 1 through $i$ from the one corresponding to $\la$. Then
for the vertical strip case, ${\us{inv}((T_\la)_S)}=\left|\tfrac{i_1}{\la}\right|+\cdots+\left|\tfrac{i_n}{\la}\right|$, where $i_1,\dots,i_n$ are the rows of $\la$ to which a box is added. While for the horizontal strip case, ${\us{inv}((T_\la)_S)}=NE(\mu)-NE(\la)-\left|\tfrac{i_1}{\la^T}\right|-\cdots-\left|\tfrac{i_n}{\la^T}\right|$, where $NE(\la)$ is the total number of \textit{strictly southwest-northeast} (SW-NE) pairs of boxes in $\la$ and $i_1,\dots,i_n$ are the columns of $\la$ to which a box is added. Both cases accumulate the number of boxes strictly southwest to the new box added to the intermediate diagram step by step with increasing labels.
\end{remark}

In particular, we get the formula expanding $h_\la$ (resp. $e_\la$) as odd Schur functions by Theorem \ref{pieri} (1) (resp. (2)) as follows,
\begin{proposition}\label{hep}
\[\begin{array}{l}
h_\la=\sum\limits_\mu\lb\sum\limits_{(\la_1)=\nu_1\subset\nu_2\subset\cdots\subset\nu_r=\mu\atop \nu_i/\nu_{i-1}\mbox{\ti{ horizotal strip with $\la_i$ boxs}}}\prod\limits_{i=2}^rc_{\nu_{i-1}(\la_i)}^{\nu_i}\rb s_\mu
=\sum\limits_\mu OK_{\mu\la}s_\mu,\\
(-1)^{\la\choose2}e_\la=\sum\limits_\mu\lb\sum\limits_{(1^{\la_1})=\nu_1\subset\nu_2\subset\cdots\subset\nu_r=\mu\atop \nu_i/\nu_{i-1}\mbox{\ti{ vertical strip with $\la_i$ boxs}}}\prod\limits_{i=2}^rc_{\nu_{i-1}(1^{\la_i})}^{\nu_i}\rb s_\mu
=\sum\limits_\mu (-1)^{NE(\mu)}OK_{\mu^T\la}s_\mu,
\end{array}\]
where $\la=(\la_1,\dots,\la_r)\in\mathcal {P}$.
\end{proposition}

\begin{remark}
Note that the formula of $h_\la$ in Prop. \ref{hep} is used as the defining relation of $s_\la$ in \cite{EK}, while the formula of $e_\la$ is also given in \cite[Cor.3.12]{EK}.
\end{remark}

Now we have the orthogonality of $s_\la$'s,
 \begin{equation}\label{ort}
 (s_\la,s_\mu)=(-1)^{\la^T \choose 2}\de_{\la,\mu},
 \end{equation}
 for $\sum\limits_{\la\vdash n}(-1)^{\la^T \choose 2}s_\la\ot s_\la=\sum\limits_{\la\vdash n}(-1)^{\la^T \choose 2}\lb\sum\limits_\mu OK_{\la\mu}m_\mu\rb\ot s_\la=\sum\limits_{\mu\vdash n}m_\mu\ot(-1)^{\la^T \choose 2}\sum\limits_\la OK_{\la\mu}s_\la=\sum\limits_{\mu\vdash n}m_\mu\ot h_\mu$. From \cite[Lemma 3.11]{EK}, we have the formula
\[\psi_3(s_\la)=\us{sign}(T_\la)(-1)^{\la^T\choose2}s_\la.\]
Applying $\psi_3$ to Theorem \ref{pieri}, we get the left-sided version
\begin{equation}\label{rp1}
s_{(n)}s_\la=\sum\limits_{\mu}(-1)^{{\la^T\choose2}+{\mu^T\choose2}}
\us{sign}(T_\la)\us{sign}(T_\mu)oc_{\la(n)}^\mu s_\mu,
\end{equation}
\begin{equation}\label{rp2}
s_{(1^n)}s_\la=\sum\limits_{\mu}(-1)^{{\la^T\choose2}+{\mu^T\choose2}}
\us{sign}(T_\la)\us{sign}(T_\mu)oc_{\la(1^n)}^\mu s_\mu.
\end{equation}

\section{Odd quasisymmetric Schur function}
Recently, a new basis of $\ms{QSym}$, called the quasisymmetric Schur functions, has been discovered as a refinement of the Schur functions in $\ms{QSym}$ \cite{HLMW1},\cite{BLW}. It arises from a specialization of nonsymmetric Macdonald polynomials into standard bases. The quasisymmetric Schur functions indexed by $C$ have many similar properties as the Schur functions and denoted by $\{\mathcal {S}_\al\}_{\al\in C}$.

In this section, we show that such basis also admits an odd version as a refinement of the odd Schur functions. As we're dealing with Young tableaux, we choose a different version, called \textit{Young quasisymmetric Schur functions} \cite[Chapter 5]{LMW}, to work with and denote this basis by $\{\mathscr{S}_\al\}_{\al\in C}$. Under the transpose $\psi:=\psi_3^*$ of $\psi_3$ in $\ms{QSym}$ with $\psi(F_\al)=F_{\ol{\al}}$, one have $\mathscr{S}_\al=\psi(\mathcal {S}_{\ol{\al}})$.

\subsection{Combinatorics of Young composition tableaux}
First we introduce the notion of \textit{Young composition tableaux} defined as follows \cite[Chapter 4]{LMW}.
Given $\al=(\al_1,\dots,\al_r)\in C$, one have the \textit{Young composition diagram} of $\al$ drawing like Young diagrams with partitions replacing by compositions. For example,
\[\xy 0;/r.18pc/:
(0,0)*{};(10,0)*{}**\dir{-};
(0,-5)*{};(15,-5)*{}**\dir{-};
(0,-10)*{};(15,-10)*{}**\dir{-};
(0,-15)*{};(5,-15)*{}**\dir{-};
(0,0)*{};(0,-15)*{}**\dir{-};
(5,0)*{};(5,-15)*{}**\dir{-};
(10,0)*{};(10,-10)*{}**\dir{-};
(15,-5)*{};(15,-10)*{}**\dir{-};
(7.5,-20)*{\al=(2,3,1)};
\endxy\]
Here we still follow English convention different from that in \cite{LMW}.

\begin{definition}
The \textit{Young composition poset} $\mathcal {L}$ is the poset structure on $C$ in which $\al=(\al_1,\dots,\al_r)$ is covered by

\noindent 1. $(\al_1,\dots,\al_r,1)$, that is, the composition obtained by suffixing a part of size 1 to $\al$.

\noindent 2. $(\al_1,\dots,\al_k+1,\dots,\al_r)$, provided that $\al_i\neq\al_k$ for all $i>k$, that is, the composition
obtained by adding 1 to a part of $\al$ as long as that part is the rightmost part
of that size.
\end{definition}
We denote the corresponding partial order by $\al\lessdot\be$, and note that the Young's lattice is just a subposet of $\mathcal {L}$. Now let $\al,\be$ be two Young composition diagrams such that $\al\lessdot\be$. Then we define
the \textit{skew Young composition shape} $\al\pa\be$ to be the array of boxes
\[\al\pa\be=\{(i,j):(i,j)\in\al\mbox{ and }(i,j)\notin\be\}.\]
For example,
\[\xy 0;/r.18pc/:
(0,0)*{};(10,0)*{}**\dir{-};
(0,-5)*{};(15,-5)*{}**\dir{-};
(0,-10)*{};(15,-10)*{}**\dir{-};
(0,-15)*{};(5,-15)*{}**\dir{-};
(0,0)*{};(0,-15)*{}**\dir{-};
(5,0)*{};(5,-15)*{}**\dir{-};
(10,0)*{};(10,-10)*{}**\dir{-};
(15,-5)*{};(15,-10)*{}**\dir{-};
(7.5,-20)*{\al\pa\be=(2,3,1)\pa(2,1)};
(2.5,-2.5)*{\bullet};(7.5,-2.5)*{\bullet};(2.5,-7.5)*{\bullet};
\endxy\]

\begin{definition}
Given a skew Young composition shape $\al\pa\be$, we define a \textit{semistandard
Young composition tableau} (abbreviated to SSYCT) $\tau$ of shape $\us{sh}(\tau)=
\al\pa\be$ to be a filling $\tau:\al\pa\be\rightarrow \mathbb{Z}^+$
of the boxes of $\al\pa\be$ such that

\noindent 1. the entries in each row are weakly increasing when read from left to right

\noindent 2. the entries in the first column are strictly increasing when read from the row with the smallest index to the largest index

\noindent 3. (\textit{triple rule}) if $i>j$ and $(j,k+1)\in\al\pa\be$ and either $(i,k)\in\be$ or $\tau(i,k)\leq\tau(j,k+1)$, then
either $(i,k+1)\in\be$ or both $(i,k+1)\in\al\pa\be$ and $\tau(i,k+1)<\tau(j,k+1)$.

Moreover, $\tau$ is a \textit{standard Young composition tableau} (abbreviated to SYCT) if it is a bijection $\tau:\al\pa\be\rightarrow[|\al\pa\be|]$.
We abuse the notation to denote the set of semistandard (resp. standard)
Young composition tableaux also as SSYCT (resp. SYCT).
\end{definition}

The naturality of  the above conditions 1 to 3 can be seen by the following result
\begin{proposition}\label{ycp}
There exists a one-to-one correspondence between saturated chains in $\mathcal {L}$ and SYCT:
\[\al^0\lessdot\al^1\lessdot\cdots\lessdot\al^n\leftrightarrow \tau,\]
where $\tau$ is the SYCT of shape $\al^n\pa\al^0$ such that the number $i$ appears in the unique box of $\al^i\pa\al^{i-1}$.
\end{proposition}
The \textit{column sequence} $\us{col}(\al^0\lessdot\al^1\lessdot\cdots\lessdot\al^n)$ of a saturated chain $\al^0\lessdot\al^1\lessdot\cdots\lessdot\al^n$ is given by $c_1,\dots,c_n$, where the box of $\al^i\pa\al^{i-1},~i=1,\dots,n$, is in the $c_i$th column of $\al^n$. For example, the saturated chain in $\mathcal {L}$ \[(2,1)\lessdot(2,1,1)\lessdot(3,1,1)
\lessdot(3,1,2)\lessdot(3,2,2)\lessdot(3,2,3)\]
corresponds to the following SYCT
\[\xy 0;/r.18pc/:
(0,0)*{};(15,0)*{}**\dir{-};
(0,-5)*{};(15,-5)*{}**\dir{-};
(0,-10)*{};(15,-10)*{}**\dir{-};
(0,-15)*{};(15,-15)*{}**\dir{-};
(0,0)*{};(0,-15)*{}**\dir{-};
(5,0)*{};(5,-15)*{}**\dir{-};
(10,0)*{};(10,-15)*{}**\dir{-};
(15,0)*{};(15,-5)*{}**\dir{-};
(15,-10)*{};(15,-15)*{}**\dir{-};
(2.5,-2.5)*{\bullet};(7.5,-2.5)*{\bullet};(2.5,-7.5)*{\bullet};
(2.5,-12.5)*{1};(12.5,-2.5)*{2};(7.5,-7.5)*{4};(7.5,-12.5)*{3};
(12.5,-12.5)*{5};
\endxy\]
and the column sequence is $1,3,2,2,3$.

Similar to the case of Young tableaux, one can define the \textit{content} $\us{cont}(\tau)\in C_w$ of $\tau\in\us{SSYCT}$, the \textit{standardization} $\us{st}(\tau)$ of $\tau$ with respect to the order defined as (\ref{st}),
 the \textit{descent} set of $\tau\in\us{SYCT}_n$ as \[\us{des}(\tau)=\{k\in[n-1]:
k+1\mbox{ appears weakly left of }k\mbox{ in }\tau\},\]
and the associated composition $c(\tau)$.
Given $\al=(\al_1,\dots,\al_r)\in C$, there exists a unique SYCT $U_\al$
satisfying $\us{sh}(U_\al) =\al$ and $c(U_\al)=\al$. In $U_\al$ the $i$th row is filled with $\al_1+\cdots+\al_{i-1}+1,\dots,\al_1+\cdots+\al_i$ from left to right successively for $1\leq i\leq r$. For example,
\[U_{323}=~\raisebox{1.8em}{\xy 0;/r.18pc/:
(0,0)*{};(15,0)*{}**\dir{-};
(0,-5)*{};(15,-5)*{}**\dir{-};
(0,-10)*{};(15,-10)*{}**\dir{-};
(0,-15)*{};(15,-15)*{}**\dir{-};
(0,0)*{};(0,-15)*{}**\dir{-};
(5,0)*{};(5,-15)*{}**\dir{-};
(10,0)*{};(10,-15)*{}**\dir{-};
(15,0)*{};(15,-5)*{}**\dir{-};
(15,-10)*{};(15,-15)*{}**\dir{-};
(2.5,-2.5)*{1};(7.5,-2.5)*{2};(2.5,-7.5)*{4};
(2.5,-12.5)*{6};(12.5,-2.5)*{3};(7.5,-7.5)*{5};(7.5,-12.5)*{7};
(12.5,-12.5)*{8};
\endxy}~,\tau=~\raisebox{1.8em}{\xy 0;/r.18pc/:
(0,0)*{};(15,0)*{}**\dir{-};
(0,-5)*{};(15,-5)*{}**\dir{-};
(0,-10)*{};(15,-10)*{}**\dir{-};
(0,-15)*{};(15,-15)*{}**\dir{-};
(0,0)*{};(0,-15)*{}**\dir{-};
(5,0)*{};(5,-15)*{}**\dir{-};
(10,0)*{};(10,-15)*{}**\dir{-};
(15,0)*{};(15,-5)*{}**\dir{-};
(15,-10)*{};(15,-15)*{}**\dir{-};
(2.5,-2.5)*{\bullet};(7.5,-2.5)*{\bullet};(2.5,-7.5)*{\bullet};
(2.5,-12.5)*{1};(12.5,-2.5)*{1};(7.5,-7.5)*{3};(7.5,-12.5)*{2};
(12.5,-12.5)*{3};
\endxy}~\rightsquigarrow\us{st}(\tau)=~\raisebox{1.8em}{\xy 0;/r.18pc/:
(0,0)*{};(15,0)*{}**\dir{-};
(0,-5)*{};(15,-5)*{}**\dir{-};
(0,-10)*{};(15,-10)*{}**\dir{-};
(0,-15)*{};(15,-15)*{}**\dir{-};
(0,0)*{};(0,-15)*{}**\dir{-};
(5,0)*{};(5,-15)*{}**\dir{-};
(10,0)*{};(10,-15)*{}**\dir{-};
(15,0)*{};(15,-5)*{}**\dir{-};
(15,-10)*{};(15,-15)*{}**\dir{-};
(2.5,-2.5)*{\bullet};(7.5,-2.5)*{\bullet};(2.5,-7.5)*{\bullet};
(2.5,-12.5)*{1};(12.5,-2.5)*{2};(7.5,-7.5)*{4};(7.5,-12.5)*{3};
(12.5,-12.5)*{5};
\endxy}~,\us{des}(\us{st}(\tau))=\{2,3\}\]

Next we introduce a key bijection found by Mason \cite{Mas}. Let
SSYCT$(-\pa\al)$ denote the set of all SSYCT with \textit{base shape} $\al$, and SSYT$(-/\tilde{\al})$
denote the set of all SSYT with base shape $\tilde{\al}$.
\begin{proposition}[{\cite[\S 4.3]{LMW}}]
There exists a bijection
\[\rho_\al:\us{SSYCT}(-\pa\al)\rightarrow\us{SSYT}(-/\tilde{\al}),\]
  mapping $\tau\in\us{SSYCT}(-\pa\al)$ to be $\rho_\al(\tau)\in\us{SSYT}(-/\tilde{\al})$ obtained by writing the entries in each column of $\tau$ in increasing order and top justifying these new columns on the base shape $\tilde{\al}$ if it exists.
\end{proposition}
  When $\al=\emptyset$, we abbreviate $\rho_\emptyset$ as $\rho$. Note that by definition,  if given $\tau\in\us{SSYCT}(\al\pa\be)$, then $\rho_\be(\tau)\in\us{SSYT}(\tilde{\al}/\tilde{\be})$. Conversely, the inverse map $\rho_\al^{-1}:\us{SSYT}(-/\tilde{\al})\rightarrow\us{SSYCT}(-\pa\al)$
 is defined as follows. Given $T\in\us{SSYT}$,

\noindent 1. If the first column of the base shape has $i$ boxes, then take the set of entries in the
first column of $T$ and write them in increasing order in rows $i+1,i+2,\dots$ to form the first column of $\tau$.

\noindent 2. Take the set of entries in column 2 in increasing order and place them in the row
with the largest index so that either

 $\bullet$ the box to the immediate left of the number being placed is filled and the row
entries weakly increase when read from left to right

 $\bullet$ the box to the immediate left of the number being placed belongs to the base shape.

\noindent 3. Repeat the previous step with the set of entries in column $k$ for $k=3,\dots,\tilde{\al}_1$.

For example,
\[\raisebox{1.8em}{\xy 0;/r.18pc/:
(0,0)*{};(20,0)*{}**\dir{-};
(0,-5)*{};(20,-5)*{}**\dir{-};
(0,-10)*{};(15,-10)*{}**\dir{-};
(0,-15)*{};(10,-15)*{}**\dir{-};
(0,0)*{};(0,-15)*{}**\dir{-};
(5,0)*{};(5,-15)*{}**\dir{-};
(10,0)*{};(10,-15)*{}**\dir{-};
(15,0)*{};(15,-10)*{}**\dir{-};
(20,0)*{};(20,-5)*{}**\dir{-};
(2.5,-2.5)*{\bullet};(7.5,-2.5)*{\bullet};(12.5,-2.5)*{4};
(17.5,-2.5)*{5};
(2.5,-7.5)*{\bullet};
(2.5,-12.5)*{2};(7.5,-7.5)*{1};(7.5,-12.5)*{3};
(12.5,-7.5)*{6};
\endxy}~\stackrel{\rho^{-1}_{(1,2)}}{\longrightarrow}~
\raisebox{1.8em}{\xy 0;/r.18pc/:
(0,0)*{};(10,0)*{}**\dir{-};
(0,-5)*{};(15,-5)*{}**\dir{-};
(0,-10)*{};(20,-10)*{}**\dir{-};
(0,-15)*{};(20,-15)*{}**\dir{-};
(0,0)*{};(0,-15)*{}**\dir{-};
(5,0)*{};(5,-15)*{}**\dir{-};
(10,0)*{};(10,-15)*{}**\dir{-};
(15,-5)*{};(15,-15)*{}**\dir{-};
(20,-10)*{};(20,-15)*{}**\dir{-};
(2.5,-2.5)*{\bullet};(7.5,-2.5)*{1};(2.5,-7.5)*{\bullet};
(2.5,-12.5)*{2};(12.5,-7.5)*{6};(7.5,-7.5)*{\bullet};(7.5,-12.5)*{3};
(12.5,-12.5)*{4};(17.5,-12.5)*{5};
\endxy}\]

Now given $\tau\in\us{SSYCT}$, its \textit{column reading word}, denoted by $w_{col}(\tau)$, is obtained by
listing the entries from the leftmost column in decreasing order, followed by the entries from the second leftmost column, again in decreasing order, and so on. Define the \textit{rectification} $\us{rect}(\tau)$ as $\rho^{-1}(P(w_{col}(\tau)))$. Note that rectification preserves descents of tableaux, i.e. $\us{des}(\us{rect}(\tau))=\us{des}(\tau)$. As in \cite{BLW}, we define the \textit{$C$-equivalence} of permutations as
\[w\sim_Cw'\mbox{ if and only if }Q(w)=Q(w')\mbox{ and }\us{sh}(\rho^{-1}(P(w)))=\us{sh}(\rho^{-1}(P(w'))),\]
 then the \textit{$C$-equivalence} of SYCT with skew shapes as
\[\tau\sim_C\tau'\mbox{ if and only if }w_{col}(\tau)\sim_Cw_{col}(\tau')\mbox{ and }\us{sh}(\tau)=\us{sh}(\tau').\]
Denote the $C$-equivalence class of $w\in\mathfrak{S}$ (resp. $\tau\in\us{SYCT}$) by $[w]_C$ (resp. $[\tau]_C$) and it gives the following key result
\begin{proposition}[{\cite[Prop. 4.4]{BLW}}]\label{ceq}
For any $\tau\in\us{SYCT}$, $\{w_{col}(\tau'):\tau'\in[\tau]_C\}=[w_{col}(\tau)]_C$.
\end{proposition}
By the RSK-correspondence, Mason's bijection $\rho$ gives
\[\us{SYCT}(\al)
\stackrel{1:1}{\longleftrightarrow}\{P(w)\in\us{SYT}(\tilde{\al}):w\in[w']_C\}
\stackrel{1:1}{\longleftrightarrow}[w']_C\]
for any $w'\in\mathfrak{S}$ satisfying $\us{sh}(\rho^{-1}(P(w')))=\al$. Combining it with Prop. \ref{ceq}, we have
\begin{equation}\label{rect}
\us{SYCT}(\al)=\{\us{rect}(\tau):\tau\in[\tau']_C\}\stackrel{1:1}{\longleftrightarrow}[\tau']_C
\end{equation}
for any $\tau'\in\us{SYCT}$ satisfying $\us{sh}(\us{rect}(\tau'))=\al$. In particular, given a class $[\tau']_C$ with $\us{sh}(\us{rect}(\tau'))=\al$, there exists a unique $\tau\in[\tau']_C$ such that $\us{rect}(\tau')=U_\al$, thus we have the following decomposition
\begin{equation}\label{dec}
\us{SYCT}(\ga\pa\be)=\bigcup_\al^\cdot\bigcup^\cdot_{\tau\in\us{\ti{SYCT}}(\ga\pa\be)
\atop\us{\ti{rect}}(\tau)=U_\al}[\tau]_C.
\end{equation}

In order to deal with the odd case, we also define the \textit{inversion number} of $\tau\in\us{SSYCT}(\al\pa\be)$ \textit{with respect to columns} as follows,
\[\us{inv}^c(\tau)=|\{((i,j),(i',j'))\in(\al\pa\be)^{\times2}:\tau(i,j)> \tau(i',j'),~j<j'\}|.\]
Similarly, we define $\us{inv}^c(T)$ for $T\in\us{SSYT}(\la/\mu)$. Since Mason's bijection $\rho_\be$ just rearranges entries in each column, we have $\us{inv}^c(\rho_{\be}(\tau))=\us{inv}^c(\tau)$,~$\us{st}(\rho_{\be}(\tau))=\rho_{\be}(\us{st}(\tau))$. Further if $\tau\in\us{SYCT}(\al\pa\be)$, then $\us{des}(\tau)=\us{des}(\rho_\be(\tau))$.

Note that for SSYT, inversions only appear in those strictly southwest-northeast pairs of boxes in $\la$, whose total number is denoted by $NE(\la)$. More explicitly, for any $((i,j),(i',j'))\in\la\times\la,~i>i',j<j'$ , it contributes 1 to $\us{inv}(T)$ if $T(i,j)\leq T(i',j')$ and to $\us{inv}^c(T)$ otherwise. As a result, we have
\begin{equation}\label{inv}
\us{inv}(T)+\us{inv}^c(T)=NE(\la)=\sum\limits_{\la_i>j\geq1\atop \la^T_j>i\geq1}(\la_i-j)(\la^T_j-i).
\end{equation}
From such identity and $\us{inv}(\us{st}(T))=\us{inv}(T)$, we get $\us{inv}^c(\us{st}(T))=\us{inv}^c(T)$.

\begin{example} Let $\la=431\vdash 8$ and $T=\raisebox{1.6em}{\xy 0;/r.16pc/:
(0,0)*{};(20,0)*{}**\dir{-};
(0,-5)*{};(20,-5)*{}**\dir{-};
(0,-10)*{};(15,-10)*{}**\dir{-};
(0,-15)*{};(5,-15)*{}**\dir{-};
(0,0)*{};(0,-15)*{}**\dir{-};
(5,0)*{};(5,-15)*{}**\dir{-};
(10,0)*{};(10,-10)*{}**\dir{-};
(15,0)*{};(15,-10)*{}**\dir{-};
(20,0)*{};(20,-5)*{}**\dir{-};
(2.5,-2.5)*{\stt{1}};(7.5,-2.5)*{\stt{2}};(12.5,-2.5)*{\stt{2}};(17.5,-2.5)*{\stt{5}};
(2.5,-7.5)*{\stt{2}};(7.5,-7.5)*{\stt{3}};(12.5,-7.5)*{\stt{4}};(2.5,-12.5)*{\stt{5}};
\endxy}~$. Now $NE(\la)=11$, $\us{inv}(T)=6$  and $\us{inv}^c(T)=5$.
\end{example}

Meanwhile, we still need the following lemma to track the signs.
\begin{lemma}\label{sile}
For any $\tau,\tau'\in\us{SYCT}$, if $\tau\sim_C\tau'$, then $(-1)^{\us{\s{inv}}^c(\tau)+\us{\s{inv}}^c(\us{\s{rect}}(\tau))}
=(-1)^{\us{\s{inv}}^c(\tau')+\us{\s{inv}}^c(\us{\s{rect}}(\tau'))}$.
\end{lemma}
\begin{proof}
Let $w=w_{col}(\tau)$, $w'=w_{col}(\tau')$. The $C$-equivalence gives $Q(w)=Q(w')$, thus $\us{sign}(P(w))(-1)^{\ell(w)}=\us{sign}(P(w'))(-1)^{\ell(w')}$ by the identity (\ref{sil}). On the other hand, if $\tau\in\us{SYCT}(\ga\pa\be)$, then $\ell(w)-\us{inv}^c(\tau)=\sum\limits_i{\tilde{\ga}^T_i
-\tilde{\be}^T_i\choose2}$ by definition. As $\us{sh}(\tau)=\us{sh}(\tau')$, it means that  $\us{sign}(P(w))(-1)^{\us{\s{inv}}^c(\tau)}=\us{sign}(P(w'))(-1)^{\us{\s{inv}}^c(\tau')}$,
which is equivalent to $(-1)^{\us{\s{inv}}^c(\tau)+\us{\s{inv}}^c(\us{\s{rect}}(\tau))}
=(-1)^{\us{\s{inv}}^c(\tau')+\us{\s{inv}}^c(\us{\s{rect}}(\tau'))}$, for $P(w)$ and $P(w')$ also have the same shape.
\end{proof}

\subsection{Odd quasisymmetric Schur functions}
Now we're in the position to define the odd quasisymmetric Schur functions as follows. Rewrite $\ms{QSym}_{-1}$ as $\ms{OQSym}$ and the odd Schur function as
\[\begin{split}
s_\la&=(-1)^{\la^T\choose2}\sum\limits_{T\in\us{\s{SYT}}(\la)}
(-1)^{\us{\s{inv}}(T)}F_{c(T)}=(-1)^{{\la^T\choose2}}\sum\limits_{\tilde{\al}=\la}
\lb\sum\limits_{\tau\in\us{\s{SYCT}}(\al)}
(-1)^{\us{\s{inv}}(\rho(\tau))}F_{c(\tau)}\rb\mbox{ or }\\
s_\la&=(-1)^{{\la^T\choose2}+NE(\la)}
\sum\limits_{T\in\us{\s{SYT}}(\la)}
(-1)^{\us{\s{inv}}^c(T)}F_{c(T)}=(-1)^{{\la^T\choose2}+NE(\la)}\sum\limits_{\tilde{\al}=\la}
\lb\sum\limits_{\tau\in\us{\s{SYCT}}(\al)}
(-1)^{\us{\s{inv}}^c(\tau)}F_{c(\tau)}\rb,
\end{split}
\]
where the third equality comes from (\ref{inv}), the second and the fourth ones are due to Mason's bijection $\rho$. It inspires us to give
\begin{definition}
Given $\al\in C$, the \textit{odd quasisymmetric Schur function} $\mathscr{S}_\al\in\ms{OQSym}$ is defined by
\begin{equation}\label{oqs}
 \mathscr{S}_\al=(-1)^{NE(\tilde{\al})}\sum\limits_{\tau\in\us{\s{SYCT}}(\al)}
(-1)^{\us{\s{inv}}^c(\tau)}F_{c(\tau)},
\end{equation}
which means $s_\la=(-1)^{\la^T\choose2}\sum\limits_{\tilde{\al}=\la}\mathscr{S}_\al$.
\end{definition}

Similar to \cite[Prop. 5.5]{HLMW1}, we see that the coefficient matrix $(M_{\al,\be})_{\al,\be\in C}$ of $F_\be$ in $\mathscr{S}_\al$ is triangular and its diagonal consists of $(-1)^{\us{\s{inv}}(\rho(U_\al))},~\al\in C$, thus $\{\mathscr{S}_\al:\al\vDash n\}$ forms a $\mathbb{Z}$-basis of $\ms{OQSym}_n$.

In order to convince the readers of Definition \ref{oqs}, we will figure out the odd version of the Pieri rules for $\{\mathscr{S}_\al\}_{\al\in C}$ and the Littlewood-Richardson rules for its dual $\{\mathscr{S}^*_\al\}_{\al\in C}$.
We refer the dual $\mathscr{S}_\al^*\in\ms{NSym}$ as the \textit{Young noncommutative symmetric function}. First note that the image of $\mathscr{S}_\al^*$ under the forgetful map $\phi$ is $s_{\tilde{\al}}$. Indeed, by the orthogonality (\ref{ort}),
\[\phi(\mathscr{S}_\al^*)=\sum\limits_\mu(s_\mu,\phi(\mathscr{S}_\al^*))
(-1)^{\mu^T\choose2}s_\mu=\sum\limits_\mu\lan s_\mu,\mathscr{S}_\al^*\ran
(-1)^{\mu^T\choose2}s_\mu=\sum\limits_\mu\lan \sum\limits_{\tilde{\be}=\mu}\mathscr{S}_\be,\mathscr{S}_\al^*\ran
s_\mu=s_{\tilde{\al}}.\]
Meanwhile, there exists a natural bijection (similar to \cite[Prop. 2.15]{BLW})
\[f':\us{SSYCT}(\al\pa\be)\rightarrow\{(\tau,\ga)\in\us{SYCT}(\al\pa\be)\times C_w:\ga\succeq c(\tau)\},~\tau\mapsto(\us{st}(\tau),\us{cont}(\tau)),\]
which gives us the following expansion
\begin{lemma}\label{qs}
\begin{equation}
\mathscr{S}_\al=(-1)^{NE(\tilde{\al})}\sum\limits_{\tau\in\us{\s{SSYCT}}(\al)}
(-1)^{\us{\s{inv}}^c(\tau)}x^{\us{\s{cont}}(\tau)}=(-1)^{NE(\tilde{\al})}\sum\limits_{\tau\in\us{\ti{SSYCT}}(\al)\atop \us{\ti{cont}}(\tau)\vDash |\al|}
(-1)^{\us{\s{inv}}^c(\tau)}M_{\us{\s{cont}}(\tau)}.
\end{equation}
\end{lemma}

In order to state the Pieri rule for odd quasisymmetric Schur functions, we need to introduce three operators $\us{rem},\us{row},\us{col}$ on $C$.
 Let $\al= (\al_1,\dots,\al_r)\in C$ with largest part $m$ and $s\in[m]$. If there exists $1\leq i\leq k$ such that
$s=\al_i$ and $s\neq\al_j$ for all $j<i$, then define
\[\us{rem}_s(\al)=(\al_1,\dots,\al_{i-1},s-1,\al_{i+1},\dots,\al_r),\]
otherwise define $\us{rem}_s(\al)=\emptyset$. Let $S=\{s_1<\cdots<s_j\}$ and
define
\[\us{row}_S(\al)=\us{rem}_{s_1}(\cdots(\us{rem}_{s_{j-1}}(\us{rem}_{s_j}(\al))) \cdots).\]
Similarly, let $M=\{m_1\leq\cdots\leq m_j\}$ and define
\[\us{col}_M(\al)=\us{rem}_{m_j}(\cdots(\us{rem}_{m_2}(\us{rem}_{m_1}(\al))) \cdots).\]
We remove any zeros from $\us{row}_S(\al)$ or $\us{col}_M(\al)$ to obtain a composition if necessary.

For any horizontal strip $\de$ we denote by $S(\de)$ the set of columns its skew diagram
occupies, and for any vertical strip $\ep$ we denote by $M(\ep)$ the multiset of columns its
skew diagram occupies, where multiplicities for a column are given by the number
of boxes in that column and column indices are listed in weakly increasing order. We are now ready to state the odd analogous of the Pieri rules given in \cite[Theorem 5.4.2]{LMW}.
\begin{theorem}[{\bfseries Pieri rules for odd quasisymmetric Schur functions}]\label{pie1}
Let $\al$ be a composition. Then
\begin{equation}\label{hs}
\mathscr{S}_\al\mathscr{S}_{(n)}=
\sum\limits_{\be}(-1)^{{\tilde{\al}^T\choose2}+{\tilde{\be}^T\choose2}}oc_{\tilde{\al}(n)}^{\tilde{\be}}\mathscr{S}_\be,
\end{equation}
where the sum is taken over all compositions $\be$ such that

\noindent 1. $\de=\tilde{\be}/\tilde{\al}$ is a horizontal strip of size $n$,

\noindent 2. $\us{row}_{S(\de)}(\be)=\al$.

\noindent Also,
\begin{equation}
\mathscr{S}_\al\mathscr{S}_{(1^n)}=
\sum\limits_{\be}(-1)^{{\tilde{\al}^T\choose2}+{\tilde{\be}^T\choose2}}oc_{\tilde{\al}(1^n)}^{\tilde{\be}}\mathscr{S}_\be,
\end{equation}
where the sum is taken over all compositions $\be$ such that

\noindent 1. $\ep=\tilde{\be}/\tilde{\al}$ is a vertical strip of size $n$,

\noindent 2. $\us{col}_{M(\de)}(\be)=\al$.
\end{theorem}
\begin{proof}
We only prove the horizontal strip case below, the other case is similar. First note that the proof in \cite[Theorem 6.3]{HLMW1} for the original case is based on an analogy to Schensted insertion for SSYCT which commutes with Mason's bijection $\rho$, i.e. $\rho(k\rightarrow\tau)=\rho(\tau)\leftarrow k$. And the rules were proved on the plactic monoid level indeed, i.e.
\[\lb\sum\limits_{T\in\us{\ti{SSYT}}(\tilde{\al})
\atop\us{\ti{sh}}(\rho^{-1}(T))=\al}T\rb
\lb\sum\limits_{T\in\us{\s{SSYT}}((n))}T\rb=
\sum\limits_\be\sum\limits_{T\in\us{\ti{SSYT}}(\tilde{\be})
\atop\us{\ti{sh}}(\rho^{-1}(T))=\be}T,\]
where the sum is over all $\be\in C$ satisfying the two conditions in the theorem. What we need to do is replacing it with the odd plactic monoid and taking care of the signs.

In fact, for any $T\in\us{SSYT}(\la)$ and $k_1\leq\cdots\leq k_n$, we have the identity
\[(-1)^{{\la^T\choose2}+\us{\s{inv}}(T)}x^{\us{\s{cont}}(T)}x_{k_1}\cdots x_{k_n}=
(-1)^{{\mu^T\choose2}+\us{\s{inv}}(T')}oc_{\la(n)}^\mu x^{\us{\s{cont}}(T')},\]
where $T'=(T\leftarrow k_1)\leftarrow\cdots\leftarrow k_n\in\us{SSYT}(\mu)$ and $oc_{\la(n)}^\mu$ accumulates the number of boxes strictly southwest to the new box added to the intermediate diagram step by step with increasing labels as in Remark \ref{rem}. Now by Lemma \ref{qs}, \[\mathscr{S}_\al=\sum\limits_{T\in\us{\ti{SSYT}}(\tilde{\al})
\atop\us{\ti{sh}}(\rho^{-1}(T))=\al}
(-1)^{\us{\s{inv}}(T)}x^{\us{\s{cont}}(T)},~\mathscr{S}_{(n)}=s_{(n)}=\sum\limits_{k_1\leq\cdots\leq k_n}x_{k_1}\cdots x_{k_n}.
\]
Hence,
\[\begin{split}
\mathscr{S}_\al\mathscr{S}_{(n)}&=\sum\limits_\be(-1)^{{\tilde{\al}^T\choose2}+{\tilde{\be}^T\choose2}}
oc_{\tilde{\al}(n)}^{\tilde{\be}}\sum\limits_{T\in\us{\ti{SSYT}}(\tilde{\be})
\atop\us{\ti{sh}}(\rho^{-1}(T))=\be}
(-1)^{\us{\s{inv}}(T)}x^{\us{\s{cont}}(T)}\\
&=\sum\limits_\be(-1)^{{\tilde{\al}^T\choose2}+{\tilde{\be}^T\choose2}}
oc_{\tilde{\al}(n)}^{\tilde{\be}}\mathscr{S}_\be.
\end{split}\]
\end{proof}
\begin{remark}
Since $\ms{OQSym}$ is no longer commutative, we only have a right version of the Pieri rules for $\mathscr{S}_\al$, thus a left version for $\mathcal {S}_\al$.
\end{remark}

\begin{example}
Take $\al=12$ and $n=2$, then $\tilde{\al}=21,~(-1)^{\tilde{\al}^T\choose2}=-1$. SYCT$(12)=\left\{~\raisebox{1.2em}{\xy 0;/r.16pc/:
(0,0)*{};(5,0)*{}**\dir{-};
(0,-5)*{};(10,-5)*{}**\dir{-};
(0,-10)*{};(10,-10)*{}**\dir{-};
(0,0)*{};(0,-10)*{}**\dir{-};
(5,0)*{};(5,-10)*{}**\dir{-};
(10,-5)*{};(10,-10)*{}**\dir{-};
(2.5,-2.5)*{\stt{1}};(2.5,-7.5)*{\stt{2}};(7.5,-7.5)*{\stt{3}};
\endxy}~\right\}$, thus $\mathscr{S}_{12}=-F_{12}$. We have
\[\begin{split}
\mathscr{S}_{12}\mathscr{S}_2&=-F_{12}F_2=-F_{c(213)}F_{c(12)}=-\sum\limits_w\lan 213\osh45,w\ran F_{c(w)}\\
&=-F_{14}+F_{122}-F_{23}
+F_{113}-F_{131}+F_{221}-F_{32}+F_{122}-F_{1121}-F_{212}.
\end{split}\]
\[\begin{array}{l}
\tilde{\be}=41:~(-1)^{\tilde{\be}^T\choose2}=-1,~oc_{21,2}^{41}=1\\
\be=14:~\us{SYCT}(12)=\left\{~\raisebox{1.2em}{\xy 0;/r.16pc/:
(0,0)*{};(5,0)*{}**\dir{-};
(0,-5)*{};(20,-5)*{}**\dir{-};
(0,-10)*{};(20,-10)*{}**\dir{-};
(0,0)*{};(0,-10)*{}**\dir{-};
(5,0)*{};(5,-10)*{}**\dir{-};
(10,-5)*{};(10,-10)*{}**\dir{-};
(15,-5)*{};(15,-10)*{}**\dir{-};
(20,-5)*{};(20,-10)*{}**\dir{-};
(2.5,-2.5)*{\stt{1}};(2.5,-7.5)*{\stt{2}};(7.5,-7.5)*{\stt{3}};
(12.5,-7.5)*{\stt{4}};(17.5,-7.5)*{\stt{5}};
\endxy}~\right\},~\mathscr{S}_{14}=-F_{14}\\
\\\hline\\
\tilde{\be}=32:~(-1)^{\tilde{\be}^T\choose2}=1,~oc_{21,2}^{32}=1\\
\be=23:~\us{SYCT}(23)=\left\{~\raisebox{1.2em}{\xy 0;/r.16pc/:
(0,0)*{};(10,0)*{}**\dir{-};
(0,-5)*{};(15,-5)*{}**\dir{-};
(0,-10)*{};(15,-10)*{}**\dir{-};
(0,0)*{};(0,-10)*{}**\dir{-};
(5,0)*{};(5,-10)*{}**\dir{-};
(10,0)*{};(10,-10)*{}**\dir{-};
(15,-5)*{};(15,-10)*{}**\dir{-};
(2.5,-2.5)*{\stt{1}};(7.5,-2.5)*{\stt{2}};(2.5,-7.5)*{\stt{3}};
(7.5,-7.5)*{\stt{4}};(12.5,-7.5)*{\stt{5}};
\endxy}~,~\raisebox{1.2em}{\xy 0;/r.16pc/:
(0,0)*{};(10,0)*{}**\dir{-};
(0,-5)*{};(15,-5)*{}**\dir{-};
(0,-10)*{};(15,-10)*{}**\dir{-};
(0,0)*{};(0,-10)*{}**\dir{-};
(5,0)*{};(5,-10)*{}**\dir{-};
(10,0)*{};(10,-10)*{}**\dir{-};
(15,-5)*{};(15,-10)*{}**\dir{-};
(2.5,-2.5)*{\stt{1}};(7.5,-2.5)*{\stt{4}};(2.5,-7.5)*{\stt{2}};
(7.5,-7.5)*{\stt{3}};(12.5,-7.5)*{\stt{5}};
\endxy}~,~\raisebox{1.2em}{\xy 0;/r.16pc/:
(0,0)*{};(10,0)*{}**\dir{-};
(0,-5)*{};(15,-5)*{}**\dir{-};
(0,-10)*{};(15,-10)*{}**\dir{-};
(0,0)*{};(0,-10)*{}**\dir{-};
(5,0)*{};(5,-10)*{}**\dir{-};
(10,0)*{};(10,-10)*{}**\dir{-};
(15,-5)*{};(15,-10)*{}**\dir{-};
(2.5,-2.5)*{\stt{1}};(7.5,-2.5)*{\stt{5}};(2.5,-7.5)*{\stt{2}};
(7.5,-7.5)*{\stt{3}};(12.5,-7.5)*{\stt{4}};
\endxy}~\right\},~\mathscr{S}_{23}=F_{23}-F_{122}+F_{131}\\
\be=32:~\us{SYCT}(32)=\left\{~\raisebox{1.2em}{\xy 0;/r.16pc/:
(0,0)*{};(15,0)*{}**\dir{-};
(0,-5)*{};(15,-5)*{}**\dir{-};
(0,-10)*{};(10,-10)*{}**\dir{-};
(0,0)*{};(0,-10)*{}**\dir{-};
(5,0)*{};(5,-10)*{}**\dir{-};
(10,0)*{};(10,-10)*{}**\dir{-};
(15,0)*{};(15,-5)*{}**\dir{-};
(2.5,-2.5)*{\stt{1}};(7.5,-2.5)*{\stt{2}};(12.5,-2.5)*{\stt{3}};
(2.5,-7.5)*{\stt{4}};(7.5,-7.5)*{\stt{5}};
\endxy}~\right\},~\mathscr{S}_{32}=F_{32}\\\hline\\
\tilde{\be}=311:~(-1)^{\tilde{\be}^T\choose2}=-1,~oc_{21,2}^{311}=1\\
\be=113:~\us{SYCT}(113)=\left\{~\raisebox{1.7em}{\xy 0;/r.16pc/:
(0,0)*{};(5,0)*{}**\dir{-};
(0,-5)*{};(5,-5)*{}**\dir{-};
(0,-10)*{};(15,-10)*{}**\dir{-};
(0,-15)*{};(15,-15)*{}**\dir{-};
(0,0)*{};(0,-15)*{}**\dir{-};
(5,0)*{};(5,-15)*{}**\dir{-};
(10,-10)*{};(10,-15)*{}**\dir{-};
(15,-10)*{};(15,-15)*{}**\dir{-};
(2.5,-2.5)*{\stt{1}};(2.5,-7.5)*{\stt{2}};(2.5,-12.5)*{\stt{3}};
(7.5,-12.5)*{\stt{4}};(12.5,-12.5)*{\stt{5}};
\endxy}~\right\},~\mathscr{S}_{113}=F_{113}\\\hline\\
\tilde{\be}=221:~(-1)^{\tilde{\be}^T\choose2}=1,~oc_{21,2}^{221}=-1\\
\be=122:~\us{SYCT}(122)=\left\{~\raisebox{1.5em}{\xy 0;/r.16pc/:
(0,0)*{};(5,0)*{}**\dir{-};
(0,-5)*{};(10,-5)*{}**\dir{-};
(0,-10)*{};(10,-10)*{}**\dir{-};
(0,-15)*{};(10,-15)*{}**\dir{-};
(0,0)*{};(0,-15)*{}**\dir{-};
(5,0)*{};(5,-15)*{}**\dir{-};
(10,-5)*{};(10,-15)*{}**\dir{-};
(2.5,-2.5)*{\stt{1}};(2.5,-7.5)*{\stt{2}};(7.5,-7.5)*{\stt{3}};
(2.5,-12.5)*{\stt{4}};(7.5,-12.5)*{\stt{5}};
\endxy}~,~\raisebox{1.5em}{\xy 0;/r.16pc/:
(0,0)*{};(5,0)*{}**\dir{-};
(0,-5)*{};(10,-5)*{}**\dir{-};
(0,-10)*{};(10,-10)*{}**\dir{-};
(0,-15)*{};(10,-15)*{}**\dir{-};
(0,0)*{};(0,-15)*{}**\dir{-};
(5,0)*{};(5,-15)*{}**\dir{-};
(10,-5)*{};(10,-15)*{}**\dir{-};
(2.5,-2.5)*{\stt{1}};(2.5,-7.5)*{\stt{2}};(7.5,-7.5)*{\stt{5}};
(2.5,-12.5)*{\stt{3}};(7.5,-12.5)*{\stt{4}};
\endxy}~\right\},~\mathscr{S}_{122}=F_{122}-F_{1121}\\
\be=221:~\us{SYCT}(122)=\left\{~\raisebox{1.5em}{\xy 0;/r.16pc/:
(0,0)*{};(10,0)*{}**\dir{-};
(0,-5)*{};(10,-5)*{}**\dir{-};
(0,-10)*{};(10,-10)*{}**\dir{-};
(0,-15)*{};(5,-15)*{}**\dir{-};
(0,0)*{};(0,-15)*{}**\dir{-};
(5,0)*{};(5,-15)*{}**\dir{-};
(10,0)*{};(10,-10)*{}**\dir{-};
(2.5,-2.5)*{\stt{1}};(7.5,-2.5)*{\stt{2}};(2.5,-7.5)*{\stt{3}};
(7.5,-7.5)*{\stt{4}};(2.5,-12.5)*{\stt{5}};
\endxy}~\right\},~\mathscr{S}_{221}=F_{221}\\
\be=212:~\us{SYCT}(212)=\left\{~\raisebox{1.5em}{\xy 0;/r.16pc/:
(0,0)*{};(10,0)*{}**\dir{-};
(0,-5)*{};(10,-5)*{}**\dir{-};
(0,-10)*{};(10,-10)*{}**\dir{-};
(0,-15)*{};(10,-15)*{}**\dir{-};
(0,0)*{};(0,-15)*{}**\dir{-};
(5,0)*{};(5,-15)*{}**\dir{-};
(10,0)*{};(10,-5)*{}**\dir{-};
(10,-10)*{};(10,-15)*{}**\dir{-};
(2.5,-2.5)*{\stt{1}};(7.5,-2.5)*{\stt{2}};(2.5,-7.5)*{\stt{3}};
(2.5,-12.5)*{\stt{4}};(7.5,-12.5)*{\stt{5}};
\endxy}~\right\},~\mathscr{S}_{212}=-F_{212}
\end{array}\]
Now one can see that equation (\ref{hs}) holds.
\end{example}

\begin{theorem}[{\bfseries Littlewood-Richardson rule for Young noncommutative Schur
functions}]\label{lrr}
Let $\al,\be$ be compositions. Then
\[\mathscr{S}_\al^*\mathscr{S}_\be^*=\sum\limits_\ga OC_{\al\be}^\ga\mathscr{S}_\ga^*, \]
where $OC_{\al\be}^\ga:=\sum\limits_{S\in\us{\ti{SYT}}(\tilde{\ga}/
\tilde{\al}),~\us{\ti{sh}}(\rho^{-1}_\al(S))=\ga\pa\al\atop\us{\ti{rect}}(S)=\tilde{U}_{\tilde{\be}}}
(-1)^{\us{\s{inv}}((T_{\tilde{\al}})_S)
+\us{\s{inv}}(\tilde{U}_{\tilde{\be}})},~\tilde{U}_{\tilde{\be}}=\rho(U_\be)$.
Equivalently, we also have
\begin{equation}\label{dq}
\De'(\mathscr{S}_\ga)=\sum\limits_{\al,\be}OC_{\al\be}^\ga\mathscr{S}_\al\ot \mathscr{S}_\be,
\end{equation}
\end{theorem}
\begin{proof}
We choose to prove the second equation and adopt the method in \cite[Prop. 3.1]{BLW}. Given $\tau\in\us{SYCT}_n$ and an integer $k$, $0\leq k\leq n$, we denote by  $u_k(\tau)$ the SYCT obtained by removing from $\tau$ the boxes numbered $\{k+1,\dots,n\}$, and by $l_k(\tau)$  the standardization of the SYCT comprising the boxes of $\tau$ with entries $\{k+1,\dots,n\}$ with the lower-numbered boxes being added to the base shape. For example,
\[\tau=~\raisebox{2.2em}{\xy 0;/r.18pc/:
(0,0)*{};(15,0)*{}**\dir{-};
(0,-5)*{};(15,-5)*{}**\dir{-};
(0,-10)*{};(15,-10)*{}**\dir{-};
(0,-15)*{};(20,-15)*{}**\dir{-};
(0,-20)*{};(20,-20)*{}**\dir{-};
(0,0)*{};(0,-20)*{}**\dir{-};
(5,0)*{};(5,-20)*{}**\dir{-};
(10,0)*{};(10,-20)*{}**\dir{-};
(15,0)*{};(15,-5)*{}**\dir{-};
(15,-10)*{};(15,-20)*{}**\dir{-};
(20,-15)*{};(20,-20)*{}**\dir{-};
(2.5,-2.5)*{\bullet};(7.5,-2.5)*{\bullet};(2.5,-7.5)*{\bullet};
(2.5,-12.5)*{\bullet};(12.5,-2.5)*{5};(7.5,-7.5)*{3};(7.5,-12.5)*{\bullet};
(12.5,-12.5)*{1};(2.5,-17.5)*{\bullet};(7.5,-17.5)*{2};
(12.5,-17.5)*{4};(17.5,-17.5)*{6};
\endxy}~,~l_3(\tau)=~\raisebox{2.2em}{\xy 0;/r.18pc/:
(0,0)*{};(15,0)*{}**\dir{-};
(0,-5)*{};(15,-5)*{}**\dir{-};
(0,-10)*{};(15,-10)*{}**\dir{-};
(0,-15)*{};(20,-15)*{}**\dir{-};
(0,-20)*{};(20,-20)*{}**\dir{-};
(0,0)*{};(0,-20)*{}**\dir{-};
(5,0)*{};(5,-20)*{}**\dir{-};
(10,0)*{};(10,-20)*{}**\dir{-};
(15,0)*{};(15,-5)*{}**\dir{-};
(15,-10)*{};(15,-20)*{}**\dir{-};
(20,-15)*{};(20,-20)*{}**\dir{-};
(2.5,-2.5)*{\bullet};(7.5,-2.5)*{\bullet};(2.5,-7.5)*{\bullet};
(2.5,-12.5)*{\bullet};(12.5,-2.5)*{2};(7.5,-7.5)*{\bullet};(7.5,-12.5)*{\bullet};
(12.5,-12.5)*{\bullet};(2.5,-17.5)*{\bullet};(7.5,-17.5)*{\bullet};
(12.5,-17.5)*{1};(17.5,-17.5)*{3};
\endxy}~,~u_3(\tau)=~\raisebox{2.2em}{\xy 0;/r.18pc/:
(0,0)*{};(10,0)*{}**\dir{-};
(0,-5)*{};(10,-5)*{}**\dir{-};
(0,-10)*{};(15,-10)*{}**\dir{-};
(0,-15)*{};(15,-15)*{}**\dir{-};
(0,-20)*{};(10,-20)*{}**\dir{-};
(0,0)*{};(0,-20)*{}**\dir{-};
(5,0)*{};(5,-20)*{}**\dir{-};
(10,0)*{};(10,-20)*{}**\dir{-};
(15,-10)*{};(15,-15)*{}**\dir{-};
(2.5,-2.5)*{\bullet};(7.5,-2.5)*{\bullet};(2.5,-7.5)*{\bullet};
(2.5,-12.5)*{\bullet};(7.5,-7.5)*{3};(7.5,-12.5)*{\bullet};
(12.5,-12.5)*{1};(2.5,-17.5)*{\bullet};(7.5,-17.5)*{2};
\endxy}.\]
We denote by $\tau+k$ the tableau obtained by adding $k$ to every entry of $\tau$. Let $\tau_1$ be a filling of the diagram $\be\pa\al$, $\tau_2$ a filling of the diagram $\ga\pa\be$, then we denote by $\tau_1\cup\tau_2$ the natural filling of the diagram $\ga\pa\al$. Define
\[\us{inv}^c(\be\pa\al,\ga\pa\be):=|\{((i,j),(i',j'))\in\be\pa\al\times\ga\pa\be:
j>j'\}|.\]
 Note that for any $\tau\in\us{SYCT}(\ga\pa\al),~n=|\ga\pa\al|$ and $\eta_1,\eta_2\in C$ such that $\eta_1\eta_2=c(\tau)$ or $\eta_1\vee\eta_2=c(\tau)$,  there exists a unique integer $k:0\leq k\leq n$ such that $\tau=u_k(\tau)\cup(l_{n-k}(\tau)+k),~c(u_k(\tau))=\eta_1$ and $c(l_{n-k}(\tau))=\eta_2$. Meanwhile, if $\us{sh}(u_k(\tau))=\be\pa\al,~\us{sh}(l_{n-k}(\tau))=\ga\pa\be$, then one can see that
 \[\us{inv}^c(\tau)=\us{inv}^c(u_k(\tau))+\us{inv}^c(l_{n-k}(\tau))+\us{inv}^c(\be\pa\al,\ga\pa\be).\]

Now combining with the coproduct rule (\ref{fun}), we have \[\begin{split}
\De'(\mathscr{S}_\ga)&=(-1)^{NE(\tilde{\ga})}\sum\limits_{\tau\in\us{\s{SYCT}}(\ga)}
(-1)^{\us{\s{inv}}^c(\tau)}\De(F_{c(\tau)})=(-1)^{NE(\tilde{\ga})}\sum\limits_{\tau\in\us{\s{SYCT}}(\ga)}
(-1)^{\us{\s{inv}}^c(\tau)}\sum\limits_{k=0}^nF_{c(u_k(\tau))}\ot F_{c(l_k(\tau))}\\
&=(-1)^{NE(\tilde{\ga})}
\sum\limits_{\al\lessdot\ga}(-1)^{\us{\s{inv}}^c(\al,\ga\pa\al)}
\sum\limits_{\tau_1\in\us{\ti{SYCT}}(\al)\atop\tau_2\in\us{\ti{SYCT}}(\ga\pa
\al)}(-1)^{\us{\s{inv}}^c(\tau_1)}F_{c(\tau_1)}\ot (-1)^{\us{\s{inv}}^c(\tau_2)}F_{c(\tau_2)}\\
&=(-1)^{NE(\tilde{\ga})}
\sum\limits_{\al\lessdot\ga}
\mathscr{S}_\al
\ot\lb\sum\limits_{\tau_2\in\us{\s{SYCT}}(\ga\pa
\al)}(-1)^{\us{\s{inv}}^c(\al,\ga\pa\al)+NE(\tilde{\al})+\us{\s{inv}}^c(\tau_2)}F_{c(\tau_2)}\rb
\end{split}\]

On the other hand, from (\ref{rect}) and the fact that rectification preserves descents, we have
\[\mathscr{S}_\al=(-1)^{NE(\tilde{\al})}\sum\limits_{\tau\in[\tau']_C}
(-1)^{\us{\s{inv}}^c(\us{\s{rect}}(\tau))}F_{c(\tau)}\]
for any $\tau'\in\us{SYCT}$ satisfying $\us{sh}(\us{rect}(\tau'))=\al$. Combining it with the decomposition (\ref{dec}), we get
\[\begin{split}
&\quad\sum\limits_{\tau\in\us{\s{SYCT}}(\ga\pa
\al)}(-1)^{\us{\s{inv}}^c(\al,\ga\pa\al)+NE(\tilde{\al})+\us{\s{inv}}^c(\tau)}F_{c(\tau)}\\
&=\sum\limits_\be\lb\sum\limits_{\tau'\in\us{\ti{SYCT}}(\ga\pa
\al)\atop\us{\ti{rect}}(\tau')=U_\be}(-1)^{\us{\s{inv}}^c(\al,\ga\pa\al)
+NE(\tilde{\al})+NE(\tilde{\be})}\rb
\lb\sum\limits_{\tau\in[\tau']_C}(-1)^{\us{\s{inv}}^c(\tau)}F_{c(\tau)}\rb\\
&=\sum\limits_\be\lb\sum\limits_{\tau'\in\us{\ti{SYCT}}(\ga\pa
\al)\atop\us{\ti{rect}}(\tau')=U_\be}(-1)^{\us{\s{inv}}^c(\al,\ga\pa\al)+
\us{\s{inv}}^c(\tau')+\us{\s{inv}}^c(U_\be)
+NE(\tilde{\al})+NE(\tilde{\be})}\rb
\lb\sum\limits_{\tau\in[\tau']_C}(-1)^{\us{\s{inv}}^c(\us{\s{rect}}(\tau))}F_{c(\tau)}\rb\\
&=\sum\limits_\be\lb\sum\limits_{\tau'\in\us{\ti{SYCT}}(\ga\pa
\al)\atop\us{\ti{rect}}(\tau')=U_\be}(-1)^{\us{\s{inv}}^c(\al,\ga\pa\al)+
\us{\s{inv}}^c(\tau')+\us{\s{inv}}^c(U_\be)
+NE(\tilde{\al})+NE(\tilde{\be})}\rb\mathscr{S}_\be,
\end{split}
\]
where the second equality is due to Lemma \ref{sile}.

Let $OC_{\al\be}^\ga:=\sum\limits_{\tau\in\us{\ti{SYCT}}(\ga\pa
\al)\atop\us{\ti{rect}}(\tau)=U_\be}(-1)^{\us{\s{inv}}^c(\al,\ga\pa\al)+\us{\s{inv}}^c(\tau)
+\us{\s{inv}}^c(U_\be)+NE(\tilde{\al})+NE(\tilde{\be})+NE(\tilde{\ga})}$, and we have \[\De'(\mathscr{S}_\ga)=\sum\limits_{\al,\be}OC_{\al\be}^\ga\mathscr{S}_\al\ot\mathscr{S}_\be.\]
Now we simplify the coefficient $OC_{\al\be}^\ga$ by rewriting it in terms of SYT via Mason's bijection. First note that if $\tau\in\us{SYCT}(\ga\pa
\al)$, then $\us{rect}(\rho_\al(\tau))=P(w_{col}(\tau))=\rho(\us{rect}(\tau))$. Hence,
\[
\begin{split}
OC_{\al\be}^\ga &=\sum\limits_{\tau\in\us{\ti{SYCT}}(\ga\pa
\al)\atop\us{\ti{rect}}(\tau)=U_\be}(-1)^{\us{\s{inv}}^c(\al,\ga\pa\al)+\us{\s{inv}}^c(\tau)
+\us{\s{inv}}^c(U_\be)+NE(\tilde{\al})+NE(\tilde{\be})+NE(\tilde{\ga})}\\
&=\sum\limits_{\us{\ti{sh}}(\rho^{-1}_\al(S))=\ga\pa\al\atop\us{\ti{rect}}(S)=\tilde{U}_{\tilde{\be}}}
(-1)^{\us{\s{inv}}^c(\tilde{\al},\tilde{\ga}/\tilde{\al})+\us{\s{inv}}^c(S)
+\us{\s{inv}}^c(\tilde{U}_{\tilde{\be}})+NE(\tilde{\al})+NE(\tilde{\be})+NE(\tilde{\ga})}\\
&=\sum\limits_{\us{\ti{sh}}(\rho^{-1}_\al(S))=\ga\pa\al\atop\us{\ti{rect}}(S)=\tilde{U}_{\tilde{\be}}}
(-1)^{\us{\s{inv}}^c((T_{\tilde{\al}})_S)
+\us{\s{inv}}^c(\tilde{U}_{\tilde{\be}})+\us{\s{inv}}^c(T_{\tilde{\al}})+NE(\tilde{\al})+NE(\tilde{\be})+NE(\tilde{\ga})}\\
&=\sum\limits_{\us{\ti{sh}}(\rho^{-1}_\al(S))=\ga\pa\al\atop\us{\ti{rect}}(S)=\tilde{U}_{\tilde{\be}}}
(-1)^{\us{\s{inv}}((T_{\tilde{\al}})_S)
+\us{\s{inv}}(\tilde{U}_{\tilde{\be}})},
\end{split}
\]
where we use the identity
$\us{inv}^c((T_{\tilde{\al}})_S)=\us{inv}^c(\tilde{\al},\tilde{\ga}/\tilde{\al})
+\us{inv}^c(T_{\tilde{\al}})+\us{inv}^c(S)$ for the third equality and $\us{inv}(T_{\tilde{\al}})=0$ for the fourth.
\end{proof}

\begin{corollary}
 Given $\al,\be\in C$, we have $oc_{\la\mu}^\nu=\sum\limits_{\tilde{\ga}=\nu}OC_{\al\be}^\ga$ for $\la=\tilde{\al},~\mu=\tilde{\be}$.
\end{corollary}
\begin{proof}
That's due to the two odd Littlewood-Richardson rules in $\ms{OSym}$ and $\ms{OQSym}$, together with the identity $\phi(\mathscr{S}_\al^*)=s_{\tilde{\al}}$. Meanwhile, one can also deduce it from the multiplication rule (\ref{pr'm}).
\end{proof}

\begin{corollary}[{\bfseries Pieri rules for Young noncommutative Schur functions}]\label{pie2}
Let $\al$ be a composition. Then
\[\mathscr{S}^*_\al\mathscr{S}^*_{(n)}=\sum\limits_\be oc_{\tilde{\al}(n)}^{\tilde{\be}}\mathscr{S}^*_\be,\]
where $\be$ runs over all compositions satisfying
$\al=\al^0\lessdot\al^1\lessdot\cdots\lessdot\al^n=\be$
whose column sequence is strictly increasing.
Also,
\[\mathscr{S}^*_\al\mathscr{S}^*_{(1^n)}=\sum\limits_\be oc_{\tilde{\al}(1^n)}^{\tilde{\be}}\mathscr{S}^*_\be,\]
where $\be$ runs over all compositions satisfying
$\al=\al^0\lessdot\al^1\lessdot\cdots\lessdot\al^n=\be$
whose column sequence is weakly decreasing.
\end{corollary}
\begin{proof}
Note that $U_{(n)}=\tilde{U}_{(n)}=~\xy 0;/r.15pc/:
(0,2.5)*{};(20,2.5)*{}**\dir{-};
(0,-2.5)*{};(20,-2.5)*{}**\dir{-};
(0,2.5)*{};(0,-2.5)*{}**\dir{-};
(5,2.5)*{};(5,-2.5)*{}**\dir{-};
(15,2.5)*{};(15,-2.5)*{}**\dir{-};
(20,2.5)*{};(20,-2.5)*{}**\dir{-};
(2.5,0)*{\stt{1}};(10,0)*{\dots};
(17.5,0)*{\stt{n}};
\endxy~$, and $|\{S\in\us{SYT}(\la/
\tilde{\al}):\us{rect}(S)=\tilde{U}_{(n)}\}|$ is 1 if $\la/\tilde{\al}$ is a horizontal strip of size $n$ and 0 otherwise by the classical Pieri rule. We just apply Mason's bijection $\rho_\al$ to get all those $\tau\in\us{SYCT}(\be\pa\al)$ corresponding to the saturated chains given in the theorem and see that $OC_{\al(n)}^{\be}=oc_{\tilde{\al}(n)}^{\tilde{\be}}$. The case for $\mathscr{S}^*_\al\mathscr{S}^*_{(1^n)}$ is similar.
\end{proof}

\section{$q$-dual graded graphs from $q$-Hopf algebras}
The notion of \textit{dual graded graphs} was first introduced by S. Fomin \cite{Fom} as a generalization of \textit{differential posets} due to R. Stanley \cite{Stan}. Hivert and Nzeutchap, and independently Lam and Shimozono constructed dual graded graphs from primitive elements in Hopf algebras. Here we apply the $q$-version defined by Lam \cite{Lam} to the $q$-Hopf algebras we discuss.
\begin{definition}
 A \textit{$q$-graded graph} $\Ga=(V,E,h,m)$ consists of a set of vertices $V$, a set of (directed) edges $E\subset V\times V$, a \textit{height} function $h:V\rightarrow \mathbb{N}$ and an \textit{edge weight} function $m:V\times V\rightarrow \mathbb{N}[q]$ such that $h(v)=h(u)+1$ if $(u,v)\in E$. Meanwhile, $(u,v)\in E$ if and only if $m(u,v)\neq 0$ and it always assumes that there's a single vertex $\emptyset$ of height 0.
\end{definition}
Let $\mathscr{A}=\mathbb{Z}[q]$ and $\mathscr{A}V$ be the free $\mathscr{A}$-module generated by the vertex set $V$. Given a $q$-graded graph $\Ga=(V,E,h,m)$, one can define the up and down operators $U,D:\mathscr{A}V\rightarrow\mathscr{A}V$ by
\[U_\Ga(v)=\sum\limits_{u\in V}m(v,u)u,~D_\Ga(v)=\sum\limits_{u\in V}m(u,v)u\]
and extending by linearity over $\mathscr{A}$. Assuming that $\Ga$ is locally-finite, these operators are well-defined, otherwise one should define them on the completion $\widehat{\mathscr{A}V}$. A  pair of $q$-graded graphs
$(\Ga,\Ga')$ with the same vertex set $V$ and height function $h$ is called \textit{dual} with \textit{differential coefficient} $r\in\mathscr{A}$ if
\begin{equation}\label{du}
D_{\Ga'}U_\Ga-qU_\Ga D_{\Ga'}=r\us{id}.
\end{equation}
When $q=1$, self-dual graded graphs ($\Ga=\Ga'$) with differential coefficient 1 are the usual differential posets \cite{Sta}.
Define the canonical pairing on $\mathscr{A}V$ such that $\lan u,v\ran=\de_{u,v},~u,v\in V$. For a graded graph $\Ga$, let $f_\Ga^v$ denote the \textit{weight generating function} of paths in $\Ga$ from $\emptyset$ to the vertex $v$ defined by $f_\Ga^v=\lan U_\Ga^n\emptyset,v\ran,~n=h(v)$. From \cite[Theorem 4]{Lam}, we know that there's an identity
\[\sum\limits_{v:h(v)=n}f_\Ga^vf_{\Ga'}^v=r^n(n)_q!,\]
where $(n)_q=1+q+\cdots+q^{n-1}$ and $(n)_q!=(1)_q\cdots(n)_q$ for $n\in\mathbb{N}$. Also for any path $c:v_1\rightarrow v_2\rightarrow\cdots\rightarrow v_n$, define its \textit{weight} as the accumulation of edge weights at each step: $w_\Ga(c)=\prod\limits_{i=1}^{n-1}m(v_i,v_{i+1})$.

Now we introduce the construction of $q$-dual graded graphs corresponding to a pair of graded dual $q$-Hopf algebras. Let $H_\bullet=\oplus_{n\geq0}H_n$ and $H^\bullet=\oplus_{n\geq0}H^n$ be a pair of graded dual $q$-Hopf algebras over $\mathscr{A}$ with respect to the pairing $\lan\cdot,\cdot\ran:H_\bullet\times H^\bullet\rightarrow \mathscr{A}$ and $\us{dim}(H_0)=\us{dim}(H^0)=1$ such that elements in $H_1$ and $H^1$ are all primitive. Meanwhile there exist dual sets of homogeneous free $\mathscr{A}$-module generators $\{p_\la\in H_\bullet\}_{\la\in\La}$ and $\{s_\la\in H^\bullet\}_{\la\in\La}$ such that all the structure constants lie in $\mathbb{N}[q]$.

Fixed nonzero $\al\in H_1$ and $\be\in H^1$, one can define a $q$-graded graph $\Ga(\be)=(V,E,h,m)$ where $V=\{s_\la\in H^\bullet\}_{\la\in\La}$ and the height function $h$ is defined by $h(s_\la)=\us{deg}(s_\la)$. The edge weight function $m$ is defined by
\[m(s_\la,s_\mu)=\lan p_\mu,s_\la\be\ran=\lan \De(p_\mu),s_\la\ot\be\ran,\]
which determines $E$. Similarly, one can define a $q$-graded graph $\Ga'(\al)=(V',E',h',m')$ where $V'=V,~h'=h$ and
\[m'(s_\la,s_\mu)=\lan p_\la\al,s_\mu\ran=\lan p_\la\ot\al,\De(s_\mu)\ran.\]
Note that one can also define the edge weight functions $m,m'$ by left multiplication of $\al,\be$ respectively. Now one can easily check the following result:
\begin{theorem}[{\cite{BLL}}]\label{qgra}
The graded graphs $\Ga(\be)$ and $\Ga'(\al)$ form a pair of $q$-dual graded graphs with differential coefficient $\lan\al,\be\ran$.
\end{theorem}

\begin{example}
Let's first consider the case of $\ms{OSym}$. That's the graded $q$-Hopf algebra $\ms{Sym}_q$ specializing $q$ to be $-1$. Now $\ms{OSym}$ is self-dual with respect to the bilinear form $(\cdot,\cdot)$, with the dual bases $\{s_\la\}_{\la\in\mathcal {P}}$ and $\{(-1)^{\la^T\choose2}s_\la\}_{\la\in\mathcal {P}}$. Choosing the left multiplication of the primitive elements $\al=\be=s_{(1)}$, we have
\[\begin{split}
m(s_\la,s_\mu)&=((-1)^{\mu^T\choose2}s_\mu,s_{(1)}s_\la)=(-1)^{{\la^T\choose2}+{\mu^T\choose2}}
\us{sign}(T_\la)\us{sign}(T_\mu)oc_{\la,(1)}^\mu\\
&=\begin{cases}
(-1)^{\left|\tfrac{i}{\la}\right|+|\la|+\la_i+i-1},& \mbox{if $\mu$ has exactly one more box than $\la$ in the $i$th row},\\
0,&\mbox{otherwise},
\end{cases}
\end{split}
\]
\[\begin{split}
m'(s_\la,s_\mu)&=((-1)^{\la^T\choose2}s_{(1)}s_\la ,s_\mu)=\us{sign}(T_\la)\us{sign}(T_\mu)oc_{\la,(1)}^\mu,\\
&=\begin{cases}
(-1)^{\left|\tfrac{i}{\la}\right|+|\la|+\la_i},& \mbox{if $\mu$ has exactly one more box than $\la$ in the $i$th row},\\
0,&\mbox{otherwise},
\end{cases}
\end{split}
\]
where we use the left-sided version of odd Pieri rule (\ref{rp1}),  (\ref{rp2}) and the identities ${\mu^T\choose2}-{\la^T\choose2}=i-1,~\us{sign}(T_\la)\us{sign}(T_\mu)=(-1)^{|\la|+\la_i}$ and $oc_{\la,(1)}^\mu=(-1)^{\left|\tfrac{i}{\la}\right|}$, when $\mu$ has exactly one more box than $\la$ in the $i$th row. One should note that $(-1)^{\left|\tfrac{i}{\la}\right|+|\la|+\la_i}=(-1)^{\la_1+\cdots+\la_{i-1}}$.


Now we have the following signed dual graded graphs with differential coefficient 1:
\[\xy 0;/r.35pc/:
(0,1.5)*{};(0,4)*{}**\dir{-};+(1,1)*{\stt{+}};
(-1,5.5)*{};(-4,9)*{}**\dir{-};+(-2.5,1.3)*{\stt{+}};
(1,5.5)*{};(4,9)*{}**\dir{-};+(2.5,1.3)*{\stt{+}};
(-6,11.5)*{};(-8,14)*{}**\dir{-};+(-2.5,1)*{\stt{+}};
(6,11.5)*{};(8,14)*{}**\dir{-};+(2.5,1)*{\stt{+}};
(-4,11.5)*{};(-2,14)*{}**\dir{-};+(2.5,1)*{\stt{+}};
(4,11.5)*{};(2,14)*{}**\dir{-};+(-2.5,1)*{\stt{-}};
(-11,16.5)*{};(-13,19)*{}**\dir{-};+(-2.5,1)*{\stt{+}};
(11,16.5)*{};(13,19)*{}**\dir{-};+(2.5,1)*{\stt{+}};
(-1,16.5)*{};(-5,19)*{}**\dir{-};+(-3.5,1)*{\stt{-}};
(1,16.5)*{};(5,19)*{}**\dir{-};+(3.5,1)*{\stt{+}};
(0,16.5)*{};(0,18.5)*{}**\dir{-};+(1,1.3)*{\stt{-}};
(-9,16.5)*{};(-8,18)*{}**\dir{-};+(1.7,1)*{\stt{+}};
(9,16.5)*{};(8,19)*{}**\dir{-};+(-1.5,1)*{\stt{+}};
(0,0)*{\stt{\emptyset}};
(0,5)*{\xy 0;/r.07pc/:
(0,0)*{};(5,0)*{}**\dir{-};
(0,-5)*{};(5,-5)*{}**\dir{-};
(0,0)*{};(0,-5)*{}**\dir{-};
(5,0)*{};(5,-5)*{}**\dir{-};
\endxy};
(-5,10)*{\xy 0;/r.07pc/:
(0,0)*{};(5,0)*{}**\dir{-};
(0,-5)*{};(5,-5)*{}**\dir{-};
(0,-10)*{};(5,-10)*{}**\dir{-};
(0,0)*{};(0,-10)*{}**\dir{-};
(5,0)*{};(5,-10)*{}**\dir{-};
\endxy};
(5,10)*{\xy 0;/r.07pc/:
(0,0)*{};(10,0)*{}**\dir{-};
(0,-5)*{};(10,-5)*{}**\dir{-};
(0,0)*{};(0,-5)*{}**\dir{-};
(5,0)*{};(5,-5)*{}**\dir{-};
(10,0)*{};(10,-5)*{}**\dir{-};
\endxy};
(-10,15)*{\xy 0;/r.07pc/:
(0,0)*{};(5,0)*{}**\dir{-};
(0,-5)*{};(5,-5)*{}**\dir{-};
(0,-10)*{};(5,-10)*{}**\dir{-};
(0,-15)*{};(5,-15)*{}**\dir{-};
(0,0)*{};(0,-15)*{}**\dir{-};
(5,0)*{};(5,-15)*{}**\dir{-};
\endxy};
(0,15)*{\xy 0;/r.07pc/:
(0,0)*{};(10,0)*{}**\dir{-};
(0,-5)*{};(10,-5)*{}**\dir{-};
(0,-10)*{};(5,-10)*{}**\dir{-};
(0,0)*{};(0,-10)*{}**\dir{-};
(5,0)*{};(5,-10)*{}**\dir{-};
(10,0)*{};(10,-5)*{}**\dir{-};
\endxy};
(10,15)*{\xy 0;/r.07pc/:
(0,0)*{};(15,0)*{}**\dir{-};
(0,-5)*{};(15,-5)*{}**\dir{-};
(0,0)*{};(0,-5)*{}**\dir{-};
(5,0)*{};(5,-5)*{}**\dir{-};
(10,0)*{};(10,-5)*{}**\dir{-};
(15,0)*{};(15,-5)*{}**\dir{-};
\endxy};
(-15,20)*{\xy 0;/r.07pc/:
(0,0)*{};(5,0)*{}**\dir{-};
(0,-5)*{};(5,-5)*{}**\dir{-};
(0,-10)*{};(5,-10)*{}**\dir{-};
(0,-15)*{};(5,-15)*{}**\dir{-};
(0,-20)*{};(5,-20)*{}**\dir{-};
(0,0)*{};(0,-20)*{}**\dir{-};
(5,0)*{};(5,-20)*{}**\dir{-};
\endxy};
(-7.5,20)*{\xy 0;/r.07pc/:
(0,0)*{};(10,0)*{}**\dir{-};
(0,-5)*{};(10,-5)*{}**\dir{-};
(0,-10)*{};(5,-10)*{}**\dir{-};
(0,-15)*{};(5,-15)*{}**\dir{-};
(0,0)*{};(0,-15)*{}**\dir{-};
(5,0)*{};(5,-15)*{}**\dir{-};
(10,0)*{};(10,-5)*{}**\dir{-};
\endxy};
(0,20)*{\xy 0;/r.07pc/:
(0,0)*{};(10,0)*{}**\dir{-};
(0,-5)*{};(10,-5)*{}**\dir{-};
(0,-10)*{};(10,-10)*{}**\dir{-};
(0,0)*{};(0,-10)*{}**\dir{-};
(5,0)*{};(5,-10)*{}**\dir{-};
(10,0)*{};(10,-10)*{}**\dir{-};
\endxy};
(7.5,20)*{\xy 0;/r.07pc/:
(0,0)*{};(15,0)*{}**\dir{-};
(0,-5)*{};(15,-5)*{}**\dir{-};
(0,-10)*{};(5,-10)*{}**\dir{-};
(0,0)*{};(0,-10)*{}**\dir{-};
(5,0)*{};(5,-10)*{}**\dir{-};
(10,0)*{};(10,-5)*{}**\dir{-};
(15,0)*{};(15,-5)*{}**\dir{-};
\endxy};
(15,20)*{\xy 0;/r.07pc/:
(0,0)*{};(20,0)*{}**\dir{-};
(0,-5)*{};(20,-5)*{}**\dir{-};
(0,0)*{};(0,-5)*{}**\dir{-};
(5,0)*{};(5,-5)*{}**\dir{-};
(10,0)*{};(10,-5)*{}**\dir{-};
(15,0)*{};(15,-5)*{}**\dir{-};
(20,0)*{};(20,-5)*{}**\dir{-};
\endxy};
(0,-3)*{\stt{\Ga(\al)}};
\endxy\quad\xy 0;/r.35pc/:
(0,1.5)*{};(0,4)*{}**\dir{-};+(1,1)*{\stt{+}};
(-1,5.5)*{};(-4,9)*{}**\dir{-};+(-2.5,1.3)*{\stt{-}};
(1,5.5)*{};(4,9)*{}**\dir{-};+(2.5,1.3)*{\stt{+}};
(-6,11.5)*{};(-8,14)*{}**\dir{-};+(-2.5,1)*{\stt{+}};
(6,11.5)*{};(8,14)*{}**\dir{-};+(2.5,1)*{\stt{+}};
(-4,11.5)*{};(-2,14)*{}**\dir{-};+(2.5,1)*{\stt{+}};
(4,11.5)*{};(2,14)*{}**\dir{-};+(-2.5,1)*{\stt{+}};
(-11,16.5)*{};(-13,19)*{}**\dir{-};+(-2.5,1)*{\stt{-}};
(11,16.5)*{};(13,19)*{}**\dir{-};+(2.5,1)*{\stt{+}};
(-1,16.5)*{};(-5,19)*{}**\dir{-};+(-3.5,1)*{\stt{-}};
(1,16.5)*{};(5,19)*{}**\dir{-};+(3.5,1)*{\stt{+}};
(0,16.5)*{};(0,18.5)*{}**\dir{-};+(1,1.3)*{\stt{+}};
(-9,16.5)*{};(-8,18)*{}**\dir{-};+(1.7,1)*{\stt{+}};
(9,16.5)*{};(8,19)*{}**\dir{-};+(-1.5,1)*{\stt{-}};
(0,0)*{\stt{\emptyset}};
(0,5)*{\xy 0;/r.07pc/:
(0,0)*{};(5,0)*{}**\dir{-};
(0,-5)*{};(5,-5)*{}**\dir{-};
(0,0)*{};(0,-5)*{}**\dir{-};
(5,0)*{};(5,-5)*{}**\dir{-};
\endxy};
(-5,10)*{\xy 0;/r.07pc/:
(0,0)*{};(5,0)*{}**\dir{-};
(0,-5)*{};(5,-5)*{}**\dir{-};
(0,-10)*{};(5,-10)*{}**\dir{-};
(0,0)*{};(0,-10)*{}**\dir{-};
(5,0)*{};(5,-10)*{}**\dir{-};
\endxy};
(5,10)*{\xy 0;/r.07pc/:
(0,0)*{};(10,0)*{}**\dir{-};
(0,-5)*{};(10,-5)*{}**\dir{-};
(0,0)*{};(0,-5)*{}**\dir{-};
(5,0)*{};(5,-5)*{}**\dir{-};
(10,0)*{};(10,-5)*{}**\dir{-};
\endxy};
(-10,15)*{\xy 0;/r.07pc/:
(0,0)*{};(5,0)*{}**\dir{-};
(0,-5)*{};(5,-5)*{}**\dir{-};
(0,-10)*{};(5,-10)*{}**\dir{-};
(0,-15)*{};(5,-15)*{}**\dir{-};
(0,0)*{};(0,-15)*{}**\dir{-};
(5,0)*{};(5,-15)*{}**\dir{-};
\endxy};
(0,15)*{\xy 0;/r.07pc/:
(0,0)*{};(10,0)*{}**\dir{-};
(0,-5)*{};(10,-5)*{}**\dir{-};
(0,-10)*{};(5,-10)*{}**\dir{-};
(0,0)*{};(0,-10)*{}**\dir{-};
(5,0)*{};(5,-10)*{}**\dir{-};
(10,0)*{};(10,-5)*{}**\dir{-};
\endxy};
(10,15)*{\xy 0;/r.07pc/:
(0,0)*{};(15,0)*{}**\dir{-};
(0,-5)*{};(15,-5)*{}**\dir{-};
(0,0)*{};(0,-5)*{}**\dir{-};
(5,0)*{};(5,-5)*{}**\dir{-};
(10,0)*{};(10,-5)*{}**\dir{-};
(15,0)*{};(15,-5)*{}**\dir{-};
\endxy};
(-15,20)*{\xy 0;/r.07pc/:
(0,0)*{};(5,0)*{}**\dir{-};
(0,-5)*{};(5,-5)*{}**\dir{-};
(0,-10)*{};(5,-10)*{}**\dir{-};
(0,-15)*{};(5,-15)*{}**\dir{-};
(0,-20)*{};(5,-20)*{}**\dir{-};
(0,0)*{};(0,-20)*{}**\dir{-};
(5,0)*{};(5,-20)*{}**\dir{-};
\endxy};
(-7.5,20)*{\xy 0;/r.07pc/:
(0,0)*{};(10,0)*{}**\dir{-};
(0,-5)*{};(10,-5)*{}**\dir{-};
(0,-10)*{};(5,-10)*{}**\dir{-};
(0,-15)*{};(5,-15)*{}**\dir{-};
(0,0)*{};(0,-15)*{}**\dir{-};
(5,0)*{};(5,-15)*{}**\dir{-};
(10,0)*{};(10,-5)*{}**\dir{-};
\endxy};
(0,20)*{\xy 0;/r.07pc/:
(0,0)*{};(10,0)*{}**\dir{-};
(0,-5)*{};(10,-5)*{}**\dir{-};
(0,-10)*{};(10,-10)*{}**\dir{-};
(0,0)*{};(0,-10)*{}**\dir{-};
(5,0)*{};(5,-10)*{}**\dir{-};
(10,0)*{};(10,-10)*{}**\dir{-};
\endxy};
(7.5,20)*{\xy 0;/r.07pc/:
(0,0)*{};(15,0)*{}**\dir{-};
(0,-5)*{};(15,-5)*{}**\dir{-};
(0,-10)*{};(5,-10)*{}**\dir{-};
(0,0)*{};(0,-10)*{}**\dir{-};
(5,0)*{};(5,-10)*{}**\dir{-};
(10,0)*{};(10,-5)*{}**\dir{-};
(15,0)*{};(15,-5)*{}**\dir{-};
\endxy};
(15,20)*{\xy 0;/r.07pc/:
(0,0)*{};(20,0)*{}**\dir{-};
(0,-5)*{};(20,-5)*{}**\dir{-};
(0,0)*{};(0,-5)*{}**\dir{-};
(5,0)*{};(5,-5)*{}**\dir{-};
(10,0)*{};(10,-5)*{}**\dir{-};
(15,0)*{};(15,-5)*{}**\dir{-};
(20,0)*{};(20,-5)*{}**\dir{-};
\endxy};
(0,-3)*{\stt{\Ga'(\be)}};
\endxy\]
where we use Young diagrams to represent the vertex set $\{s_\la\}_{\la\in\mathcal {P}}$ and signs for the edge weights.
\begin{remark}
(1) In \cite{Lam1}, Lam studied two kinds of signed differential posets ($\al$ type and $\be$ type), which have nice enumeration properties. Meanwhile, he constructed signed Young's lattices as the fundamental examples. Now taking $v:v(s_\la)=(-1)^{\la\choose2},~\la\in\mathcal {P}$ as the labeling of the vertex set as in \cite{Lam1}, the graph $\Ga'$ (resp. $\Ga$) conjugates to the $\al$ (resp. $\be$)-signed Young's lattice due to Lam.

(2) Note that any saturated chain (i.e. path) in the Young's lattice from $\emptyset$ to $\la$ can be naturally identified with some $T\in\us{SYT}(\la)$.
Now for the graphs $\Ga_(\al),\Ga'(\be)$, we can also figure out the weights of their paths. In fact,
if the path $c$ corresponds to $T\in\us{SYT}(\la)$, then we have $w_\Ga(c)=\us{sign}(T)(-1)^{\la^T\choose2},~w_{\Ga'}(c)=\us{sign}(T)$ by noting that at each step $m'(s_\la,s_\mu)$ is just the sign counting the number of boxes strictly above the new one, while one more multiple $(-1)^{i-1}$ for $m(s_\la,s_\mu)$. Hence, we have
\begin{equation}\label{ide}
 \sum\limits_{\la\vdash n}f_\Ga^\la f_{\Ga'}^\la=\sum\limits_{\la\vdash n}(-1)^{\la^T\choose2}\lb\sum\limits_{T\in\us{\s{SYT}}(\la)}\us{sign}(T)\rb^2=\de_{n,0}.
\end{equation}
\end{remark}
\end{example}

\begin{example}
From \cite[\S7.5]{BLL}, we know that the $q$-dual graded graphs corresponding to the dual pair $\ms{QSym}_q$ and $\ms{NSym}$, when choosing the dual bases $\{F_\al\}_{\al\in C}$ and $\{R_\al\}_{\al\in C}$, are the $q$-BinWord graphs and the lifted binary trees. Now for $\ms{OQSym}$ and $\ms{NSym}$, we consider the dual bases $\{\mathscr{S}_\al\}_{\al\in C}$ and $\{\mathscr{S}^*_\al\}_{\al\in C}$. Applying the Pieri rules \ref{pie1}, \ref{pie2} to the primitive elements $\al=\mathscr{S}_{(1)},~\be=\mathscr{S}^*_{(1)}$, we have
\[m_\mathcal {L}(\mathscr{S}_\ga,\mathscr{S}_\eta)=\lan\mathscr{S}_\eta,\mathscr{S}^*_\ga\mathscr{S}^*_{(1)}\ran
=oc_{\tilde{\ga}(1)}^{\tilde{\eta}}=
(-1)^{\left|\tfrac{i}{\tilde{\ga}}\right|},\]
if $\ga\lessdot\eta$ and the box in $\tilde{\eta}/\tilde{\ga}$ is in the $i$th row of $\tilde{\eta}$, and 0 otherwise.
\[m_\mathcal {P}(\mathscr{S}_\ga,\mathscr{S}_\eta)=\lan\mathscr{S}_\ga\mathscr{S}_{(1)},\mathscr{S}^*_\eta\ran
=(-1)^{{\tilde{\ga}^T\choose2}+{\tilde{\eta}^T\choose2}}oc_{\tilde{\ga}(1)}^{\tilde{\eta}}=(-1)^{\left|\tfrac{i}{\tilde{\eta}}\right|+i-1},\]
if $\exists s\in\mathbb{Z}^+$ s.t. $\us{rem}_s(\eta)=\ga$ and the box in $\tilde{\eta}/\tilde{\ga}$ is in the $i$th row of $\tilde{\eta}$, and 0 otherwise.

Now representing vertices by compositions, we have the following signed dual graded graphs with differential coefficient 1:
\[\xy 0;/r.35pc/:
(0,5.5)*{};(0,8)*{}**\dir{-};+(1,1)*{\stt{+}};
(-1,9.5)*{};(-6.5,13)*{}**\dir{-};+(-3.5,1.3)*{\stt{+}};
(1,9.5)*{};(6.5,13)*{}**\dir{-};+(3.5,1.2)*{\stt{+}};
(-8.5,15)*{};(-14,19)*{}**\dir{-};+(-3,1)*{\stt{+}};
(8.5,15)*{};(14,19.5)*{}**\dir{-};+(3,1.5)*{\stt{+}};
(-6.5,15.5)*{};(-5.5,19)*{}**\dir{-};+(1.5,1)*{\stt{-}};
(6.5,15.5)*{};(5.5,19)*{}**\dir{-};+(-1.5,1)*{\stt{+}};
(-16.5,22.5)*{};(-20,25.5)*{}**\dir{-};+(-2.5,1)*{\stt{+}};
(16.5,22.5)*{};(20,26)*{}**\dir{-};+(2.5,1)*{\stt{+}};
(15,22.2)*{};(15,26)*{}**\dir{-};+(1,1.2)*{\stt{+}};
(-15,23)*{};(-15,26)*{}**\dir{-};+(1,1.2)*{\stt{+}};
(-5,22.8)*{};(-4,26)*{}**\dir{-};+(1,0.8)*{\stt{-}};
(-6,22.8)*{};(-9,26)*{}**\dir{-};+(-0.2,1.5)*{\stt{+}};
(-4,22.2)*{};(8,26.2)*{}**\dir{-};+(2,-0.2)*{\stt{+}};
(4,22.8)*{};(3,26)*{}**\dir{-};+(0.5,1)*{\stt{+}};
(5,22.8)*{};(9,26)*{}**\dir{-};+(2.2,1)*{\stt{+}};
(6,22.5)*{};(14,26)*{}**\dir{-};+(4.2,1.2)*{\stt{-}};
(0,4)*{\stt{\emptyset}};
(0,9)*{\xy 0;/r.07pc/:
(0,0)*{};(5,0)*{}**\dir{-};
(0,-5)*{};(5,-5)*{}**\dir{-};
(0,0)*{};(0,-5)*{}**\dir{-};
(5,0)*{};(5,-5)*{}**\dir{-};
\endxy};
(-7.5,14)*{\xy 0;/r.07pc/:
(0,0)*{};(5,0)*{}**\dir{-};
(0,-5)*{};(5,-5)*{}**\dir{-};
(0,-10)*{};(5,-10)*{}**\dir{-};
(0,0)*{};(0,-10)*{}**\dir{-};
(5,0)*{};(5,-10)*{}**\dir{-};
\endxy};
(7.5,14)*{\xy 0;/r.07pc/:
(0,0)*{};(10,0)*{}**\dir{-};
(0,-5)*{};(10,-5)*{}**\dir{-};
(0,0)*{};(0,-5)*{}**\dir{-};
(5,0)*{};(5,-5)*{}**\dir{-};
(10,0)*{};(10,-5)*{}**\dir{-};
\endxy};
(-15,21)*{\xy 0;/r.07pc/:
(0,0)*{};(5,0)*{}**\dir{-};
(0,-5)*{};(5,-5)*{}**\dir{-};
(0,-10)*{};(5,-10)*{}**\dir{-};
(0,-15)*{};(5,-15)*{}**\dir{-};
(0,0)*{};(0,-15)*{}**\dir{-};
(5,0)*{};(5,-15)*{}**\dir{-};
\endxy};
(-5,21)*{\xy 0;/r.07pc/:
(0,0)*{};(5,0)*{}**\dir{-};
(0,-5)*{};(10,-5)*{}**\dir{-};
(0,-10)*{};(10,-10)*{}**\dir{-};
(0,0)*{};(0,-10)*{}**\dir{-};
(5,0)*{};(5,-10)*{}**\dir{-};
(10,-5)*{};(10,-10)*{}**\dir{-};
\endxy};
(5,21)*{\xy 0;/r.07pc/:
(0,0)*{};(10,0)*{}**\dir{-};
(0,-5)*{};(10,-5)*{}**\dir{-};
(0,-10)*{};(5,-10)*{}**\dir{-};
(0,0)*{};(0,-10)*{}**\dir{-};
(5,0)*{};(5,-10)*{}**\dir{-};
(10,0)*{};(10,-5)*{}**\dir{-};
\endxy};
(15,21)*{\xy 0;/r.07pc/:
(0,0)*{};(15,0)*{}**\dir{-};
(0,-5)*{};(15,-5)*{}**\dir{-};
(0,0)*{};(0,-5)*{}**\dir{-};
(5,0)*{};(5,-5)*{}**\dir{-};
(10,0)*{};(10,-5)*{}**\dir{-};
(15,0)*{};(15,-5)*{}**\dir{-};
\endxy};
(-21,28)*{\xy 0;/r.07pc/:
(0,0)*{};(5,0)*{}**\dir{-};
(0,-5)*{};(5,-5)*{}**\dir{-};
(0,-10)*{};(5,-10)*{}**\dir{-};
(0,-15)*{};(5,-15)*{}**\dir{-};
(0,-20)*{};(5,-20)*{}**\dir{-};
(0,0)*{};(0,-20)*{}**\dir{-};
(5,0)*{};(5,-20)*{}**\dir{-};
\endxy};
(-15,28)*{\xy 0;/r.07pc/:
(0,0)*{};(5,0)*{}**\dir{-};
(0,-5)*{};(5,-5)*{}**\dir{-};
(0,-10)*{};(10,-10)*{}**\dir{-};
(0,-15)*{};(10,-15)*{}**\dir{-};
(0,0)*{};(0,-15)*{}**\dir{-};
(5,0)*{};(5,-15)*{}**\dir{-};
(10,-10)*{};(10,-15)*{}**\dir{-};
\endxy};
(-9,28)*{\xy 0;/r.07pc/:
(0,0)*{};(5,0)*{}**\dir{-};
(0,-5)*{};(10,-5)*{}**\dir{-};
(0,-10)*{};(10,-10)*{}**\dir{-};
(0,-15)*{};(5,-15)*{}**\dir{-};
(0,0)*{};(0,-15)*{}**\dir{-};
(5,0)*{};(5,-15)*{}**\dir{-};
(10,-5)*{};(10,-10)*{}**\dir{-};
\endxy};
(3,28)*{\xy 0;/r.07pc/:
(0,0)*{};(10,0)*{}**\dir{-};
(0,-5)*{};(10,-5)*{}**\dir{-};
(0,-10)*{};(5,-10)*{}**\dir{-};
(0,-15)*{};(5,-15)*{}**\dir{-};
(0,0)*{};(0,-15)*{}**\dir{-};
(5,0)*{};(5,-15)*{}**\dir{-};
(10,0)*{};(10,-5)*{}**\dir{-};
\endxy};
(9,28)*{\xy 0;/r.07pc/:
(0,0)*{};(10,0)*{}**\dir{-};
(0,-5)*{};(10,-5)*{}**\dir{-};
(0,-10)*{};(10,-10)*{}**\dir{-};
(0,0)*{};(0,-10)*{}**\dir{-};
(5,0)*{};(5,-10)*{}**\dir{-};
(10,0)*{};(10,-10)*{}**\dir{-};
\endxy};
(-3,28)*{\xy 0;/r.07pc/:
(0,0)*{};(5,0)*{}**\dir{-};
(0,-5)*{};(15,-5)*{}**\dir{-};
(0,-10)*{};(15,-10)*{}**\dir{-};
(0,0)*{};(0,-10)*{}**\dir{-};
(5,0)*{};(5,-10)*{}**\dir{-};
(10,-5)*{};(10,-10)*{}**\dir{-};
(15,-5)*{};(15,-10)*{}**\dir{-};
\endxy};
(15,28)*{\xy 0;/r.07pc/:
(0,0)*{};(15,0)*{}**\dir{-};
(0,-5)*{};(15,-5)*{}**\dir{-};
(0,-10)*{};(5,-10)*{}**\dir{-};
(0,0)*{};(0,-10)*{}**\dir{-};
(5,0)*{};(5,-10)*{}**\dir{-};
(10,0)*{};(10,-5)*{}**\dir{-};
(15,0)*{};(15,-5)*{}**\dir{-};
\endxy};
(21,28)*{\xy 0;/r.07pc/:
(0,0)*{};(20,0)*{}**\dir{-};
(0,-5)*{};(20,-5)*{}**\dir{-};
(0,0)*{};(0,-5)*{}**\dir{-};
(5,0)*{};(5,-5)*{}**\dir{-};
(10,0)*{};(10,-5)*{}**\dir{-};
(15,0)*{};(15,-5)*{}**\dir{-};
(20,0)*{};(20,-5)*{}**\dir{-};
\endxy};
(0,1)*{\stt{\mathcal {L}}};
\endxy\quad
\xy 0;/r.35pc/:
(0,5.5)*{};(0,8)*{}**\dir{-};+(1,1)*{\stt{+}};
(-1,9.5)*{};(-6.5,13)*{}**\dir{-};+(-3.5,1.3)*{\stt{-}};
(1,9.5)*{};(6.5,13)*{}**\dir{-};+(3.5,1.2)*{\stt{+}};
(-8.5,15)*{};(-14,19)*{}**\dir{-};+(-3,1)*{\stt{+}};
(8.5,15)*{};(14,19.5)*{}**\dir{-};+(3,1.5)*{\stt{+}};
(-7,15.5)*{};(-5.5,19)*{}**\dir{-};+(-0.5,2)*{\stt{-}};
(7,15.5)*{};(5.5,19)*{}**\dir{-};+(0.5,2)*{\stt{-}};
(-6,15.5)*{};(4,19)*{}**\dir{-};+(3.5,0)*{\stt{-}};
(6,15.5)*{};(-4,19)*{}**\dir{-};+(-3.5,0)*{\stt{-}};
(-16.5,22.5)*{};(-20,25.5)*{}**\dir{-};+(-2.5,1)*{\stt{-}};
(16.5,22.2)*{};(20,26)*{}**\dir{-};+(2.5,1)*{\stt{+}};
(-15,23)*{};(-15,26)*{}**\dir{-};+(-1,1.2)*{\stt{+}};
(15,22.2)*{};(15,26)*{}**\dir{-};+(1,1.2)*{\stt{-}};
(-14,23)*{};(-10,26)*{}**\dir{-};+(0.2,1.2)*{\stt{+}};
(-13.5,22.5)*{};(2,26)*{}**\dir{-};+(2,-0.5)*{\stt{+}};
(13.5,22.2)*{};(-2,26)*{}**\dir{-};+(-2,-0.5)*{\stt{-}};
(-5,22.8)*{};(-4,26)*{}**\dir{-};+(1,0.8)*{\stt{-}};
(5,22.5)*{};(3.5,26)*{}**\dir{-};+(-1.2,0.8)*{\stt{+}};
(-6,22.5)*{};(-14,26)*{}**\dir{-};+(-1.5,-0.2)*{\stt{+}};
(6,22.5)*{};(14,26)*{}**\dir{-};+(1.5,-0.2)*{\stt{-}};
(4,22.2)*{};(-8.5,26)*{}**\dir{-};+(-2,-0.2)*{\stt{+}};
(-4,22.2)*{};(8.5,26)*{}**\dir{-};+(2,-0.2)*{\stt{-}};
(0,4)*{\stt{\emptyset}};
(0,9)*{\xy 0;/r.07pc/:
(0,0)*{};(5,0)*{}**\crv{~*=<.1pt>{.}(2,0)};
(0,-5)*{};(5,-5)*{}**\crv{~*=<.1pt>{.}(2,-5)};
(0,0)*{};(0,-5)*{}**\crv{~*=<.1pt>{.}(0,2)};
(5,0)*{};(5,-5)*{}**\crv{~*=<.1pt>{.}(5,-2)};
\endxy};
(-7.5,14)*{\xy 0;/r.07pc/:
(0,0)*{};(5,0)*{}**\crv{~*=<.1pt>{.}(2,0)};
(0,-5)*{};(5,-5)*{}**\crv{~*=<.1pt>{.}(2,-5)};
(0,-10)*{};(5,-10)*{}**\crv{~*=<.1pt>{.}(2,-10)};
(0,0)*{};(0,-10)*{}**\crv{~*=<.1pt>{.}(0,-2)};
(5,0)*{};(5,-10)*{}**\crv{~*=<.1pt>{.}(5,-2)};
\endxy};
(7.5,14)*{\xy 0;/r.07pc/:
(0,0)*{};(10,0)*{}**\crv{~*=<.1pt>{.}(2,0)};
(0,-5)*{};(10,-5)*{}**\crv{~*=<.1pt>{.}(2,-5)};
(0,0)*{};(0,-5)*{}**\crv{~*=<.1pt>{.}(0,-2)};
(5,0)*{};(5,-5)*{}**\crv{~*=<.1pt>{.}(5,-2)};
(10,0)*{};(10,-5)*{}**\crv{~*=<.1pt>{.}(10,-2)};
\endxy};
(-15,21)*{\xy 0;/r.07pc/:
(0,0)*{};(5,0)*{}**\crv{~*=<.1pt>{.}(2,0)};
(0,-5)*{};(5,-5)*{}**\crv{~*=<.1pt>{.}(2,-5)};
(0,-10)*{};(5,-10)*{}**\crv{~*=<.1pt>{.}(2,-10)};
(0,-15)*{};(5,-15)*{}**\crv{~*=<.1pt>{.}(2,-15)};
(0,0)*{};(0,-15)*{}**\crv{~*=<.1pt>{.}(0,-2)};
(5,0)*{};(5,-15)*{}**\crv{~*=<.1pt>{.}(5,-2)};
\endxy};
(-5,21)*{\xy 0;/r.07pc/:
(0,0)*{};(5,0)*{}**\crv{~*=<.1pt>{.}(2,0)};
(0,-5)*{};(10,-5)*{}**\crv{~*=<.1pt>{.}(2,-5)};
(0,-10)*{};(10,-10)*{}**\crv{~*=<.1pt>{.}(2,-10)};
(0,0)*{};(0,-10)*{}**\crv{~*=<.1pt>{.}(0,-2)};
(5,0)*{};(5,-10)*{}**\crv{~*=<.1pt>{.}(5,-2)};
(10,-5)*{};(10,-10)*{}**\crv{~*=<.1pt>{.}(10,-7)};
\endxy};
(5,21)*{\xy 0;/r.07pc/:
(0,0)*{};(10,0)*{}**\dir{-};
(0,-5)*{};(10,-5)*{}**\dir{-};
(0,-10)*{};(5,-10)*{}**\dir{-};
(0,0)*{};(0,-10)*{}**\dir{-};
(5,0)*{};(5,-10)*{}**\dir{-};
(10,0)*{};(10,-5)*{}**\dir{-};
\endxy};
(15,21)*{\xy 0;/r.07pc/:
(0,0)*{};(15,0)*{}**\crv{~*=<.1pt>{.}(2,0)};
(0,-5)*{};(15,-5)*{}**\crv{~*=<.1pt>{.}(2,-5)};
(0,0)*{};(0,-5)*{}**\crv{~*=<.1pt>{.}(0,-2)};
(5,0)*{};(5,-5)*{}**\crv{~*=<.1pt>{.}(5,-2)};
(10,0)*{};(10,-5)*{}**\crv{~*=<.1pt>{.}(10,-2)};
(15,0)*{};(15,-5)*{}**\crv{~*=<.1pt>{.}(15,-2)};
\endxy};
(-21,28)*{\xy 0;/r.07pc/:
(0,0)*{};(5,0)*{}**\crv{~*=<.1pt>{.}(2,0)};
(0,-5)*{};(5,-5)*{}**\crv{~*=<.1pt>{.}(2,-5)};
(0,-10)*{};(5,-10)*{}**\crv{~*=<.1pt>{.}(2,-10)};
(0,-15)*{};(5,-15)*{}**\crv{~*=<.1pt>{.}(2,-15)};
(0,-20)*{};(5,-20)*{}**\crv{~*=<.1pt>{.}(2,-20)};
(0,0)*{};(0,-20)*{}**\crv{~*=<.1pt>{.}(0,-2)};
(5,0)*{};(5,-20)*{}**\crv{~*=<.1pt>{.}(5,-2)};
\endxy};
(-15,28)*{\xy 0;/r.07pc/:
(0,0)*{};(5,0)*{}**\crv{~*=<.1pt>{.}(2,0)};
(0,-5)*{};(5,-5)*{}**\crv{~*=<.1pt>{.}(2,-5)};
(0,-10)*{};(10,-10)*{}**\crv{~*=<.1pt>{.}(2,-10)};
(0,-15)*{};(10,-15)*{}**\crv{~*=<.1pt>{.}(2,-15)};
(0,0)*{};(0,-15)*{}**\crv{~*=<.1pt>{.}(0,-2)};
(5,0)*{};(5,-15)*{}**\crv{~*=<.1pt>{.}(5,-2)};
(10,-10)*{};(10,-15)*{}**\crv{~*=<.1pt>{.}(10,-12)};
\endxy};
(-9,28)*{\xy 0;/r.07pc/:
(0,0)*{};(5,0)*{}**\dir{-};
(0,-5)*{};(10,-5)*{}**\dir{-};
(0,-10)*{};(10,-10)*{}**\dir{-};
(0,-15)*{};(5,-15)*{}**\dir{-};
(0,0)*{};(0,-15)*{}**\dir{-};
(5,0)*{};(5,-15)*{}**\dir{-};
(10,-5)*{};(10,-10)*{}**\dir{-};
\endxy};
(3,28)*{\xy 0;/r.07pc/:
(0,0)*{};(10,0)*{}**\dir{-};
(0,-5)*{};(10,-5)*{}**\dir{-};
(0,-10)*{};(5,-10)*{}**\dir{-};
(0,-15)*{};(5,-15)*{}**\dir{-};
(0,0)*{};(0,-15)*{}**\dir{-};
(5,0)*{};(5,-15)*{}**\dir{-};
(10,0)*{};(10,-5)*{}**\dir{-};
\endxy};
(9,28)*{\xy 0;/r.07pc/:
(0,0)*{};(10,0)*{}**\crv{~*=<.1pt>{.}(5,0)};
(0,-5)*{};(10,-5)*{}**\crv{~*=<.1pt>{.}(5,-5)};
(0,-10)*{};(10,-10)*{}**\crv{~*=<.1pt>{.}(5,-10)};
(0,0)*{};(0,-10)*{}**\crv{~*=<.1pt>{.}(0,-5)};
(5,0)*{};(5,-10)*{}**\crv{~*=<.1pt>{.}(5,-5)};
(10,0)*{};(10,-10)*{}**\crv{~*=<.1pt>{.}(10,-5)};
\endxy};
(-3,28)*{\xy 0;/r.07pc/:
(0,0)*{};(5,0)*{}**\crv{~*=<.1pt>{.}(2,0)};
(0,-5)*{};(15,-5)*{}**\crv{~*=<.1pt>{.}(2,-5)};
(0,-10)*{};(15,-10)*{}**\crv{~*=<.1pt>{.}(2,-10)};
(0,0)*{};(0,-10)*{}**\crv{~*=<.1pt>{.}(0,-2)};
(5,0)*{};(5,-10)*{}**\crv{~*=<.1pt>{.}(5,-2)};
(10,-5)*{};(10,-10)*{}**\crv{~*=<.1pt>{.}(10,-8)};
(15,-5)*{};(15,-10)*{}**\crv{~*=<.1pt>{.}(15,-8)};
\endxy};
(15,28)*{\xy 0;/r.07pc/:
(0,0)*{};(15,0)*{}**\dir{-};
(0,-5)*{};(15,-5)*{}**\dir{-};
(0,-10)*{};(5,-10)*{}**\dir{-};
(0,0)*{};(0,-10)*{}**\dir{-};
(5,0)*{};(5,-10)*{}**\dir{-};
(10,0)*{};(10,-5)*{}**\dir{-};
(15,0)*{};(15,-5)*{}**\dir{-};
\endxy};
(21,28)*{\xy 0;/r.07pc/:
(0,0)*{};(20,0)*{}**\crv{~*=<.1pt>{.}(10,0)};
(0,-5)*{};(20,-5)*{}**\crv{~*=<.1pt>{.}(2,-5)};
(0,0)*{};(0,-5)*{}**\crv{~*=<.1pt>{.}(0,-2)};
(5,0)*{};(5,-5)*{}**\crv{~*=<.1pt>{.}(5,-2)};
(10,0)*{};(10,-5)*{}*\crv{~*=<.1pt>{.}(10,-2)};
(15,0)*{};(15,-5)*{}**\crv{~*=<.1pt>{.}(15,-2)};
(20,0)*{};(20,-5)*{}**\crv{~*=<.1pt>{.}(20,-2)};
\endxy};
(0,1)*{\stt{\mathcal {P}}};
\endxy
\]
where we draw Young's lattice in bold as a subposet of $\mathcal {P}$ with Young diagrams in French notation though. Meanwhile, one should observe that

\noindent$\bullet$ As pointed out in Prop. \ref{ycp}, any path $c:\al=\al^0\lessdot\cdots\lessdot\al^n=\ga$ in $\mathcal {L}$ corresponds to a SYCT of skew shape $\ga\pa\al$, namely $\tau$, then $w_\mathcal {L}(c)=(-1)^{\us{\s{inv}}(\rho_\al(\tau))}$. Hence, $f_\mathcal {L}^\al=\sum\limits_{\tau\in\us{\s{SYCT}}(\al)}(-1)^{\us{\s{inv}}(\rho_\al(\tau))}$.

\noindent$\bullet$ The number of paths in $\mathcal {P}$ from $\emptyset$ to $\al$ is just $|\us{SYT}(\tilde{\al})|$. Moreover, we have $f_\mathcal {P}^\al=f_\mathcal {P}^{\tilde{\al}}=\sum\limits_{T\in\us{\s{SYT}}(\tilde{\al})}(-1)^{{\tilde{\al}^T\choose2}+\us{\s{inv}}(T)}$, as Young's lattice embeds in $\mathcal {P}$.

Therefore,
\[\begin{split}
\sum\limits_{\al\vDash n}f_\mathcal {L}^\al f_\mathcal {P}^\al&=
\sum\limits_{\la\vdash n}f_\mathcal {P}^\la\lb\sum\limits_{\tilde{\al}=\la}f_\mathcal {L}^\al\rb=
\sum\limits_{\la\vdash n}(-1)^{\la^T\choose2}\lb\sum\limits_{T\in\us{\s{SYT}}(\la)}(-1)^{\us{\s{inv}}(T)}\rb^2\\
&=\sum\limits_{\la\vdash n}(-1)^{\la^T\choose2}\lb\sum\limits_{T\in\us{\s{SYT}}(\la)}\us{sign}(T)\rb^2=\de_{n,0},
\end{split}\]
back to the  identity (\ref{ide}).
\end{example}

\begin{example}
Next we consider the $q$-dual graded graphs corresponding to the dual pair $\ms{MR}_q$ and $\ms{MR}'_q$. Choosing the canonical self-dual basis $\{w\}_{w\in\mathfrak{S}}$ and the right multiplication of the primitive elements $\al=\be=1\in\mathfrak{S}_1$, we have
\[\begin{array}{l}
m(u,w)=\lan w,u*1\ran=\de_{u,\us{\s{st}}(w_1\cdots w_{n-1})},\\
m'(u,w)=\lan u*'_q 1,w\ran=q^{n-w^{-1}(n)}\de_{u,w|_{\{1,\dots,n-1\}}},
\end{array}
\]
where $n=|w|$ and $n-w^{-1}(n)$ is the number of letters behind $n$ in $w$.
Hence, we realize the following $q$-dual graded graphs ($r=1$) as the $q$-analogue of the permutation trees, denoted $(\us{Perm},\us{Perm}')$,
\[\xy 0;/r.35pc/:
(0,0)*{\emptyset};(0,6)*{1};(-8,12)*{12};(8,12)*{21};
(-15,18)*{123};(-9,18)*{132};(-3,18)*{213};
(3,18)*{231};(9,18)*{312};(15,18)*{321};
(0,1.5)*{};(0,4.5)*{}**\dir{-};
(-1,7.5)*{};(-7,10.5)*{}**\dir{-};
(1,7.5)*{};(7,10.5)*{}**\dir{-};
(-9.5,13)*{};(-14,16.5)*{}**\dir{-};
(9.5,13)*{};(14,16.5)*{}**\dir{-};
(-8,13)*{};(-8,16.5)*{}**\dir{-};
(8,13)*{};(8,16.5)*{}**\dir{-};
(6.5,13)*{};(-2.5,16.5)*{}**\dir{-};
(-6.5,13)*{};(2.5,16.5)*{}**\dir{-};
(0,-3)*{\us{Perm}};
\endxy\quad\xy 0;/r.35pc/:
(0,0)*{\emptyset};(0,6)*{1};(-8,12)*{12};(8,12)*{21};
(-15,18)*{123};(-9,18)*{132};(-3,18)*{213};
(3,18)*{231};(9,18)*{312};(15,18)*{321};
(0,1.5)*{};(0,4.5)*{}**\dir{-};+(1,1.5)*{\stt{1}};
(-1,7.5)*{};(-7,10.5)*{}**\dir{-};+(-4,0.5)*{\stt{1}};
(1,7.5)*{};(7,10.5)*{}**\dir{-};+(4,0.5)*{\stt{q}};
(-9.5,13)*{};(-14,16.5)*{}**\dir{-};+(-1,2)*{\stt{1}};
(9.5,13)*{};(14,16.5)*{}**\dir{-};+(3,1)*{\stt{q^2}};
(-8,13)*{};(-8,16.5)*{}**\dir{-};+(1,2)*{\stt{q}};
(8,13)*{};(2.5,16.5)*{}**\dir{-};+(-1,2)*{\stt{q}};
(6.5,13)*{};(-2.5,16.5)*{}**\dir{-};+(-3,0)*{\stt{1}};
(-6.5,13)*{};(8,16.5)*{}**\dir{-};+(5,0)*{\stt{q^2}};
(0,-3)*{\us{Perm}'};
\endxy\]
where the edge weights of $\us{Perm}$ are all 1. Note that for any $w\in\mathfrak{S}$, there exists exactly one path from $\emptyset$ to $w$ in both graphs, and $f^w_{\us{\s{Perm}}}=1, f^w_{\us{\s{Perm}}'}=q^{\ell(w)}$, thus $\sum\limits_{w\in\mathfrak{S}_n}f^w_{\us{\s{Perm}}}f^w_{\us{\s{Perm}}'}
=\sum\limits_{w\in\mathfrak{S}_n}q^{\ell(w)}=(n)_q!$.
\end{example}

\begin{example}
Finally we consider the $q$-dual graded graphs corresponding to the dual pair $\ms{PR}_q$ and $\ms{PR}'_q$. First choose the canonical dual base $\{c_q(T)\}_{T\in\us{\s{SYT}}},~\{c^*_q(T)\}_{T\in\us{\s{SYT}}}$ and the right multiplication of the primitive elements $\al=c_q(\mb{1}),
\be=c^*_q(\mb{1})$, where $\mb{1}:=~\xy 0;/r.15pc/:
(0,0)*{\stt{1}};
(-2.5,2.5)*{};(-2.5,-2.5)*{}**\dir{-};
(2.5,2.5)*{};(2.5,-2.5)*{}**\dir{-};
(-2.5,-2.5)*{};(2.5,-2.5)*{}**\dir{-};
(2.5,2.5)*{};(-2.5,2.5)*{}**\dir{-};
\endxy~\in\us{SYT}_1$. Meanwhile, simply using $\us{SYT}$ to represent the vertex set, we have
\[\begin{array}{l}
m(S,T)=\lan c_q(T),c^*_q(S)*c^*_q(\mb{1})\ran=\sum\limits_{k=1}^nq^{n-k}\de_{T,S_k\leftarrow k},\\
m'(S,T)=\lan c_q(S)*'_q c_q(\mb{1}),c^*_q(T)\ran=\de_{S,T\backslash n},
\end{array}
\]
where $T\in\us{SYT}_n$, $T\backslash n$ is the tableau obtained from $T$ by removing the box occupied by $n$ and $S_k$ is the tableau obtained from $S$ by shifting all its entries greater than or equal to $k$ by 1, followed by the Schensted insertion $S_k\leftarrow k$. Hence, we realize the following $q$-dual graded graphs ($r=1$) as the $q$-analogue of the $\us{SYT}$ trees and Schensted graphs \cite{Fom}, denoted $(\us{Tab},\us{Tab}')$,
\[\xy 0;/r.35pc/:
(0,0)*{\stt{\emptyset}};(0,6)*{\stt{1}};(-5,12)*{\xy 0;/r.1pc/:(-2,0)*{\stt{1}};(2,0)*{\stt{2}};
\endxy};
(5,12)*{\xy 0;/r.1pc/:(0,-3)*{\stt{2}};(0,3)*{\stt{1}};
\endxy};
(-12,18)*{\xy 0;/r.1pc/:(-4,0)*{\stt{1}};(0,0)*{\stt{2}};(4,0)*{\stt{3}};
\endxy};
(-4,18)*{\xy 0;/r.1pc/:(-2,3)*{\stt{1}};(2,3)*{\stt{2}};(-2,-3)*{\stt{3}};
\endxy};
(4,18)*{\xy 0;/r.1pc/:(-2,3)*{\stt{1}};(2,3)*{\stt{3}};(-2,-3)*{\stt{2}};
\endxy};
(12,18)*{\xy 0;/r.1pc/:(0,6)*{\stt{1}};(0,0)*{\stt{2}};(0,-6)*{\stt{3}};
\endxy};
(0,1.5)*{};(0,4.5)*{}**\dir{-};+(1,1.5)*{\stt{1}};
(-1,7.5)*{};(-4,10.5)*{}**\dir{-};+(-2.5,1)*{\stt{1}};
(1,7.5)*{};(4,10.5)*{}**\dir{-};+(2.5,1)*{\stt{q}};
(-6,13.5)*{};(-11,16.5)*{}**\dir{-};+(-3.5,1)*{\stt{1}};
(-5,13.5)*{};(-4,16.5)*{}**\dir{-};
(-4,13.5)*{};(3,16)*{}**\dir{-};+(2,-0.5)*{\stt{q^2}};
(4,13.5)*{};(-3,16)*{}**\dir{-};+(-2,-0.5)*{\stt{q}};
(6,13.5)*{};(11,16.5)*{}**\dir{-};+(3.5,1)*{\stt{q^2}};
(-5,13.5)*{};(-4,16.5)*{}**\dir{-};
(-5,13.5)*{};(-4,16.5)*{}**\dir{-};+(-0.5,2)*{\stt{q}};
(5,14)*{};(4,16.5)*{}**\dir{-};+(0,2)*{\stt{1}};
(0,-3)*{\us{Tab}};
\endxy\quad
\xy 0;/r.35pc/:
(0,0)*{\stt{\emptyset}};(0,6)*{\stt{1}};(-5,12)*{\xy 0;/r.1pc/:(-2,0)*{\stt{1}};(2,0)*{\stt{2}};
\endxy};
(5,12)*{\xy 0;/r.1pc/:(0,-3)*{\stt{2}};(0,3)*{\stt{1}};
\endxy};
(-12,18)*{\xy 0;/r.1pc/:(-4,0)*{\stt{1}};(0,0)*{\stt{2}};(4,0)*{\stt{3}};
\endxy};
(-4,18)*{\xy 0;/r.1pc/:(-2,3)*{\stt{1}};(2,3)*{\stt{2}};(-2,-3)*{\stt{3}};
\endxy};
(4,18)*{\xy 0;/r.1pc/:(-2,3)*{\stt{1}};(2,3)*{\stt{3}};(-2,-3)*{\stt{2}};
\endxy};
(12,18)*{\xy 0;/r.1pc/:(0,6)*{\stt{1}};(0,0)*{\stt{2}};(0,-6)*{\stt{3}};
\endxy};
(0,1.5)*{};(0,4.5)*{}**\dir{-};
(-1,7.5)*{};(-4,10.5)*{}**\dir{-};
(1,7.5)*{};(4,10.5)*{}**\dir{-};
(-6.5,13.5)*{};(-11,16.5)*{}**\dir{-};
(-5,13.5)*{};(-4,16.5)*{}**\dir{-};
(6.5,13.5)*{};(11,16.5)*{}**\dir{-};
(5,14)*{};(4,16.5)*{}**\dir{-};
(0,-3)*{\us{Tab}'};
\endxy\]
where the edge weights of $\us{Tab}'$ are all 1.

Note that for any $T\in\us{SYT}_n$, every path $c$ from $\emptyset$ to $T$ in $\us{Tab}$ corresponds to a $w\in\mathfrak{S}_n$ such that $m_{\us{\s{Tab}}}(c)=q^{\ell(w)}$, as there exists a Schensted insertion at each step. Hence, $f^T_{\us{\s{Tab}}}=\sum\limits_{P(w)=T}q^{\ell(w)}, f^T_{\us{\s{Tab}}'}=1$, and $\sum\limits_{T\in\us{\s{SYT}}_n}f^T_{\us{\s{Tab}}}f^T_{\us{\s{Tab}}'}
=\sum\limits_{w\in\mathfrak{S}_n}q^{\ell(w)}=(n)_q!$.
\end{example}


\bigskip
\bibliographystyle{amsalpha}

\clearpage
\end{document}